\documentclass[a4paper,reqno]{amsart}

\usepackage{tikz,pgf,amscd,enumerate,multicol}
\usepackage{amsmath,amssymb,tikz-cd,tkz-graph}%

\usepackage[margin=1.2cm]{geometry}
\usepackage{longtable}
\usetikzlibrary{decorations.markings,arrows,cd}

\usepackage{tabularx}

\usepackage{color, colortbl}
	
\definecolor{Gray}{gray}{0.9}
\definecolor{LGray}{gray}{0.91}

\usepackage[pagebackref,colorlinks,linkcolor=red,citecolor=blue,urlcolor=blue,hypertexnames=false]{hyperref}

\tikzset{midarrow/.style = {postaction=decorate, decoration={markings,mark=at position .5 with \arrow{stealth}}}} 

\pgfdeclarelayer{background}
\pgfsetlayers{background,main}

\usepackage[all,2cell]{xy}
\usepackage{float}

\def\pv#1{%
  \pgfkeysvalueof{/tikz/quiver/#1}%
}

\tikzset{
  curve/.style={
    quiver/.cd,
    #1,
    /tikz/.cd,
    to path={
      (\tikztostart) .. controls ($(\tikztostart)!\pv{pos}!(\tikztotarget)!\pv{height}!270:(\tikztotarget)$)
    and ($(\tikztostart)!1-\pv{pos}!(\tikztotarget)!\pv{height}!270:(\tikztotarget)$)
    .. (\tikztotarget)\tikztonodes}
  },
  quiver/.cd,
  pos/.initial=0.35,
  height/.initial=0
}

\newsavebox\E
\sbox\E{\begin{tikzpicture}[scale=2]
\fill (0,0)  circle[radius=.6pt];
\fill (-.35,-.61)  circle[radius=.6pt];
\fill (.35,-.61)  circle[radius=.6pt];
\draw (-.70, 0) node{$E^1_8$:};
\draw[->, shorten >=5pt] (0,0) to (.35,-.61);
\draw[->, shorten >=5pt] (.35,-.61) to [in=0,out=60] (0,0);
\draw[->, shorten >=5pt] (.35,-.61) to (-.35,-.61);
\draw[->, shorten >=5pt] (-.35,-.61) to [in=-120,out=-60] (.35,-.61) ;
\draw[->, shorten >=5pt] (-.35,-.61) to (0,0);
\draw[->, shorten >=5pt] (0,0) to[in=130,out=50, loop] (0,0);
\draw[->, shorten >=5pt] (-.35,-.61) to[in=250,out=170, loop] (-.35,-.61);
\end{tikzpicture}} 

\newsavebox\F
\sbox\F{\begin{tikzpicture}[scale=2]
\fill (0,0)  circle[radius=.6pt];
\fill (-.35,-.61)  circle[radius=.6pt];
\fill (.35,-.61)  circle[radius=.6pt];
\draw (-.70, 0) node{$E^1_{12}:$};
\draw[->, shorten >=5pt] (0,0) to (.35,-.61);
\draw[->, shorten >=5pt] (.35,-.61) to [in=0,out=60] (0,0);
\draw[->, shorten >=5pt] (-.35,-.61) to (0,0);
\draw[->, shorten >=5pt] (0,0) to [in=120,out=180
] (-.35,-.61) ;
\draw[->, shorten >=5pt] (0,0) to[in=130,out=50, loop] (0,0);
\draw[->, shorten >=5pt] (-.35,-.61) to[in=250,out=170, loop] (-.35,-.61);
\end{tikzpicture}}

\newsavebox\Cone
\sbox\Cone{\begin{tikzpicture}[scale=2.5]
\fill (0,0)  circle[radius=.6pt];
\fill (.5,0)  circle[radius=.6pt];
\fill (0,-.5)  circle[radius=.6pt];
\fill (.5,-.5)  circle[radius=.6pt];
\draw[->, shorten >=5pt] (0,0) to (.5,0);
\draw[->, shorten >=5pt]  (0,-.5) to (0,0);
\draw[->, shorten >=5pt]  (.5,-.5) to (0,-.5);
\draw[->, shorten >=5pt] (.5,0) to (.5,-.5);
\draw[->, shorten >=5pt] (0,-.5) to[in=200, out=70] (.5,0);
\draw[->, shorten >=5pt] (.5,0) to[in=20, out=250] (0,-.5);
\draw[->, shorten >=5pt] (.5,0) to[in=45, out=135] (0,0);
\draw[->, shorten >=5pt] (0,0) to[in=175,out=95, loop] (0,0);
\draw[->, shorten >=5pt] (.5,-.5) to[in=275, out=355, loop] (0,-.5);
\end{tikzpicture}}

\newsavebox\Ctwo
\sbox\Ctwo{\begin{tikzpicture}[scale=2.5]
\fill (0,0)  circle[radius=.6pt];
\fill (.5,0)  circle[radius=.6pt];
\fill (0,-.5)  circle[radius=.6pt];
\fill (.5,-.5)  circle[radius=.6pt];
\draw[->, shorten >=5pt] (0,0) to (.5,0);
\draw[->, shorten >=5pt]  (0,-.5) to (0,0);
\draw[->, shorten >=5pt]  (.5,-.5) to (0,-.5);
\draw[->, shorten >=5pt] (.5,0) to (.5,-.5);
\draw[->, shorten >=5pt] (0,-.5) to[in=190, out=80] (.5,0);
\draw[->, shorten >=5pt] (.5,0) to[in=10, out=260] (0,-.5);
\draw[->, shorten >=5pt] (.5,0) to (0,-.5);
\draw[->, shorten >=5pt] (0,0) to[in=175,out=95, loop] (0,0);
\draw[->, shorten >=5pt] (.5,-.5) to[in=275, out=355, loop] (0,-.5);
\end{tikzpicture}}

\newsavebox\Cthree
\sbox\Cthree{\begin{tikzpicture}[scale=2]
\fill (0,0)  circle[radius=.6pt];
\fill (-.35,-.61)  circle[radius=.6pt];
\fill (.35,-.61)  circle[radius=.6pt];
\draw[->, shorten >=5pt] (0,0) to (.35,-.61);
\draw[->, shorten >=5pt] (.35,-.61) to (-.35,-.61);
\draw[->, shorten >=5pt] (-.35,-.61) to [in=-130,out=-50] (.35,-.61) ;
\draw[->, shorten >=5pt] (-.35,-.61) to [in=130,out=50] (.35,-.61) ;
\draw[->, shorten >=5pt] (-.35,-.61) to (0,0);
\draw[->, shorten >=5pt] (0,0) to[in=130,out=50, loop] (0,0);
\draw[->, shorten >=5pt] (-.35,-.61) to[in=250,out=170, loop] (-.35,-.61);
\end{tikzpicture}}

\newsavebox\Cfour
\sbox\Cfour{\begin{tikzpicture}[scale=2.5]
\fill (0,0)  circle[radius=.6pt];
\fill (.5,0)  circle[radius=.6pt];
\fill (0,-.5)  circle[radius=.6pt];
\fill (.5,-.5)  circle[radius=.6pt];
\draw[->, shorten >=5pt] (0,0) to (.5,0);
\draw[->, shorten >=5pt]  (0,-.5) to (0,0);
\draw[->, shorten >=5pt]  (.5,-.5) to (0,-.5);
\draw[->, shorten >=5pt] (.5,0) to (.5,-.5);
\draw[->, shorten >=5pt] (0,-.5) to[in=200, out=70] (.5,0);
\draw[->, shorten >=5pt] (.5,0) to[in=20, out=250] (0,-.5);
\draw[->, shorten >=5pt] (.5,-.5) to[in=-45, out=45] (.5,0);
\draw[->, shorten >=5pt] (0,0) to[in=175,out=95, loop] (0,0);
\draw[->, shorten >=5pt] (.5,-.5) to[in=275, out=355, loop] (0,-.5);
\end{tikzpicture}}

\newsavebox\Cfive
\sbox\Cfive{\begin{tikzpicture}[scale=2.5]
\fill (0,0)  circle[radius=.6pt];
\fill (.5,0)  circle[radius=.6pt];
\fill (0,-.5)  circle[radius=.6pt];
\fill (.5,-.5)  circle[radius=.6pt];
\draw[->, shorten >=5pt] (.5,0) to (0,0);
\draw[->, shorten >=5pt] (0,-.5) to (0,0);
\draw[->, shorten >=5pt] (0,-.5) to (.5,-.5);
\draw[->, shorten >=5pt] (.5,0) to (.5,-.5);
\draw[->, shorten >=5pt] (0,0) to (.5,-.5);
\draw[->, shorten >=5pt] (.5,-.5) to[in=-45, out=-135] (0,-.5);
\draw[->, shorten >=5pt] (.5,-.5) to[in=-45, out=45] (.5,0);
\draw[->, shorten >=5pt] (0,0) to[in=175,out=95, loop] (0,0);
\draw[->, shorten >=5pt] (0,-.5) to[in=265,out=185, loop] (0,-.5);
\end{tikzpicture}}

\newsavebox\Done
\sbox\Done{\begin{tikzpicture}[scale=2]
\fill (0,0)  circle[radius=.6pt];
\fill (-.35,-.61)  circle[radius=.6pt];
\fill (.35,-.61)  circle[radius=.6pt];
\draw[->, shorten >=5pt] (0,0) to (.35,-.61);
\draw[->, shorten >=5pt] (.35,-.61) to (-.35,-.61);
\draw[->, shorten >=5pt] (-.35,-.61) to (0,0);
\draw[->, shorten >=5pt] (0,0) to[in=130,out=50, loop] (0,0);
\draw[->, shorten >=5pt] (-.35,-.61) to[in=250,out=170, loop] (-.35,-.61);
\draw[->, shorten >=5pt] (.35,-.61) to[in=0 ,out=60 ] (0,0);
\draw[->, shorten >=5pt] (.35,-.61) to[in=-90 ,out=180 ] (0,0);
\end{tikzpicture}}

\newsavebox\Dtwo
\sbox\Dtwo{\begin{tikzpicture}[scale=2]
\fill (0,0)  circle[radius=.6pt];
\fill (-.35,-.61)  circle[radius=.6pt];
\fill (.35,-.61)  circle[radius=.6pt];
\draw[->, shorten >=5pt] (0,0) to (.35,-.61);
\draw[->, shorten >=5pt] (.35,-.61) to [in=0,out=60] (0,0);
\draw[->, shorten >=5pt] (.35,-.61) to (-.35,-.61);
\draw[->, shorten >=5pt] (-.35,-.61) to [in=-120,out=-60] (.35,-.61) ;
\draw[->, shorten >=5pt] (-.35,-.61) to (0,0);
\draw[->, shorten >=5pt] (0,0) to[in=130,out=50, loop] (0,0);
\draw[->, shorten >=5pt] (-.35,-.61) to[in=250,out=170, loop] (-.35,-.61);
\end{tikzpicture}} 

\newsavebox\key
\sbox\key{\begin{tikzpicture}
\draw [-{Stealth[length=3mm]}] (0,0) to (1,0); 
\draw (-1.5,0) node{Out-splits to:};
\draw[-{Diamond[open]}, thick] (0,-1)--(1,-1);
\draw (-1.5,-1) node{In-splits I+ to:};
\end{tikzpicture}}

\newcommand{\Zz}{\mathcal{Z}}
\newcommand{\Nz}{\mathbb{Z}}

\newcommand{\ES}{\operatorname{ES}}
\newcommand{\Ss}{\operatorname{SS}}

\newcommand{\K}{\mathsf k}

\def\Zz{\mathbb{Z}}

\newcommand{\GL}{\operatorname{GL}}

\newcommand{\id}{\operatorname{id}}

\newcommand{\End}{\operatorname{End}}

\newcommand\Gr[1][]{{\operatorname{{Gr}-}}}
\newcommand{\Modd}{\operatorname{Mod-}}

\newcommand{\M}{\operatorname{\mathbb M}}

\newtheorem{thm}{Theorem}[section]
\newtheorem{cor}[thm]{Corollary}
\newtheorem{lem}[thm]{Lemma}
\newtheorem{prop}[thm]{Proposition}

\newtheorem{ques}{Question}

\theoremstyle{definition}

\theoremstyle{remark}
\newtheorem{rmk}[thm]{Remark}

\newtheorem{example}[thm]{Example}

\newtheorem{defn}[thm]{Definition}

\newcommand{\gr}{\operatorname{gr}}

\newcommand{\uloopr}[1]{\ar@'{@+{[0,0]+(-4,5)}@+{[0,0]+(0,10)}@+{[0,0] +(4,5)}}^{#1}}
\newcommand{\uloopd}[1]{\ar@'{@+{[0,0]+(5,4)}@+{[0,0]+(10,0)}@+{[0,0]+ (5,-4)}}^{#1}}
\newcommand{\dloopr}[1]{\ar@'{@+{[0,0]+(-4,-5)}@+{[0,0]+(0,-10)}@+{[0, 0]+(4,-5)}}_{#1}}
\newcommand{\dloopd}[1]{\ar@'{@+{[0,0]+(-5,4)}@+{[0,0]+(-10,0)}@+{[0,0 ]+(-5,-4)}}_{#1}}
\newcommand{\luloop}[1]{\ar@'{@+{[0,0]+(-8,2)}@+{[0,0]+(-10,10)}@+{[0, 0]+(2,2)}}^{#1}}

\author{Roozbeh Hazrat}
\address{ 
Centre for Research in Mathematics and Data Science\\
Western Sydney University\\
Australia} \email{r.hazrat@westernsydney.edu.au}

\author{Elizabeth Pacheco}
\address{
Centre for Research in Mathematics and Data Science \\
Western Sydney University\\
Australia} \email{e.pacheco@westernsydney.edu.au}

\begin{document}

\title[Williams' and Graded Equivalence Conjectures]{Williams' and Graded Equivalence Conjectures for small graphs}

\begin{abstract}
We prove what might have been expected: The Williams Conjecture in symbolic dynamics and Graded Classification Conjecture for Leavitt/$C^*$-graph algebras hold for ``small graphs'', i.e., connected graphs with three vertices, no parallel edges, no sinks and with no trivial hereditary and saturated subsets.  Namely, two small graphs are shift equivalent if and only if they are strong shift equivalent if and only if their Leavitt/$C^*$-graph algebras are graded/equivariant Morita equivalent. 

\end{abstract}

\maketitle

\section{introduction}

One of the fundamental objects of study in symbolic dynamics is called a \emph{shift of finite type} (SFT), which consists of sequences (often indexed by 
$\mathbb{Z}$) of symbols chosen from a finite set,   that do not include certain ``forbidden'' finite sequences, equipped with a \emph{shift map} which creates a dynamical behaviour.   The applications  abound, from topological quantum field theory, ergodic
theory, and statistical mechanics to coding and information theory~\cite{Kim-RoushW1999,Lind-Marcus2021}. 

The fundamental work of Williams in the 1970's~\cite{williams} shows that understanding these dynamical systems reduces to understanding certain relations between matrices: two shifts of finite type are topologically  ``conjugate'' if and only if their underlying matrices (or graphs)  are  ``strong shift equivalent''. Williams also introduced the notion of ``shift equivalent'' which were originally thought to be the same as strong shift equivalent but rather more tractable. However, it took 20 years to establish that they are not the same~\cite{kimroush} and after 50 years the extremely subtle  distinction between shift equivalence and strong shift equivalence still remains elusive. 

It was already noticed by Cuntz and Krieger in the 1980's~\cite{ck} that the notion of shifts of finite type and their equivalences can be related to invariances of certain graph $C^*$-algebras. This line of research was  pursued and substantially developed by Matsumoto and others (see \cite{matsumoto, ER} and the references there).  Leavitt path algebras, which are the discrete version of graph $C^*$-algebras, were introduced in 2005 and the classification of these algebras and their relationships to shifts of finite type became an active line of research~\cite{lpabook, willie}. 
As they are purely algebraic objects, they facilitate connections with other areas of algebra such as representation theory and even chip firing~\cite{Gene}.
 One of the main conjectures that would relate these topics is the Graded Morita Classification Conjecture: Two square matrices (with no zero rows) are shift equivalent if and only if their Leavitt path algebras are ``graded Morita equivalent'' if and only if their graph $C^*$-algebras are ``equivariant Morita equivalent''.   
 Reformulating this in terms of $K$-theory enables us to eliminate the no-sink assumption from the conjecture:
 for two finite graphs $E$ and $F$, the Leavitt path algebra $L_\K(E)$ is graded Morita equivalent to $L_\K(F)$ if and only if there is an order-preserving $\mathbb Z[x,x^{-1}]$-module isomorphism $K_0^{\gr}(L_\K(E)) \cong K_0^{\gr}(L_\K(F))$. A similar conjecture can be written in the analytic setting~\cite{willie}: The graph $C^*$-algebras \(C^*(E)\) and \(C^*(F)\) are equivariant Morita equivalent if and only if $K_0^{\mathbb T}(C^*(E)) \cong K_0^{\mathbb T}(C^*(F))$ as order-preserving $\mathbb Z[x,x^{-1}]$-modules. If under these isomorphisms the order units map to each other, then conjectures assert that these algebras are graded/equivariant isomorphic (\cite{mathann, HazaratDyn2013, haztom}).   We refer the reader to \cite{abconj, AP,guidohom,guidowillie,bilich, carlsen,eilers2,vas} for works on the graded classification conjecture and \cite{willie} for a comprehensive survey.

In this paper we show that these conjectures hold for \emph{small graphs}: graphs on three vertices, no parallel edges, and where the only hereditary and saturated subsets of vertices are the empty set and the whole vertex set. These are the graphs for which the associated Leavitt path algebra is purely infinite simple~\cite{lpabook,AALP}. There are 34 such graphs as follows:

\begin{longtable}{ccccc}
\begin{tikzpicture}[scale=2.0, every node/.style={scale=0.5}]
\fill (0,0)  circle[radius=.6pt];
\fill (-.35,-.61)  circle[radius=.6pt];
\fill (.35,-.61)  circle[radius=.6pt];
\draw (-.5,0) node{\Large $E_1^1:$};
\draw (0,-.15) node{$v_1$};
\draw (-.35+.13,-.535) node{$v_3$};
\draw (.35-.13,-.535) node{$v_2$};
\draw[->, shorten >=5pt] (0,0) to (.35,-.61);
\draw[->, shorten >=5pt] (.35,-.61) to (-.35,-.61);
\draw[->, shorten >=5pt] (-.35,-.61) to (0,0);
\draw[->, shorten >=5pt] (-.35,-.61) to[in=250,out=170, loop] 
(-.35,-.61);

\draw[draw=white,->, shorten >=5pt] (0,0) to[in=130,out=50, loop] (0,0);
\draw[draw=white,->, shorten >=5pt] (.35,-.61) to[in=10,out=-70, loop] (-.35,-.61);
\end{tikzpicture}
&
\begin{tikzpicture}[scale=2.0, every node/.style={scale=0.5}]
\fill (0,0)  circle[radius=.6pt];
\fill (-.35,-.61)  circle[radius=.6pt];
\fill (.35,-.61)  circle[radius=.6pt];
\draw (0,-.15) node{$v_1$};
\draw (-.5, 0) node{\Large $E^1_2:$};
\draw (-.35+.13,-.535) node{$v_3$};
\draw (.35-.13,-.535) node{$v_2$};
\draw[->, shorten >=5pt] (0,0) to (.35,-.61);
\draw[->, shorten >=5pt] (.35,-.61) to (-.35,-.61);
\draw[->, shorten >=5pt] (-.35,-.61) to (0,0);
\draw[->, shorten >=5pt] (0,0) to[in=130,out=50, loop] (0,0);
\draw[->, shorten >=5pt] (-.35,-.61) to[in=250,out=170, loop] (-.35,-.61);

\draw[draw=white, ->, shorten >=5pt] (.35,-.61) to[in=10,out=-70, loop] (-.35,-.61);
\end{tikzpicture}
&
\begin{tikzpicture}[scale=2.0, every node/.style={scale=0.5}]
\fill (0,0)  circle[radius=.6pt];
\fill (-.35,-.61)  circle[radius=.6pt];
\fill (.35,-.61)  circle[radius=.6pt];
\draw (0,-.15) node{$v_1$};
\draw (-.5, 0) node{\Large $E^1_3:$};
\draw (-.35+.13,-.535) node{$v_3$};
\draw (.35-.13,-.535) node{$v_2$};
\draw[->, shorten >=5pt] (0,0) to (.35,-.61);
\draw[->, shorten >=5pt] (.35,-.61) to (-.35,-.61);
\draw[->, shorten >=5pt] (-.35,-.61) to (0,0);
\draw[->, shorten >=5pt] (0,0) to[in=130,out=50, loop] (0,0);
\draw[->, shorten >=5pt] (-.35,-.61) to[in=250,out=170, loop] (-.35,-.61);
\draw[->, shorten >=5pt] (.35,-.61) to[in=10,out=-70, loop] (-.35,-.61);
\end{tikzpicture}  
&
\begin{tikzpicture}[scale=2.0, every node/.style={scale=0.5}]
\fill (0,0)  circle[radius=.6pt];
\fill (-.35,-.61)  circle[radius=.6pt];
\fill (.35,-.61)  circle[radius=.6pt];
\draw (0,-.15) node{$v_1$};
\draw (-.5, 0) node{\Large $E^1_4:$};
\draw (-.35+.13,-.535) node{$v_3$};
\draw (.35-.13,-.535) node{$v_2$};
\draw[->, shorten >=5pt] (0,0) to (.35,-.61);
\draw[->, shorten >=5pt] (.35,-.61) to (-.35,-.61);
\draw[->, shorten >=5pt] (-.35,-.61) to [in=-120,out=-60] (.35,-.61) ;
\draw[->, shorten >=5pt] (-.35,-.61) to (0,0);
\draw[draw=white, ->, shorten >=5pt] (0,0) to[in=130,out=50, loop] (0,0);
\draw[draw=white, ->, shorten >=5pt] (-.35,-.61) to[in=250,out=170, loop] (-.35,-.61);
\draw[draw=white, ->, shorten >=5pt] (.35,-.61) to[in=10,out=-70, loop] (-.35,-.61);
\end{tikzpicture}
&
\begin{tikzpicture}[scale=2.0, every node/.style={scale=0.5}]
\fill (0,0)  circle[radius=.6pt];
\fill (-.35,-.61)  circle[radius=.6pt];
\fill (.35,-.61)  circle[radius=.6pt];
\draw (0,-.15) node{$v_1$};
\draw (-.5, 0) node{\Large $E^1_5:$};
\draw (-.35+.13,-.535) node{$v_3$};
\draw (.35-.13,-.535) node{$v_2$};
\draw[->, shorten >=5pt] (0,0) to (.35,-.61);
\draw[->, shorten >=5pt] (.35,-.61) to (-.35,-.61);
\draw[->, shorten >=5pt] (-.35,-.61) to [in=-120,out=-60] (.35,-.61) ;
\draw[->, shorten >=5pt] (-.35,-.61) to (0,0);
\draw[->, shorten >=5pt] (0,0) to[in=130,out=50, loop] (0,0);

\draw[draw=white, ->, shorten >=5pt] (-.35,-.61) to[in=250,out=170, loop] (-.35,-.61);
\draw[draw=white, ->, shorten >=5pt] (.35,-.61) to[in=10,out=-70, loop] (-.35,-.61);
\end{tikzpicture}
\\
\begin{tikzpicture}[scale=2.0, every node/.style={scale=0.5}]
\fill (0,0)  circle[radius=.6pt];
\fill (-.35,-.61)  circle[radius=.6pt];
\fill (.35,-.61)  circle[radius=.6pt];
\draw (0,-.15) node{$v_1$};
\draw (-.5, 0) node{\Large $E^1_6:$};
\draw (-.35+.13,-.535) node{$v_3$};
\draw (.35-.13,-.535) node{$v_2$};
\draw[->, shorten >=5pt] (0,0) to (.35,-.61);
\draw[->, shorten >=5pt] (.35,-.61) to (-.35,-.61);
\draw[->, shorten >=5pt] (-.35,-.61) to [in=-120,out=-60] (.35,-.61) ;
\draw[->, shorten >=5pt] (-.35,-.61) to (0,0);
\draw[->, shorten >=5pt] (0,0) to[in=130,out=50, loop] (0,0);
\draw[ ->, shorten >=5pt] (-.35,-.61) to [in=250,out=170, loop] (-.35,-.61);

\draw[draw=white, ->, shorten >=5pt] (.35,-.61) to[in=10,out=-70, loop] (-.35,-.61);
\end{tikzpicture}
&
\begin{tikzpicture}[scale=2.0, every node/.style={scale=0.5}]
\fill (0,0)  circle[radius=.6pt];
\fill (-.35,-.61)  circle[radius=.6pt];
\fill (.35,-.61)  circle[radius=.6pt];
\draw (0,-.15) node{$v_1$};
\draw (-.5, 0) node{\Large $E^1_7:$};
\draw (-.35+.13,-.535) node{$v_3$};
\draw (.35-.13,-.535) node{$v_2$};
\draw[->, shorten >=5pt] (0,0) to (.35,-.61);
\draw[->, shorten >=5pt] (.35,-.61) to (-.35,-.61);
\draw[->, shorten >=5pt] (-.35,-.61) to [in=-120,out=-60] (.35,-.61) ;
\draw[->, shorten >=5pt] (-.35,-.61) to (0,0);
\draw[->, shorten >=5pt] (0,0) to[in=130,out=50, loop] (0,0);
\draw[->, shorten >=5pt] (.35,-.61) to [in=10,out=-70, loop] (.35,-.61);
\draw[draw=white, ->, shorten >=5pt] (-.35,-.61) to[in=250,out=170, loop] (-.35,-.61);

\end{tikzpicture}
&
\begin{tikzpicture}[scale=2.0, every node/.style={scale=0.5}]
\fill (0,0)  circle[radius=.6pt];
\fill (-.35,-.61)  circle[radius=.6pt];
\fill (.35,-.61)  circle[radius=.6pt];
\draw (0,-.15) node{$v_1$};
\draw (-.5, 0) node{\Large $E^1_8$:};
\draw (-.35+.13,-.535) node{$v_3$};
\draw (.35-.13,-.535) node{$v_2$};
;
\draw[->, shorten >=5pt] (0,0) to (.35,-.61);
\draw[->, shorten >=5pt] (.35,-.61) to [in=0,out=60] (0,0);
\draw[->, shorten >=5pt] (.35,-.61) to (-.35,-.61);
\draw[->, shorten >=5pt] (-.35,-.61) to [in=-120,out=-60] (.35,-.61) ;
\draw[->, shorten >=5pt] (-.35,-.61) to (0,0);
\draw[->, shorten >=5pt] (0,0) to[in=130,out=50, loop] (0,0);
\draw[->, shorten >=5pt] (-.35,-.61) to[in=250,out=170, loop] (-.35,-.61);
\draw[draw=white,->, shorten >=5pt] (.35,-.61) to[in=10,out=-70, loop] (-.35,-.61);
\end{tikzpicture} 
&
\begin{tikzpicture}[scale=2.0, every node/.style={scale=0.5}]
\fill (0,0)  circle[radius=.6pt];
\fill (-.35,-.61)  circle[radius=.6pt];
\fill (.35,-.61)  circle[radius=.6pt];
\draw (0,-.15) node{$v_1$};
\draw (-.5, 0) node{\Large $E^1_9:$};
\draw (-.35+.13,-.535) node{$v_3$};
\draw (.35-.13,-.535) node{$v_2$};
\draw[->, shorten >=5pt] (0,0) to (.35,-.61);
\draw[->, shorten >=5pt] (.35,-.61) to [in=0,out=60] (0,0);
\draw[->, shorten >=5pt] (.35,-.61) to (-.35,-.61);
\draw[->, shorten >=5pt] (-.35,-.61) to [in=-120,out=-60] (.35,-.61) ;
\draw[->, shorten >=5pt] (-.35,-.61) to (0,0);
\draw[->, shorten >=5pt] (0,0) to[in=130,out=50, loop] (0,0);
\draw[->, shorten >=5pt] (-.35,-.61) to[in=250,out=170, loop] (-.35,-.61);
\draw[->, shorten >=5pt] (.35,-.61) to[in=10,out=-70, loop] (-.35,-.61);
\end{tikzpicture}
&
\begin{tikzpicture}[scale=2.0, every node/.style={scale=0.5}]
\fill (0,0)  circle[radius=.6pt];
\fill (-.35,-.61)  circle[radius=.6pt];
\fill (.35,-.61)  circle[radius=.6pt];
\draw (0,-.15) node{$v_1$};
\draw (-.5, 0) node{\Large $E^1_{10}:$};
\draw (-.35+.13,-.535) node{$v_3$};
\draw (.35-.13,-.535) node{$v_2$};
\draw[->, shorten >=5pt] (0,0) to (.35,-.61);
\draw[->, shorten >=5pt] (.35,-.61) to [in=0,out=60] (0,0);
\draw[->, shorten >=5pt] (-.35,-.61) to (0,0);
\draw[->, shorten >=5pt] (0,0) to [in=120,out=180
] (-.35,-.61) ;
\draw[draw=white, ->, shorten >=5pt] (0,0) to[in=130,out=50, loop] (0,0);
\draw[draw=white, ->, shorten >=5pt] (-.35,-.61) to[in=250,out=170, loop] (-.35,-.61);
\draw[draw=white, ->, shorten >=5pt] (.35,-.61) to[in=10,out=-70, loop] (-.35,-.61);
\end{tikzpicture}
\\
\begin{tikzpicture}[scale=2.0, every node/.style={scale=0.5}]
\fill (0,0)  circle[radius=.6pt];
\fill (-.35,-.61)  circle[radius=.6pt];
\fill (.35,-.61)  circle[radius=.6pt];
\draw (0,-.15) node{$v_1$};
\draw (-.5, 0) node{\Large $E^1_{11}$:};
\draw (-.35+.13,-.535) node{$v_3$};
\draw (.35-.13,-.535) node{$v_2$};
\draw[->, shorten >=5pt] (0,0) to (.35,-.61);
\draw[->, shorten >=5pt] (.35,-.61) to [in=0,out=60] (0,0);
\draw[->, shorten >=5pt] (-.35,-.61) to (0,0);
\draw[->, shorten >=5pt] (0,0) to [in=120,out=180
] (-.35,-.61) ;
\draw[->, shorten >=5pt] (-.35,-.61) to[in=250,out=170, loop] (-.35,-.61);
\draw[draw=white, ->, shorten >=5pt] (0,0) to[in=130,out=50, loop] (0,0);
\draw[draw=white, ->, shorten >=5pt] (.35,-.61) to[in=10,out=-70, loop] (-.35,-.61);
\end{tikzpicture}
&
\begin{tikzpicture}[scale=2.0, every node/.style={scale=0.5}]
\fill (0,0)  circle[radius=.6pt];
\fill (-.35,-.61)  circle[radius=.6pt];
\fill (.35,-.61)  circle[radius=.6pt];
\draw (0,-.15) node{$v_1$};
\draw (-.5, 0) node{\Large $E^1_{12}:$};
\draw (-.35+.13,-.535) node{$v_3$};
\draw (.35-.13,-.535) node{$v_2$};
\draw[->, shorten >=5pt] (0,0) to (.35,-.61);
\draw[->, shorten >=5pt] (.35,-.61) to [in=0,out=60] (0,0);
\draw[->, shorten >=5pt] (-.35,-.61) to (0,0);
\draw[->, shorten >=5pt] (0,0) to [in=120,out=180
] (-.35,-.61) ;
\draw[->, shorten >=5pt] (0,0) to[in=130,out=50, loop] (0,0);
\draw[->, shorten >=5pt] (-.35,-.61) to[in=250,out=170, loop] (-.35,-.61);
\draw[draw=white, ->, shorten >=5pt] (.35,-.61) to[in=10,out=-70, loop] (-.35,-.61);
\end{tikzpicture}
&
\begin{tikzpicture}[scale=2.0, every node/.style={scale=0.5}]
\fill (0,0)  circle[radius=.6pt];
\fill (-.35,-.61)  circle[radius=.6pt];
\fill (.35,-.61)  circle[radius=.6pt];
\draw (0,-.15) node{$v_1$};
\draw (-.5, 0) node{\Large $E^1_{13}:$};
\draw (-.35+.13,-.535) node{$v_3$};
\draw (.35-.13,-.535) node{$v_2$};
\draw[->, shorten >=5pt] (.35,-.61) to (0,0);
\draw[->, shorten >=5pt] (-.35,-.61) to (0,0);
\draw[->, shorten >=5pt] (0,0) to [in=120,out=180
] (-.35,-.61) ;
\draw[->, shorten >=5pt] (-.35,-.61) to [in=250,out=170, loop] (-.35,-.61);
\draw[draw=white, ->, shorten >=5pt] (0,0) to[in=130,out=50, loop] (0,0);
\draw[draw=white, ->, shorten >=5pt] (.35,-.61) to[in=10,out=-70, loop] (-.35,-.61);
\end{tikzpicture}
&
\begin{tikzpicture}[scale=2.0, every node/.style={scale=0.5}]
\fill (0,0)  circle[radius=.6pt];
\fill (-.35,-.61)  circle[radius=.6pt];
\fill (.35,-.61)  circle[radius=.6pt];
\draw (0,-.15) node{$v_1$};
\draw (-.5, 0) node{\Large $E^1_{14}:$};
\draw (-.35+.13,-.535) node{$v_3$};
\draw (.35-.13,-.535) node{$v_2$};
\draw[->, shorten >=5pt] (.35,-.61) to (0,0);
\draw[->, shorten >=5pt] (-.35,-.61) to (0,0);
\draw[->, shorten >=5pt] (0,0) to [in=120,out=180
] (-.35,-.61) ;
\draw[->, shorten >=5pt] (0,0) to[in=130,out=50, loop] (0,0);
\draw[draw=white, ->, shorten >=5pt] (-.35,-.61) to [in=250,out=170, loop] (-.35,-.61);
\draw[draw=white, ->, shorten >=5pt] (-.35,-.61) to[in=250,out=170, loop] (-.35,-.61);
\draw[draw=white, ->, shorten >=5pt] (.35,-.61) to[in=10,out=-70, loop] (-.35,-.61);
\end{tikzpicture}
&
\begin{tikzpicture}[scale=2.0, every node/.style={scale=0.5}]
\fill (0,0)  circle[radius=.6pt];
\fill (-.35,-.61)  circle[radius=.6pt];
\fill (.35,-.61)  circle[radius=.6pt];
\draw (0,-.15) node{$v_1$};
\draw (-.5, 0) node{\Large $E^1_{15}:$};
\draw (-.35+.13,-.535) node{$v_3$};
\draw (.35-.13,-.535) node{$v_2$};
\draw[->, shorten >=5pt] (.35,-.61) to (0,0);
\draw[->, shorten >=5pt] (-.35,-.61) to (0,0);
\draw[->, shorten >=5pt] (0,0) to [in=120,out=180
] (-.35,-.61) ;
\draw[->, shorten >=5pt] (0,0) to[in=130,out=50, loop] (0,0);
\draw[->, shorten >=5pt] (-.35,-.61) to[in=250,out=170, loop] (-.35,-.61);
\draw[draw=white, ->, shorten >=5pt] (.35,-.61) to[in=10,out=-70, loop] (-.35,-.61);
\end{tikzpicture}
\\

\begin{tikzpicture}[scale=2.0, every node/.style={scale=0.5}]
\fill (0,0)  circle[radius=.6pt];
\fill (-.35,-.61)  circle[radius=.6pt];
\fill (.35,-.61)  circle[radius=.6pt];
\draw (0,-.15) node{$v_1$};
\draw (-.5, 0) node{\Large $E^1_{16}:$};
\draw (-.35+.13,-.535) node{$v_3$};
\draw (.35-.13,-.535) node{$v_2$};
\draw[->, shorten >=5pt] (0,0) to (.35,-.61);
\draw[->, shorten >=5pt] (.35,-.61) to (-.35,-.61);
\draw[->, shorten >=5pt] (-.35,-.61) to [in=-120,out=-60] (.35,-.61) ;
\draw[->, shorten >=5pt] (0,0) to (-.35,-.61) ;
\draw[->, shorten >=5pt] (-.35,-.61) to[in=250,out=170, loop] (-.35,-.61);
\draw[draw=white, ->, shorten >=5pt] (0,0) to[in=130,out=50, loop] (0,0);
\draw[draw=white, ->, shorten >=5pt] (.35,-.61) to[in=10,out=-70, loop] (-.35,-.61);
\end{tikzpicture}

&
\begin{tikzpicture}[scale=2.0, every node/.style={scale=0.5}]
\fill (0,0)  circle[radius=.6pt];
\fill (-.35,-.61)  circle[radius=.6pt];
\fill (.35,-.61)  circle[radius=.6pt];
\draw (0,-.15) node{$v_1$};
\draw (-.5, 0) node{\Large $E^1_{17}$:};
\draw (-.35+.13,-.535) node{$v_3$};
\draw (.35-.13,-.535) node{$v_2$};
\draw[->, shorten >=5pt] (0,0) to (.35,-.61);
\draw[->, shorten >=5pt] (.35,-.61) to (-.35,-.61);
\draw[->, shorten >=5pt] (-.35,-.61) to [in=-120,out=-60] (.35,-.61) ;
\draw[->, shorten >=5pt] (0,0) to (-.35,-.61);
\draw[->, shorten >=5pt] (-.35,-.61) to[in=250,out=170, loop] (-.35,-.61);
\draw[->, shorten >=5pt] (.35,-.61) to[in=10,out=-70, loop] (-.35,-.61);
\draw[draw=white, ->, shorten >=5pt] (0,0) to[in=130,out=50, loop] (0,0);
\end{tikzpicture}
&
\begin{tikzpicture}[scale=2.0, every node/.style={scale=0.5}]
\fill (0,0)  circle[radius=.6pt];
\fill (-.35,-.61)  circle[radius=.6pt];
\fill (.35,-.61)  circle[radius=.6pt];
\draw (0,-.15) node{$v_1$};
\draw (-.5, 0) node{\Large $E^1_{18}$};
\draw (-.35+.13,-.535) node{$v_3$};
\draw (.35-.13,-.535) node{$v_2$};
\draw[->, shorten >=5pt] (0,0) to (.35,-.61);
\draw[->, shorten >=5pt] (.35,-.61) to (-.35,-.61);
\draw[->, shorten >=5pt] (-.35,-.61) to [in=-120,out=-60] (.35,-.61) ;
\draw[->, shorten >=5pt] (-.35,-.61) to (0,0);
\draw[->, shorten >=5pt] (0,0) to[in=130,out=50, loop] (0,0);
\draw[->, shorten >=5pt] (-.35,-.61) to[in=250,out=170, loop] (-.35,-.61);
\draw[->, shorten >=5pt] (.35,-.61) to[in=10,out=-70, loop] (-.35,-.61);
\end{tikzpicture}
&
\begin{tikzpicture}[scale=2.0, every node/.style={scale=0.5}]
\fill (0,0)  circle[radius=.6pt];
\fill (-.35,-.61)  circle[radius=.6pt];
\fill (.35,-.61)  circle[radius=.6pt];
\draw (0,-.15) node{$v_1$};
\draw (-.5, 0) node{\Large $E^2_1$:};
\draw (-.35+.13,-.535) node{$v_3$};
\draw (.35-.13,-.535) node{$v_2$};
\draw[->, shorten >=5pt] (0,0) to (.35,-.61);
\draw[->, shorten >=5pt] (.35,-.61) to (-.35,-.61);
\draw[->, shorten >=5pt] (-.35,-.61) to [in=-120,out=-60] (.35,-.61) ;
\draw[->, shorten >=5pt] (-.35,-.61) to (0,0);
\draw[->, shorten >=5pt] (-.35,-.61) to[in=250,out=170, loop] (-.35,-.61);
\draw[->, shorten >=5pt] (.35,-.61) to[in=10,out=-70, loop] (-.35,-.61);
\draw[draw=white, ->, shorten >=5pt] (0,0) to[in=130,out=50, loop] (0,0);
\end{tikzpicture}
&
\begin{tikzpicture}[scale=2.0, every node/.style={scale=0.5}]
\fill (0,0)  circle[radius=.6pt];
\fill (-.35,-.61)  circle[radius=.6pt];
\fill (.35,-.61)  circle[radius=.6pt];
\draw (0,-.15) node{$v_1$};
\draw (-.5, 0) node{\Large $E^2_2$};
\draw (-.35+.13,-.535) node{$v_3$};
\draw (.35-.13,-.535) node{$v_2$};
\draw[->, shorten >=5pt] (0,0) to (.35,-.61);
\draw[->, shorten >=5pt] (.35,-.61) to [in=0,out=60] (0,0);
\draw[->, shorten >=5pt] (.35,-.61) to (-.35,-.61);
\draw[->, shorten >=5pt] (-.35,-.61) to [in=-120,out=-60] (.35,-.61) ;
\draw[->, shorten >=5pt] (-.35,-.61) to (0,0);
\draw[draw=white, ->, shorten >=5pt] (0,0) to[in=130,out=50, loop] (0,0);
\draw[draw=white, ->, shorten >=5pt] (-.35,-.61) to[in=250,out=170, loop] (-.35,-.61);
\draw[draw=white, ->, shorten >=5pt] (.35,-.61) to[in=10,out=-70, loop] (-.35,-.61);
\end{tikzpicture}
\\

\begin{tikzpicture}[scale=2.0, every node/.style={scale=0.5}]
\fill (0,0)  circle[radius=.6pt];
\fill (-.35,-.61)  circle[radius=.6pt];
\fill (.35,-.61)  circle[radius=.6pt];
\draw (0,-.15) node{$v_1$};
\draw (-.5, 0) node{\Large $E^2_3:$};
\draw (-.35+.13,-.535) node{$v_3$};
\draw (.35-.13,-.535) node{$v_2$};
\draw[->, shorten >=5pt] (0,0) to (.35,-.61);
\draw[->, shorten >=5pt] (.35,-.61) to [in=0,out=60] (0,0);
\draw[->, shorten >=5pt] (.35,-.61) to (-.35,-.61);
\draw[->, shorten >=5pt] (-.35,-.61) to [in=-120,out=-60] (.35,-.61) ;
\draw[->, shorten >=5pt] (-.35,-.61) to (0,0);
\draw[->, shorten >=5pt] (-.35,-.61) to[in=250,out=170, loop] (-.35,-.61);
\draw[draw=white, ->, shorten >=5pt] (0,0) to[in=130,out=50, loop] (0,0);
\draw[draw=white, ->, shorten >=5pt] (.35,-.61) to[in=10,out=-70, loop] (-.35,-.61);
\end{tikzpicture}
&
\begin{tikzpicture}[scale=2.0, every node/.style={scale=0.5}]
\fill (0,0)  circle[radius=.6pt];
\fill (-.35,-.61)  circle[radius=.6pt];
\fill (.35,-.61)  circle[radius=.6pt];
\draw (0,-.15) node{$v_1$};
\draw (-.5, 0) node{\Large $E^2_4:$};
\draw (-.35+.13,-.535) node{$v_3$};
\draw (.35-.13,-.535) node{$v_2$};
\draw[->, shorten >=5pt] (0,0) to (.35,-.61);
\draw[->, shorten >=5pt] (.35,-.61) to [in=0,out=60] (0,0);
\draw[->, shorten >=5pt] (.35,-.61) to (-.35,-.61);
\draw[->, shorten >=5pt] (-.35,-.61) to [in=-120,out=-60] (.35,-.61) ;
\draw[->, shorten >=5pt] (-.35,-.61) to (0,0);
\draw[->, shorten >=5pt] (0,0) to[in=130,out=50, loop] (0,0);
\draw[draw=white, ->, shorten >=5pt] (-.35,-.61) to[in=250,out=170, loop] (-.35,-.61);
\draw[draw=white, ->, shorten >=5pt] (.35,-.61) to[in=10,out=-70, loop] (-.35,-.61);
\end{tikzpicture}
&

\begin{tikzpicture}[scale=2.0, every node/.style={scale=0.5}]
\fill (0,0)  circle[radius=.6pt];
\fill (-.35,-.61)  circle[radius=.6pt];
\fill (.35,-.61)  circle[radius=.6pt];
\draw (0,-.15) node{$v_1$};
\draw (-.5, 0) node{\Large $E^2_5:$};
\draw (-.35+.13,-.535) node{$v_3$};
\draw (.35-.13,-.535) node{$v_2$};
\draw[->, shorten >=5pt] (0,0) to (.35,-.61);
\draw[->, shorten >=5pt] (.35,-.61) to (-.35,-.61);
\draw[->, shorten >=5pt] (-.35,-.61) to [in=-120,out=-60] (.35,-.61) ;
\draw[->, shorten >=5pt] (-.35,-.61) to (0,0);
\draw[->, shorten >=5pt] (.35,-.61) to[in=10,out=-70, loop] (-.35,-.61);
\draw[draw=white, ->, shorten >=5pt] (0,0) to[in=130,out=50, loop] (0,0);
\draw[draw=white, ->, shorten >=5pt] (-.35,-.61) to[in=250,out=170, loop] (-.35,-.61);
\end{tikzpicture}
&
\begin{tikzpicture}[scale=2.0, every node/.style={scale=0.5}]
\fill (0,0)  circle[radius=.6pt];
\fill (-.35,-.61)  circle[radius=.6pt];
\fill (.35,-.61)  circle[radius=.6pt];
\draw (0,-.15) node{$v_1$};
\draw (-.5, 0) node{\Large $E^2_6:$};
\draw (-.35+.13,-.535) node{$v_3$};
\draw (.35-.13,-.535) node{$v_2$};
\draw[->, shorten >=5pt] (0,0) to (.35,-.61);
\draw[->, shorten >=5pt] (.35,-.61) to [in=0,out=60] (0,0);
\draw[->, shorten >=5pt] (.35,-.61) to (-.35,-.61);
\draw[->, shorten >=5pt] (-.35,-.61) to [in=-120,out=-60] (.35,-.61) ;
\draw[->, shorten >=5pt] (-.35,-.61) to (0,0);
\draw[->, shorten >=5pt] (0,0) to[in=130,out=50, loop] (0,0);
\draw[draw=white, ->, shorten >=5pt] (-.35,-.61) to[in=250,out=170, loop] (-.35,-.61);
\draw[->, shorten >=5pt] (.35,-.61) to[in=10,out=-70, loop] (-.35,-.61);
\end{tikzpicture}

&
\begin{tikzpicture}[scale=2.0, every node/.style={scale=0.5}]
\fill (0,0)  circle[radius=.6pt];
\fill (-.35,-.61)  circle[radius=.6pt];
\fill (.35,-.61)  circle[radius=.6pt];
\draw (0,-.15) node{$v_1$};
\draw (-.5, 0) node{\Large $E^3_1:$};
\draw (-.35+.13,-.535) node{$v_3$};
\draw (.35-.13,-.535) node{$v_2$};
\draw[->, shorten >=5pt] (0,0) to (.35,-.61);
\draw[->, shorten >=5pt] (.35,-.61) to (-.35,-.61);
\draw[->, shorten >=5pt] (-.35,-.61) to [in=-120,out=-60] (.35,-.61) ;
\draw[->, shorten >=5pt] (-.35,-.61) to (0,0);
\draw[->, shorten >=5pt] (-.35,-.61) to[in=250,out=170, loop] (-.35,-.61);
\draw[draw=white, ->, shorten >=5pt] (0,0) to[in=130,out=50, loop] (0,0);
\draw[draw=white, ->, shorten >=5pt] (.35,-.61) to[in=10,out=-70, loop] (-.35,-.61);
\end{tikzpicture}
\\

\begin{tikzpicture}[scale=2.0, every node/.style={scale=0.5}]
\fill (0,0)  circle[radius=.6pt];
\fill (-.35,-.61)  circle[radius=.6pt];
\fill (.35,-.61)  circle[radius=.6pt];
\draw (0,-.15) node{$v_1$};
\draw (-.5, 0) node{\Large $E^3_2:$};
\draw (-.35+.13,-.535) node{$v_3$};
\draw (.35-.13,-.535) node{$v_2$};
\draw[->, shorten >=5pt] (0,0) to (.35,-.61);
\draw[->, shorten >=5pt] (.35,-.61) to [in=0,out=60] (0,0);
\draw[->, shorten >=5pt] (.35,-.61) to (-.35,-.61);
\draw[->, shorten >=5pt] (-.35,-.61) to [in=-120,out=-60] (.35,-.61) ;
\draw[->, shorten >=5pt] (-.35,-.61) to (0,0);
\draw[->, shorten >=5pt] (-.35,-.61) to[in=250,out=170, loop] (-.35,-.61);
\draw[->, shorten >=5pt] (.35,-.61) to[in=10,out=-70, loop] (-.35,-.61);
\draw[draw=white, ->, shorten >=5pt] (0,0) to[in=130,out=50, loop] (0,0);
\end{tikzpicture}
&
\begin{tikzpicture}[scale=2.0, every node/.style={scale=0.5}]
\fill (0,0)  circle[radius=.6pt];
\fill (-.35,-.61)  circle[radius=.6pt];
\fill (.35,-.61)  circle[radius=.6pt];
\draw (0,-.15) node{$v_1$};
\draw (-.5, 0) node{\Large $E^3_3:$};
\draw (-.35+.13,-.535) node{$v_3$};
\draw (.35-.13,-.535) node{$v_2$};
\draw[->, shorten >=5pt] (0,0) to (.35,-.61);
\draw[->, shorten >=5pt] (.35,-.61) to [in=0,out=60] (0,0);
\draw[->, shorten >=5pt] (.35,-.61) to (-.35,-.61);
\draw[->, shorten >=5pt] (-.35,-.61) to [in=-120,out=-60] (.35,-.61) ;
\draw[->, shorten >=5pt] (-.35,-.61) to (0,0);
\draw[->, shorten >=5pt] (0,0) to [in=120,out=180
] (-.35,-.61) ;
\draw[->, shorten >=5pt] (0,0) to[in=130,out=50, loop] (0,0);
\draw[->, shorten >=5pt] (-.35,-.61) to[in=250,out=170, loop] (-.35,-.61);
\draw[->, shorten >=5pt] (.35,-.61) to[in=10,out=-70, loop] (-.35,-.61);
\end{tikzpicture}
&
\begin{tikzpicture}[scale=2.0, every node/.style={scale=0.5}]
\fill (0,0)  circle[radius=.6pt];
\fill (-.35,-.61)  circle[radius=.6pt];
\fill (.35,-.61)  circle[radius=.6pt];
\draw (0,-.15) node{$v_1$};
\draw (-.5, 0) node{\Large $E^3_4:$};
\draw (-.35+.13,-.535) node{$v_3$};
\draw (.35-.13,-.535) node{$v_2$};
\draw[->, shorten >=5pt] (0,0) to (.35,-.61);
\draw[->, shorten >=5pt] (.35,-.61) to [in=0,out=60] (0,0);
\draw[->, shorten >=5pt] (-.35,-.61) to (0,0);
\draw[->, shorten >=5pt] (0,0) to [in=120,out=180
] (-.35,-.61) ;
\draw[->, shorten >=5pt] (0,0) to[in=130,out=50, loop] (0,0);
\draw[draw=white, ->, shorten >=5pt] (-.35,-.61) to[in=250,out=170, loop] (-.35,-.61);
\draw[draw=white, ->, shorten >=5pt] (.35,-.61) to[in=10,out=-70, loop] (-.35,-.61);
\end{tikzpicture}
&
\begin{tikzpicture}[scale=2.0, every node/.style={scale=0.5}]
\fill (0,0)  circle[radius=.6pt];
\fill (-.35,-.61)  circle[radius=.6pt];
\fill (.35,-.61)  circle[radius=.6pt];
\draw (0,-.15) node{$v_1$};
\draw (-.5, 0) node{\Large $E^4_1:$};
\draw (-.35+.13,-.535) node{$v_3$};
\draw (.35-.13,-.535) node{$v_2$};
\draw[->, shorten >=5pt] (0,0) to (.35,-.61);
\draw[->, shorten >=5pt] (.35,-.61) to [in=0,out=60] (0,0);
\draw[->, shorten >=5pt] (.35,-.61) to (-.35,-.61);
\draw[->, shorten >=5pt] (-.35,-.61) to [in=-120,out=-60] (.35,-.61) ;
\draw[->, shorten >=5pt] (-.35,-.61) to (0,0);
\draw[->, shorten >=5pt] (.35,-.61) to[in=10,out=-70, loop] (-.35,-.61);
\draw[draw=white, ->, shorten >=5pt] (0,0) to[in=130,out=50, loop] (0,0);
\draw[draw=white, ->, shorten >=5pt] (-.35,-.61) to[in=250,out=170, loop] (-.35,-.61);
\end{tikzpicture}
&
\begin{tikzpicture}[scale=2.0, every node/.style={scale=0.5}]
\fill (0,0)  circle[radius=.6pt];
\fill (-.35,-.61)  circle[radius=.6pt];
\fill (.35,-.61)  circle[radius=.6pt];
\draw (0,-.15) node{$v_1$};
\draw (-.5, 0) node{\Large $E^4_2:$};
\draw (-.35+.13,-.535) node{$v_3$};
\draw (.35-.13,-.535) node{$v_2$};
\draw[->, shorten >=5pt] (0,0) to (.35,-.61);
\draw[->, shorten >=5pt] (.35,-.61) to [in=0,out=60] (0,0);
\draw[->, shorten >=5pt] (.35,-.61) to (-.35,-.61);
\draw[->, shorten >=5pt] (-.35,-.61) to [in=-120,out=-60] (.35,-.61) ;
\draw[->, shorten >=5pt] (-.35,-.61) to (0,0);
\draw[->, shorten >=5pt] (0,0) to [in=120,out=180
] (-.35,-.61) ;
\draw[->, shorten >=5pt] (0,0) to[in=130,out=50, loop] (0,0);
\draw[->, shorten >=5pt] (.35,-.61) to[in=10,out=-70, loop] (.35,-.61);
\draw[draw=white, ->, shorten >=5pt] (-.35,-.61) to[in=250,out=170, loop] (-.35,-.61);
\end{tikzpicture}
\\

\begin{tikzpicture}[scale=2.0, every node/.style={scale=0.5}]
\fill (0,0)  circle[radius=.6pt];
\fill (-.35,-.61)  circle[radius=.6pt];
\fill (.35,-.61)  circle[radius=.6pt];
\draw (0,-.15) node{$v_1$};
\draw (-.5, 0) node{\Large $E^5_1:$};
\draw (-.35+.13,-.535) node{$v_3$};
\draw (.35-.13,-.535) node{$v_2$};
\draw[->, shorten >=5pt] (0,0) to (.35,-.61);
\draw[->, shorten >=5pt] (.35,-.61) to [in=0,out=60] (0,0);
\draw[->, shorten >=5pt] (.35,-.61) to (-.35,-.61);
\draw[->, shorten >=5pt] (-.35,-.61) to [in=-120,out=-60] (.35,-.61) ;
\draw[->, shorten >=5pt] (-.35,-.61) to (0,0);
\draw[->, shorten >=5pt] (0,0) to [in=120,out=180
] (-.35,-.61) ;
\draw[->, shorten >=5pt] (0,0) to[in=130,out=50, loop] (0,0);
\draw[draw=white, ->, shorten >=5pt] (.35,-.61) to[in=10,out=-70, loop]  (.35,-.61) ; 
\draw[draw=white, ->, shorten >=5pt] (-.35,-.61) to[in=250,out=170, loop] (-.35,-.61);
\draw[draw=white, ->, shorten >=5pt] (.35,-.61) to[in=10,out=-70, loop] (-.35,-.61);
\end{tikzpicture}
&
\begin{tikzpicture}[scale=2.0, every node/.style={scale=0.5}]
\fill (0,0)  circle[radius=.6pt];
\fill (-.35,-.61)  circle[radius=.6pt];
\fill (.35,-.61)  circle[radius=.6pt];
\draw (0,-.15) node{$v_1$};
\draw (-.5, 0) node{\Large $E^6_1$};
\draw (-.35+.13,-.535) node{$v_3$};
\draw (.35-.13,-.535) node{$v_2$};
\draw[->, shorten >=5pt] (0,0) to (.35,-.61);
\draw[->, shorten >=5pt] (.35,-.61) to [in=0,out=60] (0,0);
\draw[->, shorten >=5pt] (.35,-.61) to (-.35,-.61);
\draw[->, shorten >=5pt] (-.35,-.61) to [in=-120,out=-60] (.35,-.61) ;
\draw[->, shorten >=5pt] (-.35,-.61) to (0,0);
\draw[->, shorten >=5pt] (0,0) to [in=120,out=180
] (-.35,-.61) ;
\draw[draw=white, ->, shorten >=5pt] (0,0) to[in=130,out=50, loop] (0,0);
\draw[draw=white, ->, shorten >=5pt] (-.35,-.61) to[in=250,out=170, loop] (-.35,-.61);
\draw[draw=white, ->, shorten >=5pt] (.35,-.61) to[in=10,out=-70, loop] (-.35,-.61); 
\end{tikzpicture}
&
\begin{tikzpicture}[scale=2.0, every node/.style={scale=0.5}]
\fill (0,0)  circle[radius=.6pt];
\fill (-.35,-.61)  circle[radius=.6pt];
\fill (.35,-.61)  circle[radius=.6pt];
\draw (0,-.15) node{$v_1$};
\draw (-.5, 0) node{\Large $E^7_1:$};
\draw (-.35+.13,-.535) node{$v_3$};
\draw (.35-.13,-.535) node{$v_2$};
\draw[->, shorten >=5pt] (0,0) to (.35,-.61);
\draw[->, shorten >=5pt] (.35,-.61) to [in=0,out=60] (0,0);
\draw[->, shorten >=5pt] (-.35,-.61) to (0,0);
\draw[->, shorten >=5pt] (0,0) to [in=120,out=180
] (-.35,-.61) ;
\draw[->, shorten >=5pt] (-.35,-.61) to[in=250,out=170, loop] (-.35,-.61);
\draw[->, shorten >=5pt] (.35,-.61) to[in=10,out=-70, loop] (-.35,-.61);
\draw[draw=white, ->, shorten >=5pt] (0,0) to[in=130,out=50, loop] (0,0);
\end{tikzpicture}
&
\begin{tikzpicture}[scale=2.0, every node/.style={scale=0.5}]
\fill (0,0)  circle[radius=.6pt];
\fill (-.35,-.61)  circle[radius=.6pt];
\fill (.35,-.61)  circle[radius=.6pt];
\draw (0,-.15) node{$v_1$};
\draw (-.5, 0) node{\Large $E^7_2:$};
\draw (-.35+.13,-.535) node{$v_3$};
\draw (.35-.13,-.535) node{$v_2$};
\draw[->, shorten >=5pt] (0,0) to (.35,-.61);
\draw[->, shorten >=5pt] (.35,-.61) to [in=0,out=60] (0,0);
\draw[->, shorten >=5pt] (-.35,-.61) to (0,0);
\draw[->, shorten >=5pt] (0,0) to [in=120,out=180
] (-.35,-.61) ;
\draw[->, shorten >=5pt] (0,0) to[in=130,out=50, loop] (0,0);
\draw[->, shorten >=5pt] (-.35,-.61) to[in=250,out=170, loop] (-.35,-.61);
\draw[->, shorten >=5pt] (.35,-.61) to[in=10,out=-70, loop] (-.35,-.61);
\end{tikzpicture}
& \\
\end{longtable}

For a graph $E$, we denote its adjacency matrix by $A_E$. We establish what might have been expected: 

\begin{thm}\label{alergy1}
    Let \(E\) and \(F\) be small graphs.    Then, the following are equivalent:
    \begin{enumerate}[\upshape(1)]
       
        \item The adjacency matrices  \(A_E\) and \(A_F\) are shift equivalent;
        
        \item The adjacency matrices \(A_E\) and \(A_F\) are strong shift equivalent;

 \item The Leavitt path algebras \(L_\K(E)\) and \(L_\K(F)\) are graded Morita equivalent;
 
 \item The graph $C^*$-algebras \(C^*(E)\) and \(C^*(F)\) are equivariant Morita equivalent;
      
          \item There is an order-preserving 
        $\mathbb Z[x,x^{-1}]$-module isomorphism
$K_0^{\gr}(L_\K(E))\rightarrow K_0^{\gr}(L_\K(F))$; 
        \item The talented monoids \(T_E\) and \(T_F\) are \(\mathbb{Z}\)-isomorphic.
    \end{enumerate}
       
\end{thm}

This paper was motivated by the work~\cite{AALP}, where the following 20 year old open question is investigated:

\begin{ques} [The Classification Question for purely infinite simple unital Leavitt path algebras] \label{hfhfgd}Suppose $E$ and $F$ are graphs for which $L_\K(E)$ and $L_\K(F)$ are purely infinite simple unital. If $K_0(L_\K(E)) \cong  K_0(L_\K(F))$ via an isomorphism $\phi$ having $\phi([L_\K(E)]) = [L(F)]$, must $L_\K(E)$ and $L_\K(F)$ be isomorphic?
\end{ques}

In \cite{AALP} it was shown  that Question~\ref{hfhfgd} holds for the table of small graphs above. In \cite{AALP} these small graphs have been divided in seven groups along the $K$-theory data  $\big(K_0(L_\K(E)),[L(E)]\big)$. It was then shown that the Leavitt path algebras within each of these groups are indeed pairwise isomorphic, thus affirming Question~\ref{hfhfgd}.  The parallel graph $C^*$-algebra version of Question~\ref{hfhfgd} for small graphs, had already been answered in positive in early 1980's by Enomoto, Fujii and Watatani (see~\cite{enomoto, watatani}).

 The graded Grothendieck group  $K_0^{\gr}$ is a finer invariant than the Grothendieck group $K_0$ as the following commutative diagram shows: 
\begin{equation}\label{lol153332}\tag{GK-K}
\xymatrix@=15pt{
K^{\gr}_0(L_\K(E)) \ar[d]^{\psi} \ar[r]^{\phi} &
K^{\gr}_0(L_\K(E)) \ar[d]^{\psi} \ar[r]^{U}& K_0(L_\K(E))\ar[d] \ar[r] &0\\
K^{\gr}_0(L_\K(F)) \ar[r]^{\phi}  &
K^{\gr}_0(L_\K(F)) \ar[r]^{U}& K_0(L_\K(F)) \ar[r] &0,}
\end{equation}
Here the map $\phi$ sends $[P]\in K^{\gr}_0(L_\K(E))$ to $[P(1)]-[P]$, where $P(1)$ is the shift of the module by $1$ and $U$ is the homomorphism induced by the forgetful functor $\mbox{Gr-}A\rightarrow \mbox{Mod-}A$ (see~\S\ref{algpre}).  It follows that if $K^{\gr}_0(L_\K(E)) \cong K^{\gr}_0(L_\K(F))$ as $\mathbb Z[x,x^{-1}]$-modules, then 
$K_0(L_\K(E))\cong K_0(L_\K(F))$. This allows us to start with the seven groups of graphs in \cite{AALP}, and further refine the pairs along the notion of shift equivalence. 

The strategy to prove Theorem~\ref{alergy1} is simple: We calculate the graded Grothendieck groups $K_0^{\gr}$, of Leavitt path algebras associated to these graphs. If for two graphs these 
$K_0^{\gr}$-groups coincide, it is known that their adjacency matrices are shift equivalent. We then prove that these graphs can be obtained from the in-split/out-split graph moves. Williams' theorem then guarantees these matrices are strong shift equivalent. This also allows us to show that their Leavitt path algebras are graded Morita equivalent. 

Although the graphs are ``small'', they produce very interesting cases. We show that, except one case, shift equivalence among these graphs implies elementary shift equivalence.   We further investigate the second Graded Classification Conjecture and prove that it holds for this class of graphs.

\begin{thm} \label{alergy2}
    Let \(E\) and \(F\) be small graphs.  Then, the following are equivalent.
    \begin{enumerate}[\upshape(1)]
       
      \item The Leavitt path algebras \(L_\K(E)\) and \(L_\K(F)\) are graded isomorphic;
      
          \item There is a pointed order-preserving 
        $\mathbb Z[x,x^{-1}]$-module isomorphism
$K_0^{\gr}(L_\K(E))\rightarrow K_0^{\gr}(L_\K(F))$.
        
 \end{enumerate}
       
\end{thm}

The paper is organised as follows: In Section~\ref{preliminsec} we collect the concepts from the graded ring theory, graded $K$-theory and symbolic dynamics that we need in the paper. In Section~\ref{shiftequivalent}  we study all 34 graphs above and prove Theorems~\ref{alergy1} and \ref{alergy2}. 

\medskip 

{\bf Acknowledgements.} The authors acknowledge Australian Research Council Discovery Project Grant DP230103184. We would like to thank David Pask and Efren Ruiz for bringing to our attention some of the references related to this investigation.  We are grateful to S\o ren Eilers who provided us with the path of moves from $E^1_8$ to $E^1_{12}$, using a computer search, which allowed us to complete Theorem~\ref{alergy2}.

\section{Preliminary: Algebra, K-theory and dynamics}\label{preliminsec}

\subsection{Graded algebras}\label{algpre}
Throughout the paper rings will have identities and modules are unitary in the sense that the identity of the ring acts as the identity operator on the module. We refer the reader to \cite{H} for the theory of graded rings and the standard terminologies. We provide some definitions for the sake of notation.

Let $\Gamma$ be an abelian group.
A ring $R$ is \emph{$\Gamma$-graded} if $R$ decomposes as a direct sum $\bigoplus_{\gamma\in \Gamma} R_\gamma$, where each $R_\gamma$ is an additive subgroup of $R$ and $R_\gamma R_\eta \subseteq R_{\gamma+\eta}$, for all $\gamma,\eta\in \Gamma$. 

Let $R$ be a $\Gamma$-graded ring.
A right $R$-module $M$ is \emph{$\Gamma$-graded} if $M$ decomposes as a direct sum $\bigoplus_{\gamma \in \Gamma} M_\gamma$, where each $M_\gamma$ is an additive subgroup of $M$ and $M_\gamma R_\eta \subseteq M_{\gamma+\eta}$, for all $\gamma, \eta\in \Gamma$.
A graded left $R$-module is defined analogously. 
If $R$ and $S$ are $\Gamma$-graded rings, then an $R$-$S$ bimodule $M$ is $\Gamma$-graded if $M = \bigoplus_{\gamma\in \Gamma} M_\gamma$ is both a graded left $R$-module and a graded right $S$-module.
In particular, $R_\gamma M_\eta S_\zeta \subseteq M_{\gamma+\eta+\zeta}$, for all $\gamma,\eta,\zeta\in \Gamma$.
If $M$ is a $\Gamma$-graded $R$-module and $\gamma\in \Gamma$, then the $\gamma$-shifted graded right $R$-module is 
\[M(\gamma) := \bigoplus_{\eta\in \Gamma} M(\gamma)_\eta, \text{ where } M(\gamma)_\eta = M_{\gamma+\eta}, \text{ for all }\eta\in \Gamma.\]

For a $\Gamma$-graded ring $R$, there is a graded isomorphism of rings given by the regular representation $\eta : R \to \End_R(R)$ 
which sends $r\in R$ to $\eta_r$ where $\eta_r(x) = rx$, for all $x\in R$, see e.g. \cite[proof of Theorem 2.3.7]{H}.

We write $\Modd R$ for the category of unitary right $R$-modules and module homomorphisms. If $R$ is $\Gamma$-graded, we denote by $\Gr R$ the category of unitary graded right $R$-modules with morphisms preserving the grading. For $\Gamma$-graded rings $R$ and $S$, a functor $\mathcal F : \Gr R \to \Gr S$ is called \emph{graded} if $\mathcal F(M(\gamma))=\mathcal F(M)(\gamma)$ for all graded $R$-modules $M$ and all 
$\gamma \in \Gamma$.  For $\gamma\in\Gamma$, the \emph{shift functor}
\begin{equation}\label{shiftshift}
\mathcal{T}_{\gamma}: \Gr R\longrightarrow \Gr R,\quad M\mapsto M(\gamma)
\end{equation}
is an isomorphism with the property that $\mathcal{T}_{\gamma}\mathcal{T}_{\eta}=\mathcal{T}_{\gamma+\eta}$, for $\gamma,\eta\in\Gamma$.

Let $\Gamma$ be an abelian group, and let $R$ and $S$ be $\Gamma$-graded rings.
A $\Gamma$-graded $R$-$S$ bimodule $M$ determines a graded functor \[\mathcal F_M := -\otimes_R M : \mbox{Gr-}R \rightarrow \mbox{Gr-}S\] between categories of graded modules. We start with the following observation which is a generalisation of~\cite[Proposition~3.1]{adamefrenkevin}.

For an $n$-tuple $(\alpha_1,\dots,\alpha_n)$, where $\alpha_i\in \Gamma$, there is a natural grading on the matrix ring $\M_n(R)$, denoted by 
$\M_n(R)(\alpha_1,\dots,\alpha_n)$  (see~\cite[\S1.3]{H}), such that we have a $\Gamma$-graded ring isomorphism 
\[\End_R\big(R(-\alpha_1)\oplus R(-\alpha_2) \oplus \dots \oplus R(-\alpha_n)\big) \cong_{\gr} \M_n(R)(\alpha_1,\dots,\alpha_n).\]

\begin{prop}\label{prop:ring-lift}
Let $M$ be a graded $R$-$S$ bimodule such that $M\cong_{\gr} S(-\alpha_1) \oplus S(-\alpha_2)\oplus \cdots \oplus S(-\alpha_n)$, $\alpha_i \in \Gamma$,   as graded right $S$-modules.
Then, there exists a graded unital ring homomorphism  \[\xi : R \rightarrow \M_n(S)(\alpha_1, \alpha_2, \cdots, \alpha_n). \] 
Moreover, if $\mathcal F_M$ is a graded equivalence, then $\xi$  is a graded isomorphism of rings.
\end{prop}
\begin{proof}

We obtain a graded homomorphism of rings $\xi : R \to \M_n(S)(\alpha_1, \alpha_2, \cdots, \alpha_n)$ from the composition
\begin{multline}
\xi : R  \overset{\cong}{\longrightarrow}  \End_R(R) \overset{\mathcal F_M}{\longrightarrow} \End_S(R\otimes_R M) \overset{\cong}{\longrightarrow} \End_S(M) \\ 
\overset{\cong}{\longrightarrow} \End_S(S(-\alpha_1) \oplus S(-\alpha_2)\oplus \cdots \oplus S(-\alpha_n)) 
\overset{\cong}{\longrightarrow}  \M_n(S)(\alpha_1, \alpha_2, \cdots, \alpha_n). 
\end{multline}

Clearly, if $\mathcal F_M$ is an equivalence, then all the maps in the above sequence are graded isomorphism and thus $\xi$ is the graded isomorphism. 
\end{proof}

For a $\Gamma$-graded ring $R$, recall that $K_0^{\gr}(R)$ is the \emph{graded Grothendieck group}, which is the group completion of the  monoid consisting of the graded isomorphisms classes of finitely generated $\Gamma$-graded  projective $R$-modules equipped with the direct sum operation \cite[Section 3.1.2]{H}.
This group is a $\Zz[\Gamma]$-module with a structure induced from the $\Gamma$-module structure on the monoid. More specifically, if $\gamma \in \Gamma$, then for $[P]\in K_0^{\gr}(R)$ we have that $\gamma \cdot [P] := [P(\gamma)]$ defines the action of $\Gamma$.
The class of $R$ itself defines an order unit in $K_0^{\gr}(R)$. If $S$ is also a $\Gamma$-graded ring, we say that a homomorphism $K_0^{\gr}(R) \to K_0^{\gr}(S)$ is \emph{pointed} if it takes $[R]$ to $[S]$. The commutative monoid consisting of all isomorphism classes $[P]$, denoted by $K_0^{\gr}(R)^+$, is called the 
\emph{positive cone} of $K_0^{\gr}(R)$. 
Of particular interest to us is the case when $\Gamma = \Zz$ in which case $K_0^{\gr}(R)$ is a $\Zz[x,x^{-1}]$-module. Note that any unital graded homomorphism $\theta : R\to S$ of $\Zz$-graded rings $R$ and $S$, induces a pointed order-preserving $\mathbb Z[x,x^{-1}]$-module homomorphism $K_0^{\gr}(R)\to K_0^{\gr}(S)$, $[P]\mapsto [P\otimes_R S]$. 

Setting the grade group $\Gamma$ to be trivial, we obtain the Grothendieck group $K_0(R)$. For certain class of rings, such as right regular Noetherian $\mathbb Z$-graded rings or in our case Leavitt path algebras, we have the following relations between $K_0^{\gr}$ and $K_0$~\cite[\S6]{H}: 
\[ 
\begin{split}
\xymatrix{
K^{\gr}_0(R) \ar[rr]^{[P]\mapsto [P(1)]-[P]} && K^{\gr}_0(R) \ar[r]^{U} &  K_0(R) \ar[r] & 0.
}
\end{split}
\]

\subsection{Leavitt path algebras}\label{hfjduier87}

We refer the reader to \cite{lpabook} for the definition of Leavitt path algebras and the standard terminologies. We provide some definitions for the sake of notation.

Let $E = (E^0,E^1, r, s)$ be a row-finite directed graph and $\K$ a field. Throughout the paper, we assume $E$ is a finite graph, i.e., has finite number of vertices and edges. The \emph{Leavitt path algebra} $L_{\K}(E)$ is the universal $\K$-algebra generated by formal elements $\{ v \ | \ v \in E^0 \}$ and $\{ e,e^* \ | \ e\in E^1  \}$ subject to the relations
\begin{enumerate}
    \item $v^2 = v$, and $vw = 0$ if $v\neq w$ are in $E^0$
    \item $s(e) e = e r(e) = e, \ r(e)e^* = e^*s(e) = e^*$, for all $e\in E^0$
    \item $e^*e = r(e)$, for $e\in E^1$
    \item $v= \sum_{s(e) = v}ee^*$, for $v$ which is not a sink. 
\end{enumerate}

The main invariant used in this paper is the \emph{graded Grothendieck group} $K^{\gr}_0(L_\K(E))$, for the graph $E$ and its positive cone $K_0^{\gr}(L_\K(E))^+$. 
Using the deep work of George Bergman on realisation of conical monoids~\cite{lpabook}, the monoid $K_0^{\gr}(L_\K(E))^+$ can be described directly from the graph $E$. Namely, there is a $\mathbb Z$-monoid isomorphism 
 $K_0^{\gr}(L_\K(E))^+\cong T_E$, where $T_E$ is the so-called \emph{talented monoid} of the graph $E$:
 \begin{equation*}
T_E= \Big \langle \, v(i), v \in E^0, i \in \mathbb Z  \, \,  \Big \vert \, \, v(i)= \sum_{e\in s^{-1}(v)} r(e)(i+1) \ \Big \rangle. 
\end{equation*}
Here  the relations in the monoid are defined for those vertices that emit edges.  The action of $\mathbb Z$ on $T_E$ is defined by ${}^n v(i)=v(i+n)$, where $n\in \mathbb Z$.  Under this isomorphism, the class $[L_\K(E)]$ is sent to $1_E:= \sum_{v\in E^0} v$. 

Throughout the paper, we occasionally use the notation $E$ to represent a graph and its adjacency matrix $A_E$ interchangeably.

\subsection{Symbolic dynamics} \label{sec:symbolic-dynamics}
We refer the reader to \cite{Lind-Marcus2021} for an introduction to symbolic dynamics.

Two square nonnegative integer matrices $A$ and $B$ are called {\it elementary shift equivalent},  denoted  $A\sim_{\ES} B$, if there are nonnegative integer matrices $R$ and $S$ such that 
$$A=RS\,\, \, \text{and } \, \, B=SR.\\$$
Note that $R$ and $S$ need not be square matrices.  

The equivalence relation $\sim_{\Ss}$  generated by elementary shift equivalence on the set of  nonnegative integer square matrices (of any finite size) is called {\it strong shift equivalence}. That is,  $A\sim_{\Ss} B$ in case there is a sequence of square matrices $A_i$ and a sequence of elementary shift equivalence $(R_i,S_i)$, $1\leq i \leq k$, such that $$A=A_0\sim_{\ES} A_1 \sim_{\ES} A_2 \sim_{\ES}  \cdots  \sim_{\ES} A_k = B.$$
Here the intermediate matrices are related by $A_{i-1}=R_iS_i$ and $A_{i}=S_iR_i$.

Next we introduce the notion of shift equivalences.  The nonnegative  integer square matrices $A$ and $B$ are called {\it shift equivalent}, denoted $A\sim_S B$,  if there exist $k\ge 1$ and nonnegative integer matrices $R$ and $S$ such that 
\begin{gather}
A^k= RS, \ \ \, B^k=SR, \label{sdjhsjh}\\ 
AR=RB, \ \ \, SA=BS. \notag
\end{gather}

If $A$ and $B$ are strong shift equivalent via a sequence of elementary shift equivalence $(R_i,S_i)$, $1\leq i \leq k$, then $R=R_1\dots R_k$ and $S=S_k \dots S_1$ give (\ref{sdjhsjh}), i.e., $A$ and $B$ are shift equivalent~\cite[Proposition 7.3.2]{Lind-Marcus2021}.

Next we recall the definition of Krieger's dimension triple, which is a complete invariant for shift equivalence.  Let $A$ be an $r\times r$ matrix over $\mathbb N$.
The \emph{dimension group} of $A$ is the inductive limit of the stationary inductive system given by $A$ as a linear map on $\Nz^{r}$ (acting from right), i.e., 
\begin{equation}\label{thuluncht}
\mathbb Z^{r} \stackrel{A}{\longrightarrow} \mathbb Z^{r} \stackrel{A}{\longrightarrow}  \mathbb Z^{r} \stackrel{A}{\longrightarrow} \cdots.
\end{equation}

As a group, this is isomorphic to 
\begin{equation}\label{dimgroup1}
    G_A := (\Nz^{r}\times \mathbb N)/ \sim_{A}
\end{equation}
 where $(v,k) \sim_{A} (w,l)$  if $v A^{m-k}  = w A^{m-l} $, for some $m\in \Nz$. Let $[v,k]$ denote the equivalence class of $(v,k)$. Clearly $[ v A^n ,n+k]=[v,k]$, for any $n\in \mathbb N$. Then, it is not difficult to show that the direct limit of the system \eqref{thuluncht} is the abelian group consists of equivalence classes $[v,k]$, $v\in \mathbb Z^{r}$, $k \in \mathbb N$, with addition defined by
\[[v,k]+[w,k']=[v A^{k'}+w A^k,k+k'].\]
  
The \emph{positive cone} of $G_A$ is the commutative monoid \[G_A^+ := \big \{ [v,k] \in G_A \mid  v\in \mathbb N^{r}, k \in \mathbb N \big \}.\] 

The multiplication by $A$ defines an automorphism $\theta_A$ on $G_A$ that preserves the positive cone. This automorphism induces a $\mathbb Z[x,x^{-1}]$-module structure on $G_A$. 
The class of the unit $u_A := [(1,\ldots,1), 0]$ is an order unit in $G_A$. The triple $(G_A, G_A^+, \theta_A)$ is called \emph{Krieger's dimension triple of $A$}, or \emph{dimension group} of $A$ for short. 

The dimension group is a complete invariant:  two matrices $A$ and $B$ are shift equivalent if and only if there is a group isomorphism between the  dimension groups $G_A$ and $G_B$ 
which preserves the positive cones and the $\mathbb Z[x,x^{-1}]$-module structures on the groups  (\cite[Theorem 7.5.8]{Lind-Marcus2021}).

A shift equivalence between the matrices $A$ and $B$ implemented via $(R,S)$ (\ref{sdjhsjh}) determines the order-preserving $\mathbb Z[x,x^{-1}]$-module isomorphism 
\begin{align}\label{thuemas}
\phi_R:  G_A&\longrightarrow G_B\\
   [v, k] &\longmapsto [v R , k]. \notag
\end{align}

Finally, we relate the Leavitt path algebra of a graph via the graded Grothendieck group with the dimension group of its adjacency matrix.  For finite graphs with no sinks, there is an order-preserving group isomorphism
\begin{align}
\beta_E: K_0^{\gr} (L_\K(E)) &\longrightarrow G_{A_E}\label{ncjgloria}\\
 [L_\K(E)]&\longmapsto [(1,\dots,1),0] \notag
 \end{align}
 which preserves the $\mathbb Z[x,x^{-1}]$-module structures on the groups. Combining with (\ref{thuemas}), we then have the commutative diagram (see~\cite[\S2]{adamefrenkevin} and \cite{HazaratDyn2013}):
\begin{equation}\label{eq-kthy-shift}
\begin{tikzcd}
    K_0^{\gr}(L_\K(E)) \arrow[r, "\alpha"] \arrow[d, "\beta_E"]  & K_0^{\gr}(L_\K(F)) \arrow[d,"\beta_F"] \\
    G_{A_E} \arrow[r,"R"]   & G_{A_F} 
\end{tikzcd}
\end{equation}

It is easy to see that if an $r\times r$ integral matrix $A$ is invertible in $\GL_r(\mathbb Q)$, i.e., $\det(A) \not = 0$, then there is group monomorphism
\begin{align}\label{linkkgroup}
G_A=(\mathbb Z^r\times \mathbb N)/ \sim_{A} &\longrightarrow  \mathbb Z \left[ \frac{1}{\det(A)}\right] ^r\\
[v,k]&\longmapsto vA^{-k}. \notag
\end{align}

Combining~(\ref{linkkgroup}) with the reinterpretation of the dimension group of a matrix as the graded Grothendieck group  of the associated Leavitt path algebra (\ref{ncjgloria}), we have the following lemma which will be used throughout the paper.

\begin{lem} \label{DL} Let $E$ be a finite graph with no sinks and $A_E$ its associated adjacency matrix. If $\det(A_E) = \pm 1$, then as a group 
\[K_0^{\gr}(L_\K(E))\cong   \mathbb Z ^r.\]
\end{lem}
\begin{proof}
One can easily show that when $\det(A_E) = \pm 1$, then the monomorphism~(\ref{linkkgroup}) is surjective as well.
\end{proof}

For other cases, even when $\det(A_E) \not=0$, it is much harder to give a good presentation of the dimension group, as we shall even see  among these small graphs. 
When the determinant of an adjacency matrix is not $\pm 1$, we shall use an alternative description of the dimension group to aid in its computation. This description is taken from~\cite[Chapter~7]{Lind-Marcus2021}. 

Let $A$ be an $r \times r$ integral matrix. Consider the abelian group 
\begin{equation}\label{hgdgdgd}
 \Delta_A=\Big \{ v \in \bigcap_{k=1}^\infty \mathbb{Q}^rA^k \, | \, vA^k \in \mathbb{Z}^r \text{, for some } k \geq 0\Big\},
 \end{equation}
with positive cone 
\begin{equation}\label{posiposi8}
\Delta_A^+ := \big \{ v\in \Delta_A \mid v A^l\in \mathbb N ^{r} \text{, for some } l\in \mathbb N \big \}.
\end{equation}
The multiplication $v\mapsto vA$ defines an automorphism $\delta_A$ on $\Delta_A$ that preserves the positive cone.  The automorphism $\delta_A$ naturally makes $\Delta_A$ into a $\mathbb Z [x,x^{-1}]$-module.  

 \begin{cor}\label{alternativeformdel} 
 Let $A$ be an $r \times r$ invertible (over $\mathbb{Q}$) integral matrix. Then \[\Delta_A= \bigcup_{k=0}^\infty \mathbb{Z}^r A^{-k}.\]
    \end{cor}
    \begin{proof}
    First note that if the matrix $A$  is invertible over $\mathbb{Q}$, then $\Delta_A$ takes the simpler form:
\begin{equation*}
 \Delta_A=\Big \{ v \in \mathbb{Q}^r \, | \, vA^k \in \mathbb{Z}^r \text{, for some } k \geq 0\Big\}.
\end{equation*} 
The statement of the Corollary is now immediate. 
    \end{proof}

The two pictures of the dimension data (\ref{dimgroup1}) and (\ref{hgdgdgd}) are equivalent. Indeed, a concrete order-preserving $\Zz[x,x^{-1}]$-module isomorphism is given by 
\begin{align*}
\psi_A\colon \Delta_A &\longrightarrow G_A\\
v &\longmapsto [v A^l, l], 
\end{align*}
where $l\in \mathbb N$ is such that $vA^l\in \mathbb Z ^{r}$.

The following proposition can be considered as a generalisation of \cite[Theorem~3.15]{AP}. It will be used in Section~\ref{shiftequivalent} to show that some of the Leavitt path algebras of Table~\ref{tablegray} are graded isomorphic.

\begin{prop}\label{lateaddition}
Let $E$ and $F$ be finite graphs with no sinks which are strongly shift equivalent by the sequence of elementary shift equivalences $(R_i,S_i)$, $1\leq i \leq l$. If the isomorphism 
\[\phi_{R_1R_2\dots R_l}: K_0^{\gr}(L_\K(E)) \rightarrow K_0^{\gr}(L_\K(F)),\] induced by $R_1R_2\dots R_l$ sends $[L_\K(E)]$ to $\sum_{i=1}^{k}[L_\K(F)(n_i)]$, for some $k\in \mathbb N$ and $n_i\in \mathbb Z$, then 
\begin{equation}\label{hfuthf22}
L_\K(E) \cong_{\gr} \M_n(L_\K(F))(n_1, n_2, \cdots, n_k).
\end{equation}
\end{prop}
\begin{proof}
Since $E$ and $F$ are strong shift equivalent via $(R_i,S_i)$, there is a graded Morita equivalence
\[\mathcal F : \mbox{Gr-}R \rightarrow \mbox{Gr-}S,\] 
such that $K_0^{\gr}(\mathcal F)=\phi_{R_1R_2\dots R_l}$. 

Therefore in $K_0^{\gr}(L_\K(F))$ we have 
\begin{equation}\label{nopower}
\sum_{i=1}^{k}[L_\K(F)(n_i)]= \phi_{R_1R_2\dots R_l}([L_\K(E)])=K_0^{\gr}(\mathcal F)([L_\K(E)]) .
\end{equation}
Let $\mathcal F(L_\K(E))=M$. Then $M$ has natural structure $L_\K(E)$-$L_\K(F)$-bimodule (see~\cite[Theorem~2.3.5]{H}). Then from~(\ref{nopower}) it follows that $M\cong L_\K(F)(n_1) \oplus \cdots \oplus L_\K(F)(n_k)$ as graded right $L_\K(F)$-modules. Now Proposition~\ref{prop:ring-lift} gives the graded ring isomorphism (\ref{hfuthf22}). 
\end{proof}

One of the aims of the paper is to show that among small graphs, shift equivalence implies strong shift equivalence. By (\ref{eq-kthy-shift}) if two matrices are shift equivalent, then the graded Grothendieck groups of the associated Leavitt path algebras are $\mathbb Z[x,x^{-1}]$-module isomorphic. Then (\ref{lol153332}) implies that their $K_0$'s are isomorphic. Further the following theorem from~\cite[Proposition 7.3.7]{Lind-Marcus2021} allows us to further seperate the ``small'' graphs which are supposed to be shift equivalent into separate groups. 

\begin{prop}[\cite{Lind-Marcus2021}] \label{marcusbks}Let $A$ and $B$ be integral matrices which are  shift equivalent. Then $A$ and $B$ have the same set of nonzero eigenvalues.
\end{prop}

The following table groups the graphs with the same (non-graded and graded) $K$-theory  of their Leavitt path algebras and the same Perron-Frobenius eigenvalue of their adjacency matrices. We will show that the adjacency matrices of the graphs appearing in each group are indeed strong shift equivalent and their associated Leavitt path algebras are graded Morita equivalent, thus affirming Graded Morita Classification Conjecture for this class of graphs. 
Within each group  we also show that Graded Classification Conjecture holds. Interesting examples appear here: For instance, as the table shows, the algebras $L_\K(E^6_1)$ and $L_\K(E^7_1)$ have the same graded Grothendieck groups and the same Perron-Frobenius eigenvalues, but their $K_0$ are different. This implies that their graded Grothendieck groups are not \emph{module} isomorphic and hence they are not shift equivalent. 

\begin{longtable}{lllll}
\caption{Small graphs and their $K$-theory data}\label{tablegray}\\

\rowcolor{gray}

Graphs                                & $K_0(L_\K(E))$                                                 & $K^{\gr}_0(L_\K(E))$                          &\hspace{5pt}& Perron-Frobenius eigenvalue \\
$E^1_1$                               & $\{ \bar{0}\}$                                              & $\mathbb{Z}^3$                                                                                                        && $1.46667$                    \\

\rowcolor{LGray}
$E^1_2$                               & $\{ \bar{0}\}$                                              & $\mathbb{Z}^3$                                                                                                         && $1.75448$                    \\
$E^1_3$                               & $\{ \bar{0}\}$                                              & $\mathbb{Z}\left[ \frac{1}{2} \right]  \begin{bmatrix} 1 & 1 & 1 \end{bmatrix} \oplus 0 \times \mathbb{Z}^2$                                                               && $2$                          \\

\rowcolor{LGray}
$E^1_4$                               & $\{ \bar{0}\}$                                              & $\mathbb{Z}^3$                                                                                                         && $1.46557$                    \\
$E^1_5, E^1_{13}, E^1_{14}, E^1_{16}$ & $\{ \bar{0}\}$                                              & $\mathbb{Z}^2$                                                                                                        && $\frac{1}{2}(1 + \sqrt{5})$  \\

\rowcolor{LGray}
$E^1_6, E^1_{7}, E^1_{15}, E^1_{17}$    & $\{ \bar{0}\}$                                              & $\mathbb{Z}\left[ \frac{1}{2} \right] $                                                                              && $2$                          \\
$E^1_8,E^1_{12}$                      & $\{ \bar{0}\}$                                              & $\mathbb{Z}^3$                                                                                                        && $2.24698$                    \\
\rowcolor{LGray}
$E^1_9$                               & $\{ \bar{0}\}$                                              & $\mathbb{Z}^2$                                                                                                         && $\frac{1}{2}(3 + \sqrt{5})$  \\
$E^1_{10}$                            & $\{ \bar{0}\}$                                              & $\left( \mathbb{Z}\left[ \frac{1}{2}\right] \right) ^2$                                                               && $\sqrt{2}$                   \\

\rowcolor{LGray}
$E^1_{11}$                            & $\{ \bar{0}\}$                                              & $\mathbb{Z}^3$                                                                                                        && $1.80194$                    \\
$E^1_{18}$                            & $\{ \bar{0}\}$                                              & $\mathbb{Z}^3$                                                                                                         && $2.32472$                    \\

\rowcolor{LGray}
$E^2_1$                               & $\mathbb{Z}/2\mathbb{Z}$                                    & $\mathbb{Z}^3$                                                                                                         && $2.20557$                    \\
$E^2_2$                               & $\mathbb{Z}/2\mathbb{Z}$                                    & $\mathbb{Z}^3$                                                                                                         && $\frac{1}{2}(1 + \sqrt{5})$  \\

\rowcolor{LGray}
$E^2_3, E^2_4, E^3_4$                 & $\mathbb{Z}/ 2 \mathbb{Z}$                                   & $\mathbb{Z}\left[ \frac{1}{2} \right] \left[  \phantom{ [ } \begin{matrix} \hspace{-4pt}1 & 1 \\ \end{matrix}  \right] \phantom{ ] }  \hspace{-4.5pt} \oplus 0 \times \mathbb{Z}$                                                                && $2$                          \\
$E^2_5,E^3_1$                         & $\mathbb{Z}/2\mathbb{Z}$                                    & $\mathbb{Z}^3$                                                                                                         && $1.83929$                    \\

\rowcolor{LGray}
$E^2_6, E^3_2$                               & $\mathbb{Z}/2\mathbb{Z}$                                    & $\mathbb{Z}^2$                                                                                                       && $1+\sqrt{2}$                   \\
$E^3_3$                               & $\mathbb{Z}/ 2 \mathbb{Z}$                                   & $ \mathbb{Z} \left[ \frac{1}{3}\right] $       && $3$                          \\
\rowcolor{LGray}

\rowcolor{LGray}

$E^4_1$                               & $\mathbb{Z}/3\mathbb{Z}$                                    & $\mathbb{Z}^3$                                                                                                        && $2.1479$                     \\
$E^4_2$                               & $\mathbb{Z}/ 3 \mathbb{Z}$                                   & $\left( \mathbb{Z} \left[ \frac{1}{2}\right] \right)^2$                                                               && $1+\sqrt{3}$                  \\

\rowcolor{LGray}

$E^5_1$                               & $\mathbb{Z}/4 \mathbb{Z}$                                   & $\mathbb{Z}^3$                                                                                                         && $1+\sqrt{2}$                  \\

$E^6_1$                               & $\mathbb{Z}/ 2 \mathbb{Z} \oplus \mathbb{Z} / 2 \mathbb{Z}$ & $\mathbb{Z}\left[ \frac{1}{2} \right] \left[  \phantom{[} \begin{matrix} \hspace{-4pt}1 & 1 & 1 \\ \end{matrix}  \right] \phantom{]}  \hspace{-4.5pt} \oplus 0 \times \mathbb{Z}^2$                                                                 && $2$                          \\

\rowcolor{LGray}

$E^7_1$                               & $ \mathbb{Z}$                                               & $\mathbb{Z}\left[ \frac{1}{2} \right] \left[  \phantom{[} \begin{matrix} \hspace{-4pt}1 & 1 & 1 \\ \end{matrix}  \right] \phantom{]}  \hspace{-4.5pt} \oplus 0 \times \mathbb{Z}^2$                                                     && $2$                          \\

$E^7_2$                               & $ \mathbb{Z}$                                               & $\mathbb{Z}^3$                                                                                                        && $1+\sqrt{2}$                

\end{longtable}

Although for a graph $E$ with $\det(A_E)\not = 0$ it is rather easy to compute its graded Grothendieck group (Lemma~\ref{DL} and Corollary~\ref{alternativeformdel}), it is much harder to compute its positive cone. Even for the graph $E^1_1$, although $\det(A_{E^1_1})=1$ and thus $K_0^{\gr}(L_\K(E^1_1))=\mathbb Z^3$, its positive cone is quite complicated (see~\S\ref{shiftequivalent}). We will use Lemma~\ref{LM} below to compute  the positive cones (i.e., the talented monoid) of some of the small graphs. 

Recall that a {\it nonnegative matrix}  is a matrix where all its entries are nonnegative.  A nonnegative matrix $A$ is {\it irreducible} if for each ordered
pair of indices $i,j$ , there exists some $n> 0$ such that $(A^n)_{ij} > 0$. For the adjacency matrix of a graph, this is the same as the graph being strongly connected.
A strongly connected graph (or its adjacency matrix) has {\it period} $P$, where $P$ is the greatest common divisor of the length of its cycles. A  graph (or its adjacency matrix) is {\it aperiodic} if its period $P$ is $1$. A matrix is {\it primitive} if it is both aperiodic and irreducible. By~\cite[Theorem 4.5.8]{Lind-Marcus2021} a matrix $A$ is primitive if and only if there is a $k>0$ such that all entires of $A^k$ is positive. This fact will be used in the paper. 

Although Lemma~\ref{DL} gives that the graded Grothendieck groups in the case of invertible adjacency matrices with determinant $\pm1$, are just certain copies of $\mathbb Z$, 
the positive cones are much more involved. One can argue that it is not the graded Grothendieck group, buts its talented monoid that captures all the information. We will use the following deep statement in linear algebra to determine the positive cone in the case of primitive matrices.

\begin{lem}[\cite{Lind-Marcus2021}, Lemma 7.3.8] \label{LM} Let $A$ be an $r\times r$ primitive matrix and let $\bf{z}$ be a right Perron-Frobenius eigenvector for $A$. If 
$\mathbf{u}\in \mathbb Q^r$ is a row vector, then ${\mathbf u}A^k$ is eventually positive if and only if $\mathbf{u} \cdot \mathbf{z} >0$
\end{lem}

We can now determine the positive cone of an invertible primitive graph. 

\begin{cor} \label{halfspace} Let $E$ be a finite graph with no sinks with $|E^0|=r$. If $E$ is primitive and $\det(A_E) \not =0$, then 
\begin{equation}\label{perronformula}
K_0^{\gr}(L_\K(E))^+ = \left\{ {\bf u} \in  K_0^{\gr}(L_\K(E)) \, \Big \vert \,  {\mathbf u} \cdot {\bf z} >  0 \right\} \bigcup \, \{0\},
\end{equation}
 where $\bf{z}$ is a right Perron-Frobenius eigenvector for $A_E$.
\end{cor}

\begin{proof}
First, using the $\Delta_{A_E}$ model for the graded Grothendieck group, Corollary~\ref{alternativeformdel} implies  $K_0^{\gr}(L_\K(E)) \subseteq \mathbb Q^r$. Now by (\ref{posiposi8}), we shall collect vectors $\mathbf{u} \in K_0^{\gr}(L_\K(E))$ such that $\mathbf{u} A^k \in \mathbb N^r$, for some $k$.  If $\mathbf{u}A^k \in {\mathbb N^+ }^r$ (i.e., $\mathbf{u}$ is  eventually positive), then by  Lemma~\ref{LM} $\mathbf{u} \cdot \mathbf{z} >0$. On the other hand if $\mathbf{u} A^k =0$, since $A$ is invertible,  $\mathbf{u}=0$. The only case remained is if $\mathbf{u} A^k\in \mathbb N^r$, i.e., a vector with some (but not all) entries $0$.  Since $A$ is primitive, there is an $l\in \mathbb N$ such that all the entries of $A^l$ are positive. Then $\mathbf{u} A^{k+l} \in {\mathbb N^+ }^r$, and so 
$\mathbf{u} \cdot \mathbf{z} >0$. The reverse inclusion of (\ref{perronformula}) is immediate. 
 \end{proof}
 
 \begin{rmk}
 We remark that (\ref{perronformula}) is not in general equal to the set $$\left\{ {\bf u} \in  K_0^{\gr}(L_\K(E)) \, \Big \vert \,  {\mathbf u} \cdot {\bf z} \geq  0 \right\}.$$ For example, one can show that in such cases as $E^1_5, E^1_9, E^1_{13}, E^1_{14}, E^1_{16}$ the equality indeed hold, but for the graph  $E^7_1$ this is not the case. 
 \end{rmk}

Throughout the paper we will use the Williams' in-split and out-split moves as well as source elimination for graphs. It is now established that in-split, out-split and source eliminations preserve the graded Morita theory on the level of Leavitt path algebras. We refer the reader to~\cite{AALP2, AP, Lind-Marcus2021, HazaratDyn2013} for these concepts and the results. However in one case we need a generalised version of these now classical constructions due to Eilers and Ruiz~\cite{ER} that we include here for the convenience  of the reader. 

\begin{defn}[Move $(I-)$: Generalised in-splitting] Let $E = (E^0,E^1, s_E, r_E)$ be a graph and let
$w\in E^0$ be a regular vertex. Partition $r^{-1}(w)$ as a finite disjoint union of (possibly empty)
subsets,
\[r^{-1}(w) = \mathcal{E}_1 \sqcup \mathcal{E}_2 \sqcup \cdots \sqcup \mathcal{E}_n.\]
Let $E_I =(E^0_I,E^1_I, s_{E_I}, r_{E_I} )$ be the graph defined by
\[E^0_I= \big \{v^1 \, \mid  \, v \in E^0\backslash \{ w \} \big \}  \sqcup \big \{w^1, w^2, \ldots , w^n \big \} \]
\[E^1_I = \big \{e^1 \, \mid  \, e \in E^1, \, s_E(e) \neq w \big \} \sqcup \big \{e^1, e^2,\ldots, e^n \, \mid  \, e \in E^1, s_E(e) = w\big \} \]
\[ s_{E_I} (e^i) = \begin{cases}
    s_E(e)^1 & e \in E^1, \, s_E(e) \neq w\\
    w^i & e \in E^1, \, s_E(e) = w\\
\end{cases} \]

\[ r_{E_I} (e^i) = \begin{cases}
    r_E(e)^1 & e \in E^1, \, r_E(e) \neq w\\
    w^j & e \in E^1, \, r_E(e) = w, \, e \in \mathcal{E}_j\\
\end{cases} \]

We say $E_I$ is formed by performing move $(I-)$ to $E$.
\end{defn}

\begin{defn}[Move $(I+)$: Unital in-splitting] \label{unitalinspliII}
The graphs $E$ and $F$ are said to be move
$(I+)$ equivalent if there exists a graph $G$ and a regular vertex $w \in G^0$ such that $E$ is the
result of an $(I-)$ move applied to $G$ via a partition of $r_G^{-1}(w)$ using $n$ sets and $F$ is the
result of an $(I-)$ move applied to $G$ via a partition of $r_G^{-1}(w)$ using $n$ sets.
\end{defn}

In \cite{ER} it was shown that unital in-splitting preserves the class of graph $C^*$-algebras up to gauge actions. A similar proof shows that unital in-splitting preserves the graded isomorphism class of the Leavitt path algebras.

\begin{example}\label{gettingclose343}
Consider the following graphs $A$ and $B$:
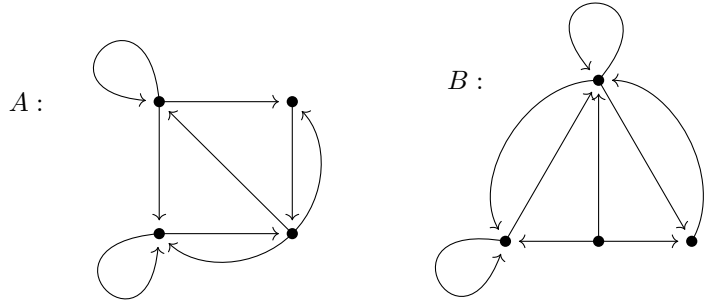
\begin{figure}[H]
\centering
\begin{tikzpicture}[scale=3.5]
\draw (-0.5, 0) node{$A:$};
\fill (0,0)  circle[radius=.6pt];
\fill (.5,0)  circle[radius=.6pt];
\fill (0,-.5)  circle[radius=.6pt];
\fill (.5,-.5)  circle[radius=.6pt];
\draw[->, shorten >=5pt] (0,0) to (.5,0);
\draw[->, shorten >=5pt] (0,0) to (0,-.5);
\draw[->, shorten >=5pt] (0,-.5) to (.5,-.5);
\draw[->, shorten >=5pt] (.5,0) to (.5,-.5);
\draw[->, shorten >=5pt] (.5,-.5) to (0,0);
\draw[->, shorten >=5pt] (.5,-.5) to[in=-45, out=-135] (0,-.5);
\draw[->, shorten >=5pt] (.5,-.5) to[in=-45, out=45] (.5,0);
\draw[->, shorten >=5pt] (0,0) to[in=175,out=95, loop] (0,0);
\draw[->, shorten >=5pt] (0,-.5) to[in=265,out=185, loop] (0,-.5);
\end{tikzpicture}
\hspace{10pt}
\begin{tikzpicture}[scale=3.5]
\draw (-0.5, 0) node{$B:$};
\fill (0,0)  circle[radius=.6pt];
\fill (-.35,-.61)  circle[radius=.6pt];
\fill (.35,-.61)  circle[radius=.6pt];
\fill (0,-.61)  circle[radius=.6pt];
\draw[->, shorten >=5pt] (0,0) to (.35,-.61);
\draw[->, shorten >=5pt] (.35,-.61) to [in=0,out=60] (0,0);
\draw[->, shorten >=5pt] (-.35,-.61) to (0,0);
\draw[->, shorten >=5pt] (0,0) to [in=120,out=180
] (-.35,-.61) ;
\draw[->, shorten >=5pt] (0,0) to[in=130,out=50, loop] (0,0);
\draw[->, shorten >=5pt] (-.35,-.61) to[in=250,out=170, loop] (-.35,-.61);
\draw[->, shorten >=5pt] (0,-.61) to (-.35,-.61);
\draw[->, shorten >=5pt] (0,-.61) to (.35,-.61);
\draw[->, shorten >=5pt] (0,-.61) to (0,0);
\end{tikzpicture}
\end{figure}
The Leavitt path algebras associated with these graphs are graded isomorphic, as they are both generalised in-splits of the graph below:
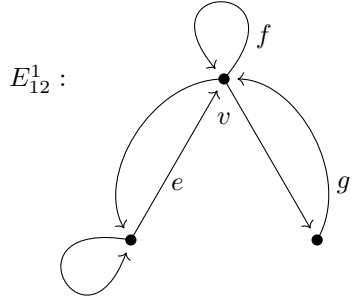
\begin{figure}[H]
\begin{tikzpicture}[scale=3.5]
\fill (0,0)  circle[radius=.6pt];
\fill (-.35,-.61)  circle[radius=.6pt];
\fill (.35,-.61)  circle[radius=.6pt];
\draw (0,-.15) node{$v$};
\draw (-.70, 0) node{$E^1_{12}:$};
\draw (.15,.16) node{$f$};
\draw (-.175,-.4) node{$e$};
\draw (.45,-.4) node{$g$};
\draw[->, shorten >=5pt] (0,0) to (.35,-.61);
\draw[->, shorten >=5pt] (.35,-.61) to [in=0,out=60] (0,0);
\draw[->, shorten >=5pt] (-.35,-.61) to (0,0);
\draw[->, shorten >=5pt] (0,0) to [in=120,out=180
] (-.35,-.61) ;
\draw[->, shorten >=5pt] (0,0) to[in=130,out=50, loop] (0,0);
\draw[->, shorten >=5pt] (-.35,-.61) to[in=250,out=170, loop] (-.35,-.61);
\end{tikzpicture}
\end{figure}
The graph $A$ is obtained from $E^1_{12}$ by partitioning $s^{-1}(v)$   to two sets $\{ f \}$  and $\{ e,g  \}$, whereas  $B$ is obtained from the partition of $s^{-1}(v)$ to an empty set and $\{ f, e, g  \}$. 
\end{example}

\section{Williams' conjecture for small  graphs}\label{shiftequivalent}

We start by calculating the graded Grothendieck groups, their positive cones (i.e, the talented monoids) and the action of $\mathbb Z$ on them, for each of the graphs above. We then show that if for two graphs these $K$-groups coincide, i.e., when the adjacency matrices are shift equivalent, then we are able to transform one graph to another with a sequence of in-split and out-split moves. That shows that they are indeed  strongly shift equivalent and consequently the Leavitt path algebras are graded Morita equivalent.  Despite being ``small graphs'', we will see that there are cases among these graphs that shift equivalence does not imply the elementary shift equivalence and one requires a chain of in/out split to transform the graphs. We follow the grouping of Table~\ref{tablegray}. 

{\bf 1.} Consider the graph $E_1^1$ below with the adjacency matrix $A_{E_1^1}= \begin{bmatrix} 0 & 1 & 0 \\ 0 & 0 & 1 \\ 1 & 0 & 1 \\ \end{bmatrix}.$

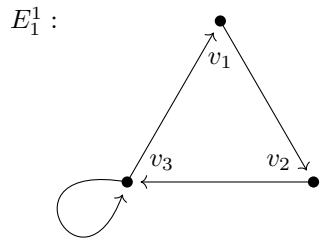
 \begin{figure}[H]
\centering
\begin{tikzpicture}[scale=3.5]
\fill (0,0)  circle[radius=.6pt];
\fill (-.35,-.61)  circle[radius=.6pt];
\fill (.35,-.61)  circle[radius=.6pt];
\draw (-.70, 0) node{$E_1^1:$};
\draw (0, -0.15) node{$v_1$};
\draw (-.35+.13,-.535) node{$v_3$};
\draw (.35-.13,-.535) node{$v_2$};
\draw[->, shorten >=5pt] (0,0) to (.35,-.61);
\draw[->, shorten >=5pt] (.35,-.61) to (-.35,-.61);
\draw[->, shorten >=5pt] (-.35,-.61) to (0,0);
\draw[->, shorten >=5pt] (-.35,-.61) to[in=250,out=170, loop] (-.35,-.61);
\end{tikzpicture}
\end{figure}

Since $\det(E^1_1) = 1 $ and $E^1_1$ is primitive,  Lemma~\ref{DL} and Corollary~\ref{halfspace} give: 
 \begin{align*}    
 K_0^{\gr}(L_\K(E^1_1)) &= \mathbb{Z}^3, \\
    K_0^{\gr}(L_\K(E^1_1))^+ &=\left\{ {\bf u} \in \mathbb{Z}^3 \, | \, {\bf u} \cdot {\bf z} > 0 \right\} \cup \, \{0\},  \text{ where  } \bf{z} \text{ is } \left[
\begin{array}{c}
 \frac{1}{3} \left(-2+\sqrt[3]{\frac{1}{2}(29-3 \sqrt{93}}+\sqrt[3]{\frac{1}{2} \left(29+3
   \sqrt{93}\right)}\right) \\ \sqrt[3]{\frac{1}{18}
   \left(9+\sqrt{93}\right)}-\sqrt[3]{\frac{2}{3
   \left(9+\sqrt{93}\right)}} \\ 1 \\
\end{array}
\right],\\
{}^1(a,b,c) & = (a,b,c) A_{E_1^1}=(c, a, b+c).
\end{align*}

{\bf 2.} Consider the graph $E_2^1$ below with the adjacency matrix $A_{E_2^1}= \begin{bmatrix} 1 & 1 & 0 \\ 0 & 0 & 1 \\ 1 & 0 & 1 \\ \end{bmatrix}$.
 
 \begin{figure}[H]
\centering
\begin{tikzpicture}[scale=3.5]
\fill (0,0)  circle[radius=.6pt];
\fill (-.35,-.61)  circle[radius=.6pt];
\fill (.35,-.61)  circle[radius=.6pt];
\draw (0,-.15) node{$v_1$};
\draw (-.70, 0) node{$E^1_2:$};
\draw (-.35+.13,-.535) node{$v_3$};
\draw (.35-.13,-.535) node{$v_2$};
\draw[->, shorten >=5pt] (0,0) to (.35,-.61);
\draw[->, shorten >=5pt] (.35,-.61) to (-.35,-.61);
\draw[->, shorten >=5pt] (-.35,-.61) to (0,0);
\draw[->, shorten >=5pt] (0,0) to[in=130,out=50, loop] (0,0);
\draw[->, shorten >=5pt] (-.35,-.61) to[in=250,out=170, loop] (-.35,-.61);
\end{tikzpicture}
\end{figure}

Since $\det(E^1_2) = 1 $ and $E^1_2$ is primitive,  Lemma~\ref{DL} and Corollary~\ref{halfspace} give: 
\begin{align*}    
K_0^{\gr}(L_\K(E^1_2)) &= \mathbb{Z}^3, \\
K_0^{\gr}(L_\K(E^1_2))^+ & =\left\{ {\bf u} \in  \mathbb{Z} ^3 \, | \, {\bf u} \cdot {\bf z} > 0 \right\}  \cup \, \{0\}, \text{ where }\bf{z} \text{ is } \left[
\begin{array}{c}
 \frac{1}{3} \left(-1+\sqrt[3]{\frac{1}{2} \left(25-3
   \sqrt{69}\right)}+\sqrt[3]{\frac{1}{2} \left(25+3
   \sqrt{69}\right)}\right) \\ 
   \frac{1}{3} \left(1-5
   \sqrt[3]{\frac{2}{11+3 \sqrt{69}}}+\sqrt[3]{\frac{1}{2}
   \left(11+3 \sqrt{69}\right)}\right) \\ 1 \\
\end{array}
\right],\\
{}^1(a,b,c) &=(a,b,c) A_{E_2^1}=  (a+c, a, b+c). 
 \end{align*}    

{\bf 3.} Consider the graph $E^1_3$ below with the adjacency matrix $A_{E^1_3}=\begin{bmatrix} 1 & 1 & 0 \\ 0 & 1 & 1 \\ 1 & 0 & 1 \\ \end{bmatrix}$.

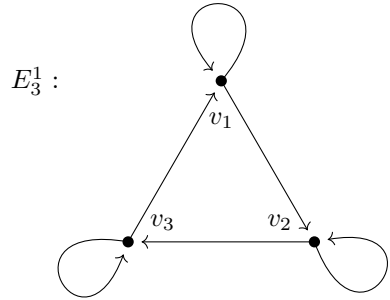
\begin{figure}[H]
\centering
\begin{tikzpicture}[scale=3.5]
\fill (0,0)  circle[radius=.6pt];
\fill (-.35,-.61)  circle[radius=.6pt];
\fill (.35,-.61)  circle[radius=.6pt];
\draw (0,-.15) node{$v_1$};
\draw (-.70, 0) node{$E^1_3:$};
\draw (-.35+.13,-.535) node{$v_3$};
\draw (.35-.13,-.535) node{$v_2$};
\draw[->, shorten >=5pt] (0,0) to (.35,-.61);
\draw[->, shorten >=5pt] (.35,-.61) to (-.35,-.61);
\draw[->, shorten >=5pt] (-.35,-.61) to (0,0);
\draw[->, shorten >=5pt] (0,0) to[in=130,out=50, loop] (0,0);
\draw[->, shorten >=5pt] (-.35,-.61) to[in=250,out=170, loop] (-.35,-.61);
\draw[->, shorten >=5pt] (.35,-.61) to[in=10,out=-70, loop] (-.35,-.61);
\end{tikzpicture}
\end{figure}

We have that $\det(E^1_3) = 2 $. In order to calculate $K_0^{\gr}(L_\K(E^1_3))$, we will use the alternative description of dimension group from Corollary~\ref{alternativeformdel}. 

\begin{prop} \label{profidsh6} We have 
    $\Delta_{A_{E^1_3}} = \mathbb{Z}\left[ \frac{1}{2} \right] \begin{bmatrix} 1 & 1 & 1 \end{bmatrix} \oplus 0 \times \mathbb{Z}^2$.
\end{prop}

\begin{proof}
First note that
$A_{E^1_3}^{-1}=\frac{1}{2}\begin{bmatrix}
 1 & -1 & 1 \\
 1 & 1 & -1 \\
 -1 & 1 & 1 
\end{bmatrix}.$
Consider $$v=\frac{a}{2^k}\begin{bmatrix} 1 & 1 & 1 \end{bmatrix} + \begin{bmatrix} 0 & b & c \end{bmatrix} \in \mathbb{Z}\left[ \frac{1}{2} \right] \begin{bmatrix} 1 & 1 & 1 \end{bmatrix} \oplus 0 \times \mathbb{Z}^2,$$ where $a,b,c \in \mathbb{Z}$ and $k\in \mathbb N$. Since  
$$\begin{bmatrix} 1 & 1 & 1 \end{bmatrix} \begin{bmatrix}
 1 & -1 & 1 \\
 1 & 1 & -1 \\
 -1 & 1 & 1 
\end{bmatrix} = \begin{bmatrix} 1 & 1 & 1 \end{bmatrix},$$
we have 
\begin{align*} 
&\frac{a}{2^k} \begin{bmatrix} 1 & 1 & 1 \end{bmatrix} A_{E^1_3}^k  
= \frac{a}{2^k} \begin{bmatrix} 1 & 1 & 1 \end{bmatrix} \begin{bmatrix}
 1 & -1 & 1 \\
 1 & 1 & -1 \\
 -1 & 1 & 1 
\end{bmatrix}^k A_{E^1_3}^k \\ 
= & a \begin{bmatrix} 1 & 1 & 1 \end{bmatrix} \left( \frac{1}{2} \begin{bmatrix}
 1 & -1 & 1 \\
 1 & 1 & -1 \\
 -1 & 1 & 1 
\end{bmatrix} \right)^k A_{E^1_3}^k 
= a \begin{bmatrix} 1 & 1 & 1 \end{bmatrix} A_{E^1_3}^{-k} A_{E^1_3}^k \\ 
=& a \begin{bmatrix} 1 & 1 & 1 \end{bmatrix}. 
\end{align*}
Thus $\frac{a}{2^k}\begin{bmatrix} 1 & 1 & 1 \end{bmatrix}$ is an element of $\Delta_{A_{E^1_3}}$. The vector $ \begin{bmatrix} 0 & b & c \end{bmatrix}$ is also an element of $\Delta_{A_{E^1_3}}$ as $\mathbb{Z}^3 \subseteq \Delta_{A_{E^1_3}}$.\\
Hence  
$$ \mathbb{Z} \left[ \frac{1}{2} \right]\begin{bmatrix} 1 & 1 & 1 \end{bmatrix} \oplus 0 \times \mathbb{Z}^2 \subseteq \Delta_{A_{E^1_3}}.$$

For the reverse inclusion, let $v \in \Delta_{A_{E^1_3}}$. Then by Corollary~\ref{alternativeformdel}, 
$v \in \mathbb{Z}^3A_{E^1_3}^{-k}$, for some $k \geq 0$. 
We will prove by induction  that $v \in \mathbb{Z} \left[ \frac{1}{2} \right] \begin{bmatrix} 1 & 1 & 1 \end{bmatrix} \oplus 0 \times \mathbb{Z}^2$. 

When $k=0$, $v \in \mathbb{Z}^3 \subseteq \mathbb{Z} \begin{bmatrix} 1 & 1 & 1 \end{bmatrix} \oplus 0 \times \mathbb{Z}^2 \subseteq \mathbb{Z} \left[ \frac{1}{2} \right]\begin{bmatrix} 1 & 1 & 1 \end{bmatrix} \oplus 0 \times \mathbb{Z}^2$.

When $k=1$,  we can write $vA_{E^1_3}=a\begin{bmatrix} 1 & 1 & 1 \end{bmatrix} + \begin{bmatrix} 0 & b & c\end{bmatrix}$, where $a,b,c \in \mathbb{Z}$. Thus 
\begin{align*}
v=&\frac{a}{2}\begin{bmatrix} 1 & 1 & 1 \end{bmatrix} + \frac{1}{2} \begin{bmatrix} b-c & b+c & -b+c\end{bmatrix}\\ =& \frac{a}{2}\begin{bmatrix} 1 & 1 & 1 \end{bmatrix} + \frac{1}{2} \begin{bmatrix} b-c & b-c & b-c \end{bmatrix} + \frac{1}{2} \begin{bmatrix} 0 & 2c & 2b\end{bmatrix} \\
=& \frac{a+(b-c)}{2}\begin{bmatrix} 1 & 1 & 1 \end{bmatrix} + \frac{1}{2} \begin{bmatrix} 0 & 2c & 2b\end{bmatrix} \\ =& \frac{a+(b-c)}{2}\begin{bmatrix} 1 & 1 & 1 \end{bmatrix} + \begin{bmatrix} 0 & c & b\end{bmatrix}.
\end{align*}

Hence, $$v \in \frac{1}{2}\mathbb{Z} \begin{bmatrix} 1 & 1 & 1 \end{bmatrix} \oplus 0 \times \mathbb{Z}^2 \subseteq \mathbb{Z} \left[ \frac{1}{2} \right]\begin{bmatrix} 1 & 1 & 1 \end{bmatrix} \oplus 0 \times \mathbb{Z}^2.$$  
Assume now that  $vA_{E^1_3}^k \in \mathbb{Z}^3$ implies that $v \in \frac{1}{2^k}\mathbb{Z} \begin{bmatrix} 1 & 1 & 1 \end{bmatrix} \oplus 0 \times \mathbb{Z}^2$.
If $vA_{E^1_3}^{k+1} \in \mathbb{Z}^3$, then $vA_{E^1_3}=\frac{a}{2^k}\begin{bmatrix} 1 & 1 & 1 \end{bmatrix} + \begin{bmatrix} 0 & b & c \end{bmatrix} $, where $a,b,c \in \mathbb{Z}$. Thus \begin{align*}
v=&\frac{a}{2^{k+1}}\begin{bmatrix} 1 & 1 & 1 \end{bmatrix} + \frac{1}{2} \begin{bmatrix} b-c & b+c & -b+c\end{bmatrix}\\ =& \frac{a}{2^{k+1}}\begin{bmatrix} 1 & 1 & 1 \end{bmatrix} + \frac{1}{2} \begin{bmatrix} b-c & b-c & b-c \end{bmatrix} + \frac{1}{2}  \begin{bmatrix} 0 & 2c & 2b\end{bmatrix} \\
=& \frac{a+2^k(b-c)}{2^{k+1}}\begin{bmatrix} 1 & 1 & 1 \end{bmatrix} + \frac{1}{2} \begin{bmatrix} 0 & 2c & 2b\end{bmatrix} \\ =& \frac{a+2^k(b-c)}{2^{k+1}}\begin{bmatrix} 1 & 1 & 1 \end{bmatrix} + \begin{bmatrix} 0 & c & b\end{bmatrix}.
\end{align*} Hence $v \in \frac{1}{2^{k+1}}\mathbb{Z} \begin{bmatrix} 1 & 1 & 1 \end{bmatrix} \oplus 0 \times \mathbb{Z}^2 \subseteq \mathbb{Z} \left[ \frac{1}{2} \right]\begin{bmatrix} 1 & 1 & 1 \end{bmatrix} \oplus 0 \times \mathbb{Z}^2$. So $$\Delta_{A_{E^1_3}} \subseteq \mathbb{Z} \left[ \frac{1}{2} \right]\begin{bmatrix} 1 & 1 & 1 \end{bmatrix} \oplus 0 \times \mathbb{Z}^2,$$ and thus the equality follows. 
\end{proof}

The action of $\mathbb Z[x,x ^{-1}]$ on the graded Grothendieck group is as follows: 
$${}^1\left( \frac{a}{2^{k}}\begin{bmatrix} 1 & 1 & 1 \end{bmatrix} + \begin{bmatrix} 0 & b & c\end{bmatrix} \right)= \frac{a+2^k(b+c)}{2^{k+1}}\begin{bmatrix} 1 & 1 & 1 \end{bmatrix} + \begin{bmatrix} 0 & -b & -c\end{bmatrix}$$

Since $E^1_3$ is primitive, we can use Corollary~\ref{halfspace} for calculating the positive cone.  Therefore we have:

\begin{align*}
K_0^{\gr}(L_\K(E^1_3)) & = \mathbb{Z}\left[ \frac{1}{2} \right] \begin{bmatrix} 1 & 1 & 1 \end{bmatrix} \oplus 0 \times \mathbb{Z}^2, \\ 
 K_0^{\gr}(L_\K(E^1_3))^+& =\left\{ {\bf u} \in \mathbb{Z}\left[ \frac{1}{2} \right] \begin{bmatrix} 1 & 1 & 1 \end{bmatrix} \oplus 0 \times \mathbb{Z}^2 \, | \, {\bf u} \cdot {\bf z} > 0 \right\} \cup \{0\}, \text{ where }\bf{z} \text{ is } \left[
\begin{array}{c}
1\\
1\\
1 \\
\end{array}
\right],\\
{}^1(a,b,c) & =(a,b,c) A_{E_1^3}= (a+c, a+b, b+c).
\end{align*}

{\bf 4.} Consider the graph $E_4^1$ below with the adjacency matrix $A_{E_4^1}=\begin{bmatrix} 0 & 1 & 0 \\ 0 & 0 & 1 \\ 1 & 1 & 0 \\ \end{bmatrix}.$

\begin{figure}[H]
\centering
\begin{tikzpicture}[scale=3.5]
\fill (0,0)  circle[radius=.6pt];
\fill (-.35,-.61)  circle[radius=.6pt];
\fill (.35,-.61)  circle[radius=.6pt];
\draw (0,-.15) node{$v_1$};
\draw (-.70, 0) node{$E^1_4:$};
\draw (-.35+.13,-.535) node{$v_3$};
\draw (.35-.13,-.535) node{$v_2$};
\draw[->, shorten >=5pt] (0,0) to (.35,-.61);
\draw[->, shorten >=5pt] (.35,-.61) to (-.35,-.61);
\draw[->, shorten >=5pt] (-.35,-.61) to [in=-120,out=-60] (.35,-.61) ;
\draw[->, shorten >=5pt] (-.35,-.61) to (0,0);
\end{tikzpicture}
\end{figure}

Since $\det(E^1_4) = 1$ and $E^1_4$ is primitive,  Lemma~\ref{DL} and Corollary~\ref{halfspace} give:
\begin{align*}
K_0^{\gr}(L_\K(E^1_4))&=  \mathbb{Z} ^3, \\
K_0^{\gr}(L_\K(E^1_4))^+&=\left\{ {\bf u} \in \mathbb{Z}^3 \, | \, {\bf u} \cdot {\bf z} > 0 \right\}  \cup \, \{0\}, \text{ where } \bf{z} \text{ is } \left[
\begin{array}{c}
 \frac{1}{3} \left(1-5 \sqrt[3]{\frac{2}{11+3
   \sqrt{69}}}+\sqrt[3]{\frac{1}{2} \left(11+3
   \sqrt{69}\right)}\right) \\ 
   \frac{1}{3}
   \left(-1+\sqrt[3]{\frac{1}{2} \left(25-3
   \sqrt{69}\right)
   }+\sqrt[3]{\frac{1}{2} \left(25+3
   \sqrt{69}\right)}\right) \\ 1 \\
\end{array}
\right]
,\\
{}^1(a,b,c)&=(a,b,c)A_{E_4^1}= (c, a+c, b).
\end{align*}

{\bf 5.} Consider the graph $E^1_5, E^1_{13}, E^1_{14}, E^1_{16}$ below with the adjacency matrices $A_{E^1_5}=\begin{bmatrix} 1 & 1 & 0 \\ 0 & 0 & 1 \\ 1 & 1 & 0 \\ \end{bmatrix},
A_{E^1_{13}}= \begin{bmatrix} 0 & 0 & 1 \\ 1 & 0 & 0 \\ 1 & 0 & 1 \\ \end{bmatrix},  A_{E^1_{14}}= \begin{bmatrix} 1 & 0 & 1 \\ 1 & 0 & 0 \\ 1 & 0 & 0 \\ \end{bmatrix}$ and 
$A_{E^1_{16}}= \begin{bmatrix} 0 & 1 & 1 \\ 0 & 0 & 1 \\ 0 & 1 & 1 \\ \end{bmatrix},$ respectively. 
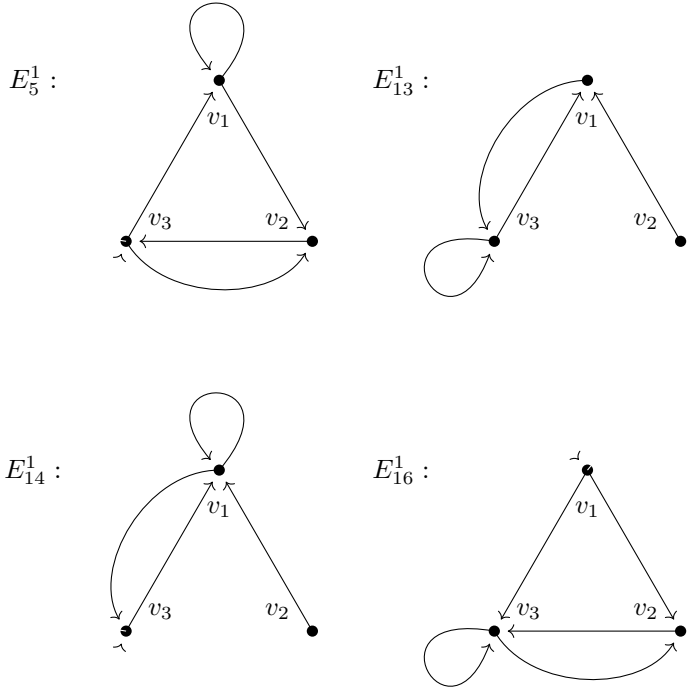
\begin{figure}[H]
\centering
\begin{tikzpicture}[scale=3.5]
\fill (0,0)  circle[radius=.6pt];
\fill (-.35,-.61)  circle[radius=.6pt];
\fill (.35,-.61)  circle[radius=.6pt];
\draw (0,-.15) node{$v_1$};
\draw (-.70, 0) node{$E^1_5:$};
\draw (-.35+.13,-.535) node{$v_3$};
\draw (.35-.13,-.535) node{$v_2$};
\draw[->, shorten >=5pt] (0,0) to (.35,-.61);
\draw[->, shorten >=5pt] (.35,-.61) to (-.35,-.61);
\draw[->, shorten >=5pt] (-.35,-.61) to [in=-120,out=-60] (.35,-.61) ;
\draw[->, shorten >=5pt] (-.35,-.61) to (0,0);
\draw[->, shorten >=5pt] (0,0) to[in=130,out=50, loop] (0,0);
\draw[draw=white, ->, shorten >=5pt] (-.35,-.61) to [in=250,out=170, loop] (-.35,-.61);
\end{tikzpicture}
\hspace{10pt}
\begin{tikzpicture}[scale=3.5]
\fill (0,0)  circle[radius=.6pt];
\fill (-.35,-.61)  circle[radius=.6pt];
\fill (.35,-.61)  circle[radius=.6pt];
\draw (0,-.15) node{$v_1$};
\draw (-.70, 0) node{$E^1_{13}:$};
\draw (-.35+.13,-.535) node{$v_3$};
\draw (.35-.13,-.535) node{$v_2$};
\draw[->, shorten >=5pt] (.35,-.61) to (0,0);
\draw[->, shorten >=5pt] (-.35,-.61) to (0,0);
\draw[->, shorten >=5pt] (0,0) to [in=120,out=180
] (-.35,-.61) ;
\draw[->, shorten >=5pt] (-.35,-.61) to [in=250,out=170, loop] (-.35,-.61);

\end{tikzpicture}
\\
\begin{tikzpicture}[scale=3.5]
\fill (0,0)  circle[radius=.6pt];
\fill (-.35,-.61)  circle[radius=.6pt];
\fill (.35,-.61)  circle[radius=.6pt];
\draw (0,-.15) node{$v_1$};
\draw (-.70, 0) node{$E^1_{14}:$};
\draw (-.35+.13,-.535) node{$v_3$};
\draw (.35-.13,-.535) node{$v_2$};
\draw[->, shorten >=5pt] (.35,-.61) to (0,0);
\draw[->, shorten >=5pt] (-.35,-.61) to (0,0);
\draw[->, shorten >=5pt] (0,0) to [in=120,out=180
] (-.35,-.61) ;
\draw[->, shorten >=5pt] (0,0) to[in=130,out=50, loop] (0,0);
\draw[draw=white, ->, shorten >=5pt] (-.35,-.61) to [in=250,out=170, loop] (-.35,-.61);
\end{tikzpicture}
\hspace{10pt}
\vspace{5pt}
\begin{tikzpicture}[scale=3.5]
\fill (0,0)  circle[radius=.6pt];
\fill (-.35,-.61)  circle[radius=.6pt];
\fill (.35,-.61)  circle[radius=.6pt];
\draw (0,-.15) node{$v_1$};
\draw (-.70, 0) node{$E^1_{16}:$};
\draw (-.35+.13,-.535) node{$v_3$};
\draw (.35-.13,-.535) node{$v_2$};
\draw[->, shorten >=5pt] (0,0) to (.35,-.61);
\draw[->, shorten >=5pt] (.35,-.61) to (-.35,-.61);
\draw[->, shorten >=5pt] (-.35,-.61) to [in=-120,out=-60] (.35,-.61) ;
\draw[->, shorten >=5pt] (0,0) to (-.35,-.61) ;
\draw[->, shorten >=5pt] (-.35,-.61) to[in=250,out=170, loop] (-.35,-.61);
\draw[draw=white, ->, shorten >=5pt] (0,0) to[in=130,out=50, loop] (0,0);
\end{tikzpicture}
\end{figure}

The determinants of the adjacency matrices of these graphs are all zero. However, they are the results of an out-split or adding a source vertex to the graph $F^1_1$ below:  
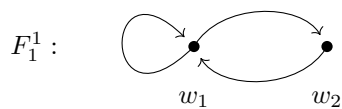
\begin{figure}[H]
\centering
\begin{tikzpicture}[scale=3.5]
\draw (-.60, 0) node{$F^1_{1}:$};
\fill (0,0)  circle[radius=.6pt];
\draw (0,-.2) node{$w_1$};
\fill (.5,0)  circle[radius=.6pt];
\draw (.5,-.2) node{$w_2$};
\draw[->, shorten >=5pt] (0,0) to[in=135, out=-135, loop] (0,0);
\draw[->, shorten >=5pt] (0,0) to[in=120,out=60] (.5,0);
\draw[->, shorten >=5pt] (.5,0) to[in=-60,out=-120] (0,0);
\end{tikzpicture}
\end{figure}

  Since the out-splitting and adding vertices preserve the graded Morita theory, the graded $K$-theory data remains the same.  This allows us to calculate the graded $K$-theory of Leavitt path algebras 
 $E^1_5, E^1_{13}, E^1_{14}$ and $E^1_{16}$, by calculating the graded $K$-theory of $L_\K(F^1_1)$. Since $\det(F^1_1) = -1 $ and $F^1_1$ is primitive, Lemma~\ref{DL} and Corollary~\ref{halfspace} give:
\begin{align*}
 K_0^{\gr}(L_\K(F^1_1)) &=  \mathbb{Z} ^2, \\
 K_0^{\gr}(L_\K(F^1_1))^+ &=\left\{ {\bf u} \in \mathbb{Z}^2 \, | \, {\bf u} \cdot {\bf z} > 0 \right\}  \cup \, \{0\}, \text{ where } \bf{z} \text{  is  }\left[
\begin{array}{c}
 \frac{1}{2} \left(1+\sqrt{5} \right) \\  1 \\
\end{array}
\right],
\\
{}^1(a,b)&=(a+b,a).
\end{align*}

We have: 

\begin{itemize}
    \item[-] $E_5^1$ is obtained from $F^1_1$ by  a maximal out-split at $w_1$.
    \item[-] $E_{13}^1$ is obtained from $F^1_1$ by adding a source and an edge that enters the vertex $w_2$.
        \item[-] $E_{14}^1$ is obtained from $F^1_1$ by adding a source and an edge that enters the vertex $w_1$.
    \item[-] $E_{16}^1$  is obtained from $F^1_1$ by adding a source and two edges that enter the vertices $w_1$ and $w_2$, respectively.
\end{itemize}

Therefore for $E$ being one of the graphs $E^1_5, E^1_{13}, E^1_{14}$ and $E^1_{16}$, we have a graded Morita equivalence $\mbox{Gr-} L_\K(E) \cong \mbox{Gr-} L_\K(F^1_1)$ and  thus the graded $K$-theory data of $L_\K(E)$  coincides with the graded $K$-theory of $L_\K(F^1_1)$. 

Next we show that all these graphs are elementary shift equivalent. 

- $E^1_5$ is elementary shift equivalent to $E^1_{13}$, as $E^1_5 = RS$ and $E^1_{13}=SR$, where
 $$R=\begin{bmatrix}
    0 & 0 & 1 \\
    1 & 0 & 0 \\
    0 & 0 & 1 \\
    \end{bmatrix} \text{ and } 
    S= \begin{bmatrix}
    0 & 0 & 1 \\
    0 & 1 & 0 \\
    1 & 1 & 0 \\
\end{bmatrix}.$$

- $E^1_5$ is elementary shift equivalent to $E^1_{14}$, as $E^1_5 = RS$ and  $E^1_{14}=SR$, where
 $$R=\begin{bmatrix}
    1 & 0 & 0 \\
    0 & 0 & 1 \\
    1 & 0 & 0 \\
    \end{bmatrix}  \text{ and }   S= \begin{bmatrix}
    1 & 1 & 0 \\
    0 & 0 & 1 \\
    0 & 0 & 1 \\
\end{bmatrix}.$$

- $E^1_5$ is elementary shift equivalent to $E^1_{16}$, as $E^1_5 = RS$ and $E^1_{16}=SR$, where
$$R=\begin{bmatrix}
    0 & 0 & 1 \\
    0 & 1 & 0 \\
    0 & 0 & 1 \\
    \end{bmatrix} \text{ and }  S= \begin{bmatrix}
    0 & 1 & 1 \\
    0 & 0 & 1 \\
    1 & 1 & 0 \\
\end{bmatrix}.$$

- $E^1_{13}$ is elementary shift equivalent to $E^1_{14}$, as $E^1_{13} = RS$ and $E^1_{14}=SR$, 
 $$R=\begin{bmatrix}
    1 & 0 & 0 \\
    0 & 0 & 1 \\
    1 & 0 & 1 \\
\end{bmatrix} \text{ and }  S= \begin{bmatrix}
    0 & 0 & 1 \\
    1 & 0 & 0 \\
    1 & 0 & 0 \\
\end{bmatrix}.$$

- $E^1_{13}$ is elementary shift equivalent to $E^1_{16}$, as $E^1_{13} = RS$ and $E^1_{16}=SR$, where 
$$R=\begin{bmatrix}
    0 & 0 & 1 \\
    0 & 1 & 0 \\
    0 & 1 & 1 \\
\end{bmatrix} \text{ and }  S= \begin{bmatrix}
    0 & 0 & 1 \\
    1 & 0 & 0 \\
    0 & 0 & 1 \\
\end{bmatrix}.$$

- $E^1_{14}$ is elementary shift equivalent to $E^1_{16}$, as $E^1_{14} = RS$ and $E^1_{16}=SR$, where
$$R= \begin{bmatrix}
    0 & 0 & 1 \\
    0 & 1 & 0 \\
    0 & 1 & 0 \\
    \end{bmatrix} \text{ and }  S= \begin{bmatrix}
    1 & 0 & 1 \\
    1 & 0 & 0 \\
    1 & 0 & 1 \\
\end{bmatrix}.$$

\medskip 

$\bullet$  \emph{The algebras $L_\K(E^1_5)$, $L_\K(E^1_{13})$ and $L_\K(E^1_{14})$ are not graded isomorphic to each other. }

\medskip

We show that there is no pointed module isomorphism between the corresponding graded $K_0$-groups of these algebras, thus they can't be graded isomorphic to each other.

First note that $E^1_5$ is elementary shift equivalent to $F^1_1$, as $E^1_{5} = R_1S_1$ and $F^1_{1}=S_1R_1$, with 
\[
 R_1=\begin{bmatrix}
    1 & 0 \\ 0 & 1 \\ 1 & 0 \\
    \end{bmatrix}
\text{ and }
  S_1=\ \begin{bmatrix} 1 & 1 & 0 \\ 0 & 0 & 1 \\
 \end{bmatrix}  
 \]

 We also have that $E^1_{13}$ is elementary shift equivalent to $F^1_1$, as $E^1_{13} = R_2S_2$ and $F^1_1=S_2R_2$, with 
  \[
 R_2=\begin{bmatrix}
    1 & 0 \\ 
    0 & 1 \\ 
    1 & 1 \\
    \end{bmatrix}
\text{ and }
  S_2=\ \begin{bmatrix} 
  0 & 0 & 1 \\ 
  1 & 0 & 0 \\
 \end{bmatrix}  
 \]

Following \S\ref{sec:symbolic-dynamics} (see (\ref{thuemas})), the matrices $R_1$ and $R_2$ induce order-preserving module isomorphisms $\phi_{R_1}$ and $\phi_{R_2}$ (diagram below) which send the order units 
 $\begin{bmatrix} 1 & 1 & 1 \end{bmatrix} R_1= \begin{bmatrix} 2 & 1  \end{bmatrix}$ and  $\begin{bmatrix} 1 & 1 & 1 \end{bmatrix} R_2= \begin{bmatrix} 2 & 2  \end{bmatrix}$:
\[
\begin{tikzcd}
K_0^{\gr}(L_\K(E^1_5)) \arrow[r, "\phi_{R_1}" ] \arrow[d, dashed,"\mu" ]
& K_0^{\gr}(L_\K(F^1_1)) \arrow[d, "T" ] \\
K_0^{\gr}(L_\K(E^1_{13})) \arrow[r, "\phi_{R_2}"]
& K_0^{\gr}(L_\K(F^1_1))
\end{tikzcd}
\]

If there was a pointed isomorphism  $\mu: K_0^{\gr}(L_\K(E^1_5)) \rightarrow K_0^{\gr}(L_\K(E^1_{13}))$, there would need to be a $T\in \GL_2( \mathbb{Z})$ such that $ T F^1_1= F^1_1 T$ and $\begin{bmatrix} 2 & 1  \end{bmatrix} T = \begin{bmatrix} 2 & 2  \end{bmatrix}$. Thus for 
 $T=\begin{bmatrix}
a & b \\
c & d 
\end{bmatrix}$ 
 these conditions give: 
\begin{align}
\begin{bmatrix}
1 & 1 \\
1 & 0 
\end{bmatrix}
\begin{bmatrix}
a & b \\
c & d 
\end{bmatrix} & =
\begin{bmatrix}
a & b \\
c & d 
\end{bmatrix}
\begin{bmatrix}
1 & 1 \\
1 & 0 
\end{bmatrix}\label{matrix11}\\ 
\begin{bmatrix}
2 & 1 
\end{bmatrix}
\begin{bmatrix}
a & b \\
c & d 
\end{bmatrix} & =
\begin{bmatrix}
2 & 2 
\end{bmatrix}\label{matrix22}
\end{align}

This yields four equations $2a+b=2$, $2b+d=2$, $a=b+d$ and $b=c$. Solving the equations, we obtain 
$T= \begin{bmatrix}
0 & 2 \\
2 & -2 
\end{bmatrix}$. However $T\not \in \GL_2(\mathbb Z)$. Thus there is no pointed module isomorphism between the algebras and consequently they are not graded isomorphic.

However, this system has no solution, implying that $L_\K(E^2_3)$ can't be graded isomorphic to  $L_\K(E^3_4)$. 

 We  have that $E^1_{14}$ is elementary shift equivalent to $F^1_1$, as $E^1_{14} = R_3S_3$ and $F^1_1=S_3R_3$, with 
  \[
 R_2=\begin{bmatrix}
    1 & 0 \\ 
    0 & 1 \\ 
    0 & 1 \\
    \end{bmatrix}
\text{ and }
  S_2=\ \begin{bmatrix} 
  1 & 0 & 1 \\ 
  1 & 0 & 0 \\
 \end{bmatrix}  
 \]
Similar argument shows $\begin{bmatrix} 1 & 1 & 1 \end{bmatrix} R_3= \begin{bmatrix} 1 & 2  \end{bmatrix}$.  Equation~\ref{matrix11} along with 
$$
\begin{bmatrix}
2 & 1 
\end{bmatrix}
\begin{bmatrix}
a & b \\
c & d 
\end{bmatrix}  =
\begin{bmatrix}
2 & 2 
\end{bmatrix}$$
again leads to a matrix $T$ which is not invertible in $\mathbb Z$ implying there no pointed isomorphism $\mu\colon K_0^{\gr}(L_\K(E^1_5)) \rightarrow K_0^{\gr}(L_\K(E^1_{14}))$. Finally, Equation~\ref{matrix11} along with 
$$
\begin{bmatrix}
2 & 2
\end{bmatrix}
\begin{bmatrix}
a & b \\
c & d 
\end{bmatrix}  =
\begin{bmatrix}
1 & 2 
\end{bmatrix}$$
shows there is no suitable matrix $T$ and so no pointed isomorphism $\mu \colon  K_0^{\gr}(L_\K(E^1_{13})) \rightarrow K_0^{\gr}(L_\K(E^1_{14}))$ exists.

\medskip

$\bullet$  \emph{$L_\K(E^1_5)$ is graded isomorphic to $L_\K(E^1_{16})$.}

\medskip 

First note that $E^1_{16}$ is elementary shift equivalent to $F^1_1$, as $E^1_{16} = R_4S_4$ and $F^1_{1}=S_4R_4$, with 
$$ R_4=\begin{bmatrix}
    1 & 0 \\
    0 & 1 \\
    1 & 0 \\
\end{bmatrix}  \text{ and }
S_4=\begin{bmatrix}
    0 & 1 & 1 \\
    0 & 0 & 1 \\
\end{bmatrix}.$$

Since  $\begin{bmatrix} 1 & 1 & 1 \end{bmatrix} R_4= \begin{bmatrix} 2 & 1  \end{bmatrix}$ and  $\begin{bmatrix} 1 & 1 & 1 \end{bmatrix} R_1= \begin{bmatrix} 2 & 1  \end{bmatrix}$, 
clearly $T=\id$, $2\times 2$-identity matrix,  gives a pointed isomorphism between $K_0^{\gr}(E^1_5)$ and $K_0^{\gr}(E^1_{16})$. 
\[
\begin{tikzcd}
K_0^{\gr}(L_\K(E^1_5)) \arrow[r, "\phi_{R_1}" ] \arrow[d, dashed,"\mu" ]
& K_0^{\gr}(L_\K(F^1_1)) \arrow[d, "T=\id" ] \\
K_0^{\gr}(L_\K(E^1_{16})) \arrow[r, "\phi_{R_4}"]
& K_0^{\gr}(L_\K(F^1_1))
\end{tikzcd}
\]

We know that 
\begin{equation}\label{ldjkd}
L_\K(E^1_5) \cong_{\gr} L_\K(F^1_1),
\end{equation}
as out-splitting preserves the graded isomorphism. Next we show that $L_\K(E^1_{16})$ is graded isomorphic to $L_\K(F^1_1)$. 
From $E^1_{16}$ the maximum out-split  at $v_1$ produces the graph $F^1_2$ below. 
\begin{figure}[H]
\begin{tikzpicture}[scale=3.5]
\fill (-.35,0)  circle[radius=.6pt];
\fill (.35,0)  circle[radius=.6pt];
\fill (-.35,0)  circle[radius=.6pt];
\fill (-.35,-.61)  circle[radius=.6pt];
\fill (.35,-.61)  circle[radius=.6pt];
\draw (-.70, 0) node{$F_1^2:$};
\draw (-.35+.13,-.61) node{$v_3$};
\draw (.35-.13,-.61) node{$v_2$};
\draw[->, shorten >=5pt] (.35,0) to (.35,-.61);
\draw[->, shorten >=5pt] (.35,-.61) to [in=60,out=120] 
 (-.35,-.61);
\draw[->, shorten >=5pt] (-.35,-.61) to [in=-120,out=-60] (.35,-.61) ;
\draw[->, shorten >=5pt] (-.35,0) to (-.35,-.61) ;
\draw[->, shorten >=5pt] (-.35,-.61) to[in=220,out=130, loop] (-.35,-.61);
\draw[draw=white, ->, shorten >=5pt] (0,0) to[in=130,out=50, loop] (0,0);
\end{tikzpicture}
\end{figure}

Thus we have $L_\K(E^1_{16}) \cong_{\gr} L_\K(F^2_1)$. By \cite[Lemma~4.18, Theorem~4.14]{Hisrael} we have 
$$L_\K(F^2_1) \cong_{\gr} \M_2\big(L_\K(F^1_1)\big)\big(0,1\big).$$
We show that 
\begin{equation}\label{hfghfghf}
 \M_2\big(L_\K(F^1_1)\big)\big(0,1\big)  \cong_{\gr} L_\K(F^1_1).
\end{equation}
Recall from Section~\ref{hfjduier87} that we have an $\mathbb Z$-monoid isomorphism 
\begin{equation}\label{data475}
K_0^{\gr}(L_\K(F^1_1))^+ \cong T_{F^1_1}\cong \frac{\big \langle w_1(i), w_2(i) \mid i\in \mathbb Z\big \rangle}{\big \langle w_1(i)=w_1(i+1)w_2(i+1), w_2(i)=w_1(i+1)\big\rangle } .
\end{equation} 
Under this isomorphism,  $[L_\K(F^1_1)]\in  K_0^{\gr}(L_\K(F^1_1))^+$ is represented by $w_1+w_2 \in T_{F^1_1}$. In the talented monoid $T_{F^1_1}$, we have
$$w_1+w_2= w_1(1)+w_2(1)+w_2=w_1(1)+w_2(1)+w_1(1)=w_1(1)+w_2(1)+w_1(2)+w_2(2).$$ Passing this back to $K_0^{\gr}(L_\K(F^1_1))^+$, we get $L_\K(F_1^1)$-module isomorphism 
$$L_\K(F^1_1) \cong_{\gr} L_\K(F^1_1)(1)\bigoplus L_\K(F^1_1) (2).$$
Taking $\End_{L_\K(F^1_1)}$ of both sides, we obtain~(\ref{hfghfghf})
$$L_\K(F^1_1)  \cong_{\gr} \M_2\big(L_\K(F^1_1) \big)\big(-1,-2\big) \cong_{\gr} \M_2\big(L_\K(F^1_1) \big)\big(0,1\big).$$
Thus $$L_\K(E^1_{16}) \cong_{\gr}L_\K(F^2_1) \cong_{\gr} \M_2\big(L_\K(F^1_1) \big)\big(0,1\big) \cong_{\gr} L_\K(F^1_1)\cong L_\K(F^1_5).$$

{\bf 6.} Consider the graphs $E^1_6,  E^1_7, E^1_{15}, E^1_{17}$ below with the adjacency matrices $A_{E^1_6}= \begin{bmatrix} 1 & 1 & 0 \\ 0 & 0 & 1 \\ 1 & 1 & 1 \\ \end{bmatrix}, 
A_{E^1_{7}}= \begin{bmatrix} 1 & 1 & 0 \\ 0 & 1 & 1 \\ 1 & 1 & 0 \\ \end{bmatrix},  A_{E^1_{15}}= \begin{bmatrix} 1 & 0 & 1 \\ 1 & 0 & 0 \\ 1 & 0 & 1 \\ \end{bmatrix}$ and $
A_{E^1_{17}}= \begin{bmatrix} 0 & 1 & 1 \\ 0 & 1 & 1 \\ 0 & 1 & 1 \\ \end{bmatrix}$, respectively.

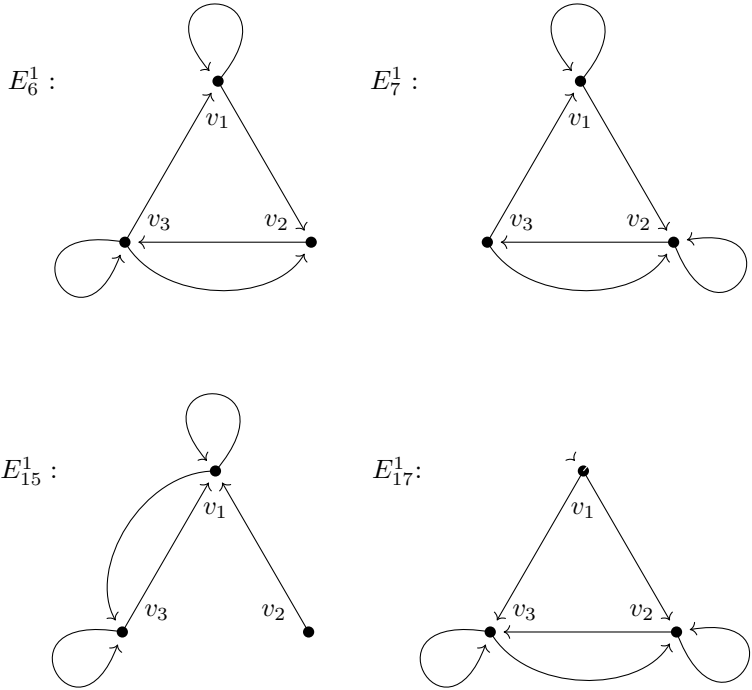
\begin{figure}[H]
\centering
\begin{tikzpicture}[scale=3.5]
\fill (0,0)  circle[radius=.6pt];
\fill (-.35,-.61)  circle[radius=.6pt];
\fill (.35,-.61)  circle[radius=.6pt];
\draw (0,-.15) node{$v_1$};
\draw (-.70, 0) node{$E^1_6:$};
\draw (-.35+.13,-.535) node{$v_3$};
\draw (.35-.13,-.535) node{$v_2$};
\draw[->, shorten >=5pt] (0,0) to (.35,-.61);
\draw[->, shorten >=5pt] (.35,-.61) to (-.35,-.61);
\draw[->, shorten >=5pt] (-.35,-.61) to [in=-120,out=-60] (.35,-.61) ;
\draw[->, shorten >=5pt] (-.35,-.61) to (0,0);
\draw[->, shorten >=5pt] (0,0) to[in=130,out=50, loop] (0,0);
\draw[ ->, shorten >=5pt] (-.35,-.61) to [in=250,out=170, loop] (-.35,-.61);
\end{tikzpicture}
\hspace{10pt}
\begin{tikzpicture}[scale=3.5]
\fill (0,0)  circle[radius=.6pt];
\fill (-.35,-.61)  circle[radius=.6pt];
\fill (.35,-.61)  circle[radius=.6pt];
\draw (0,-.15) node{$v_1$};
\draw (-.70, 0) node{$E^1_7:$};
\draw (-.35+.13,-.535) node{$v_3$};
\draw (.35-.13,-.535) node{$v_2$};
\draw[->, shorten >=5pt] (0,0) to (.35,-.61);
\draw[->, shorten >=5pt] (.35,-.61) to (-.35,-.61);
\draw[->, shorten >=5pt] (-.35,-.61) to [in=-120,out=-60] (.35,-.61) ;
\draw[->, shorten >=5pt] (-.35,-.61) to (0,0);
\draw[->, shorten >=5pt] (0,0) to[in=130,out=50, loop] (0,0);
\draw[->, shorten >=5pt] (.35,-.61) to [in=10,out=-70, loop] (.35,-.61);
\end{tikzpicture}
\\
\begin{tikzpicture}[scale=3.5]
\fill (0,0)  circle[radius=.6pt];
\fill (-.35,-.61)  circle[radius=.6pt];
\fill (.35,-.61)  circle[radius=.6pt];
\draw (0,-.15) node{$v_1$};
\draw (-.70, 0) node{$E^1_{15}:$};
\draw (-.35+.13,-.535) node{$v_3$};
\draw (.35-.13,-.535) node{$v_2$};
\draw[->, shorten >=5pt] (.35,-.61) to (0,0);
\draw[->, shorten >=5pt] (-.35,-.61) to (0,0);
\draw[->, shorten >=5pt] (0,0) to [in=120,out=180
] (-.35,-.61) ;
\draw[->, shorten >=5pt] (0,0) to[in=130,out=50, loop] (0,0);
\draw[->, shorten >=5pt] (-.35,-.61) to[in=250,out=170, loop] (-.35,-.61);
\end{tikzpicture}
\hspace{10pt}
\begin{tikzpicture}[scale=3.5]
\fill (0,0)  circle[radius=.6pt];
\fill (-.35,-.61)  circle[radius=.6pt];
\fill (.35,-.61)  circle[radius=.6pt];
\draw (0,-.15) node{$v_1$};
\draw (-.70, 0) node{$E^1_{17}$:};
\draw (-.35+.13,-.535) node{$v_3$};
\draw (.35-.13,-.535) node{$v_2$};
\draw[->, shorten >=5pt] (0,0) to (.35,-.61);
\draw[->, shorten >=5pt] (.35,-.61) to (-.35,-.61);
\draw[->, shorten >=5pt] (-.35,-.61) to [in=-120,out=-60] (.35,-.61) ;
\draw[->, shorten >=5pt] (0,0) to (-.35,-.61);
\draw[draw=white, ->, shorten >=5pt] (0,0) to[in=130,out=50, loop] (0,0);
\draw[->, shorten >=5pt] (-.35,-.61) to[in=250,out=170, loop] (-.35,-.61);
\draw[->, shorten >=5pt] (.35,-.61) to[in=10,out=-70, loop] (-.35,-.61);
\end{tikzpicture}
\end{figure}

The determinants of the adjacency matrices of these graphs are all zero. However, they are the results of an out-split, in-split or adding a source vertex to the graph $F^1_2$ below:  
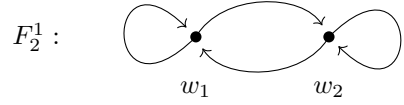
\begin{figure}[H]
\centering
\begin{tikzpicture}[scale=3.5]
\draw (-.60, 0) node{$F^1_{2}:$};

\fill (0,0)  circle[radius=.6pt];
\draw (0,-.2) node{$w_1$};
\fill (.5,0)  circle[radius=.6pt];
\draw (.5,-.2) node{$w_2$};
\draw[->, shorten >=5pt] (0,0) to[in=135, out=-135, loop] (0,0);
\draw[->, shorten >=5pt] (.5,0) to[in=-45,out=45, loop] (.5,0);
\draw[->, shorten >=5pt] (0,0) to[in=120,out=60] (.5,0);
\draw[->, shorten >=5pt] (.5,0) to[in=-60,out=-120] (0,0);
\end{tikzpicture}
\end{figure}
\noindent which in return is the maximal out-split of two petal rose:
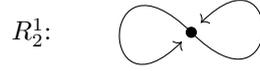
\begin{figure}[H]
\centering
\begin{tikzpicture}[scale=3.5]
\fill (0,0) circle[radius=.6pt];
\draw (-.60, 0) node{$R^1_2$:};

\draw[->, shorten >=5pt] (0,0) to[in=45, out=-45, loop] (0,0);
\draw[->, shorten >=5pt] (0,0) to[in=-135,out=135, loop] (0,0);
\end{tikzpicture}

\end{figure}

Thus the graded $K$-theory of Leavitt path algebras of the graphs $E^1_6,  E^1_7, E^1_{15}$ and  $E^1_{17}$ coincide with the graded $K$-theory of the two petal rose $R^1_2$. Since the adjacency matrix $A_{R^1_2}=[2]$, then (\ref{linkkgroup}) (or Colloray~\ref{alternativeformdel}) immediately gives: 
\begin{align*}
K_0^{\gr}(L_\K(R^1_2)) & =  \mathbb{Z}\left[ \frac{1}{2} \right], \\
K_0^{\gr}(L_\K(R^1_2))^+ & =\left\{ {\bf u} \in \mathbb{Z}\left[ \frac{1}{2} \right] \, | \, {\bf u} \geq 0 \right\}= \mathbb{N}\left[ \frac{1}{2} \right],\\ 
{}^1a & =2a.
\end{align*}

We have: 

\phantom{a} 
\begin{itemize}
    \item[-]$E^1_6$  is obtained from $F^1_2$ by a maximal out-split at $w_1$. 
    \item[-]$E^1_7$ is obtained from $F^1_2$ by a maximal in-split at $w_1$. 
    \item[-]$E^1_{15}$ is obtained from $F^1_2$ by adding a source and an edge that enters the vertex $w_1$.
    \item[-]$E^1_{17}$ is obtained from $F^1_2$ by adding a source and two edges that enter the vertices $w_1$ and $w_2$, respectively.
\end{itemize}

Therefore for $E$ being one of the graphs $E^1_6, E^1_{7}, E^1_{15}$ and $E^1_{17}$, we have a graded Morita equivalence $\mbox{Gr-} L_\K(E) \cong \mbox{Gr-} L_\K(R_2)$ and  thus the graded $K$-theory data of $L_\K(E)$  coincides with the graded $K$-theory of $L_\K(R_2)$. 

Next we show that all these graphs are elementary shift equivalent.

- $E^1_{6}$ is elementary shift equivalent to $E^1_{7}$ as $E^1_{6} = RS$ and $E^1_{7}=SR$, where, 
$$R= \begin{bmatrix}
    0 & 1 & 0 \\
    0 & 0 & 1 \\
    1 & 1 & 0 \\
    \end{bmatrix} \text{  and } S= \begin{bmatrix}
    0 & 0 & 1 \\
    1 & 1 & 0 \\
    0 & 0 & 1 \\
\end{bmatrix}.$$

- $E^1_{6}$ is elementary shift equivalent to  $E^1_{15}$ as $E^1_{6} = RS$ and $E^1_{15}=SR$, where,
$$R= \begin{bmatrix}
    0 & 0 & 1 \\
    1 & 0 & 0 \\
    1 & 0 & 1 \\
    \end{bmatrix} \text{ and } 
    S= \begin{bmatrix}
    0 & 0 & 1 \\
    0 & 1 & 0 \\
    1 & 1 & 0 \\
\end{bmatrix}.$$

- $E^1_{6}$ is elementary shift equivalent $E^1_{17}$ as $E^1_{6} = RS$ and $E^1_{17}=SR$, where, 
$$R= \begin{bmatrix}
    0 & 0 & 1 \\
    0 & 1 & 0 \\
    0 & 1 & 1 \\
    \end{bmatrix} \text{ and } S= \begin{bmatrix}
    0 & 0 & 1 \\
    0 & 0 & 1 \\
    1 & 1 & 0 \\
\end{bmatrix}.$$

- $E^1_{7}$ is elementary shift equivalent to   $E^1_{15}$ as $E^1_{7} = RS$ and  $E^1_{15}=SR$, where, 
$$R= \begin{bmatrix}
    0 & 0 & 1 \\
    1 & 0 & 0 \\
    0 & 0 & 1 \\
    \end{bmatrix} \text{ and } S= \begin{bmatrix}
    0 & 1 & 1 \\
    0 & 1 & 0 \\
    1 & 1 & 0 \\
\end{bmatrix}.$$

- $E^1_{7}$ is elementary shift equivalent to $E^1_{17}$ as $E^1_{7} = RS$ and  $E^1_{17}=SR$, where, 
$$R= \begin{bmatrix}
    0 & 0 & 1 \\
    0 & 1 & 0 \\
    0 & 0 & 1 \\
    \end{bmatrix} \text{ and } S= \begin{bmatrix}
    0 & 1 & 1 \\
    0 & 1 & 1 \\
    1 & 1 & 0 \\
\end{bmatrix}.$$

- $E^1_{15}$ is elementary shift equivalent to $E^1_{17}$ as $E^1_{15} = RS$ and $E^1_{17}=SR$, where, 
$$R= \begin{bmatrix}
    0 & 1 & 1 \\
    0 & 0 & 1 \\
    0 & 1 & 1 \\
    \end{bmatrix} \text{  and  } S= \begin{bmatrix}
    0 & 0 & 1 \\
    0 & 0 & 1 \\
    1 & 0 & 0 \\
\end{bmatrix}.$$

\medskip

$\bullet$  \emph{The algebras $L_\K(E^1_6)$, $L_\K(E^1_{15})$ and $L_\K(E^1_{17})$ are not graded isomorphic to each other. }

\medskip 

We show that there is no pointed module isomorphism between the corresponding graded $K_0$-groups of these algebras, thus they can't be graded isomorphic to each other.

First note that $E^1_6$ is elementary shift equivalent to $F^1_2$, as $E^1_{6} = R_1S_1$ and $F^1_{2}=S_1R_1$, with 
\[  R_1=\begin{bmatrix}
    1 & 0 \\
    0 & 1 \\
    1 & 1 \\
    \end{bmatrix} \text{ and }  
S_1=\begin{bmatrix}
    1 & 1 & 0 \\
    0 & 0 & 1 \\
\end{bmatrix}. 
    \]
    
We also have that $E^1_{15}$ is elementary shift equivalent to $F^1_2$, as $E^1_{15} = R_2S_2$ and $F^1_{2}=S_2R_2$, with 
\[  R_2=\begin{bmatrix}
    1 & 1 \\
    0 & 1 \\
    1 & 1 \\
    \end{bmatrix} \text{ and }  
S_2=\begin{bmatrix}
    0 & 0 & 1 \\
    1 & 0 & 0 \\
\end{bmatrix}. 
    \]

We also have that $E^1_{17}$ is elementary shift equivalent to $F^1_2$, as $E^1_{17} = R_3S_3$ and $F^1_{2}=S_3R_3$, with 
\[  R_3=\begin{bmatrix}
    0 & 1 \\
    0 & 1 \\
    1 & 0 \\
    \end{bmatrix} \text{ and }  
S_3=\begin{bmatrix}
    0 & 1 & 1 \\
    0 & 1 & 1 \\
\end{bmatrix}. 
    \]

Finally,  $F^1_{2}$ is elementary shift equivalent to $R_2$, as $F^1_{2} = RS$ and $R_{2}=SR$, with 
\[  R=\begin{bmatrix}
    1 \\
    1 \\
    \end{bmatrix} \text{ and }  
S=\begin{bmatrix}
    1&1
 \end{bmatrix}. 
    \]

For $E$ being one of the graphs $E^1_6, E^1_{15}, E^1_{17}$, the sequence of module isomorphisms 
\[K_0^{\gr}(L_\K(E)) \overset{\phi_{R_i}}{\longrightarrow}  K_0^{\gr}(L_\K(F^1_2)) \overset{\phi_{R}}{\longrightarrow}  K_0^{\gr}(L_\K(R_2))\] send the order units 
\begin{align*}
\phi_R\phi_{R_1}(\begin{bmatrix} 1 & 1 & 1 \end{bmatrix})&=4\\
\phi_R\phi_{R_2}(\begin{bmatrix} 1 & 1 & 1 \end{bmatrix})&=5\\
\phi_R\phi_{R_3}(\begin{bmatrix} 1 & 1 & 1 \end{bmatrix})&=3.
\end{align*}

The numbers $3,4,5$ cannot be multiplied to each other with invertible scalars over $K_0^{\gr}(L_\K(R_2)) \cong \mathbb{Z}\left[ \frac{1}{2} \right].$  Thus, there are no isomorphisms between  the algebras $L_\K(E^1_6)$, $L_\K( E^1_{15})$ and $L_\K(E^1_{17})$. 

\medskip 

$\bullet$  \emph{$L_\K(E^1_7)$ is graded isomorphic to $L_\K(E^1_{17})$.}

\medskip

First note that $E^1_7$ is elementary shift equivalent to $F^1_2$ as  $E^1_7=R_4S_4$ and $F^1_2=S_4R_4$, with 
\[ R_4=\begin{bmatrix}
    1 & 0 \\
    0 & 1 \\
    1 & 0 \\
\end{bmatrix} \text{ and }  
S_4=\begin{bmatrix}
    1 & 1 & 0 \\
    0 & 1 & 1 \\
\end{bmatrix} 
.\]

Since  $\begin{bmatrix} 1 & 1 & 1 \end{bmatrix} R_4R= 3$ and  $\begin{bmatrix} 1 & 1 & 1 \end{bmatrix} R_3R= 3$, 
clearly $T=\id$, $2\times 2$-identity matrix,  gives a pointed isomorphism between $K_0^{\gr}(E^1_7)$ and $K_0^{\gr}(E^1_{17})$. 
\[
\begin{tikzcd}
K_0^{\gr}(L_\K(E^1_7)) \arrow[r, "\phi_{R}\phi_{R_4}" ] \arrow[d, dashed,"\mu" ]
& K_0^{\gr}(L_\K(R_2)) \arrow[d, "T=\id" ] \\
K_0^{\gr}(L_\K(E^1_{17})) \arrow[r, "\phi_{R}\phi_{R_3}"]
& K_0^{\gr}(L_\K(R_2))
\end{tikzcd}
\]

In order to show that $L_\K(E^1_7)$ and $L_\K(E^1_{17})$ are isomorphic, we use the unital in-splitting $I_{+}$  (Definition~\ref{unitalinspliII})  on the graph $F^1_2$ at $w_1$: Consider the partition at $w_1$ of two sets: one set is empty and the other set consists of both incoming edges, which produces the graph $E^1_{17}$. The  other partition is the maximal in-splitting at $w_1$, that, as mentioned, produces $E^1_7$. Now Eilers-Ruiz's unital inplitting guarantees that the algebras associated to these graphs are graded isomorphic.

{\bf 7.} Consider the graphs $E^1_8$ and $E^1_{12}$ below with the adjacency matrices  $A_{E^1_8}= \begin{bmatrix} 1 & 1 & 0 \\ 1 & 0 & 1 \\ 1 & 1 & 1 \\ \end{bmatrix}$ and $
A_{E^1_{12}}= \begin{bmatrix} 1 & 1 & 1 \\ 1 & 0 & 0 \\ 1 & 0 & 1 \\ \end{bmatrix}$, respectively.

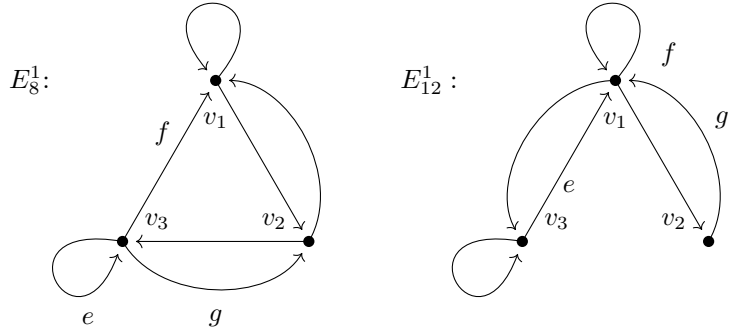
\begin{figure}[H]
\centering
\begin{tikzpicture}[scale=3.5]
\fill (0,0)  circle[radius=.6pt];
\fill (-.35,-.61)  circle[radius=.6pt];
\fill (.35,-.61)  circle[radius=.6pt];
\draw (0,-.15) node{$v_1$};
\draw (-.70, 0) node{$E^1_8$:};
\draw (-.35+.13,-.535) node{$v_3$};
\draw (.35-.13,-.535) node{$v_2$};
\draw (0,-.9) node{$g$};
\draw (-.20,-.2) node{$f$};
\draw (-.35-.13,-.9) node{$e$};
\draw[->, shorten >=5pt] (0,0) to (.35,-.61);
\draw[->, shorten >=5pt] (.35,-.61) to [in=0,out=60] (0,0);
\draw[->, shorten >=5pt] (.35,-.61) to (-.35,-.61);
\draw[->, shorten >=5pt] (-.35,-.61) to [in=-120,out=-60] (.35,-.61) ;
\draw[->, shorten >=5pt] (-.35,-.61) to (0,0);
\draw[->, shorten >=5pt] (0,0) to[in=130,out=50, loop] (0,0);
\draw[->, shorten >=5pt] (-.35,-.61) to[in=250,out=170, loop] (-.35,-.61);
\end{tikzpicture} 
\hspace{10pt}
\begin{tikzpicture}[scale=3.5]
\fill (0,0)  circle[radius=.6pt];
\fill (-.35,-.61)  circle[radius=.6pt];
\fill (.35,-.61)  circle[radius=.6pt];
\draw (0,-.15) node{$v_1$};
\draw (-.70, 0) node{$E^1_{12}:$};
\draw (-.35+.13,-.535) node{$v_3$};
\draw (.35-.13,-.535) node{$v_2$};
\draw (.2,.1) node{$f$};
\draw (-.175,-.4) node{$e$};
\draw (.4,-.15) node{$g$};
\draw[->, shorten >=5pt] (0,0) to (.35,-.61);
\draw[->, shorten >=5pt] (.35,-.61) to [in=0,out=60] (0,0);
\draw[->, shorten >=5pt] (-.35,-.61) to (0,0);
\draw[->, shorten >=5pt] (0,0) to [in=120,out=180
] (-.35,-.61) ;
\draw[->, shorten >=5pt] (0,0) to[in=130,out=50, loop] (0,0);
\draw[->, shorten >=5pt] (-.35,-.61) to[in=250,out=170, loop] (-.35,-.61);
\end{tikzpicture}
\end{figure}

Since $\det(E^1_8) = \det(E^1_{12})=-1$ and the graphs are  primitive,  Lemma~\ref{DL} and Corollary~\ref{halfspace} give:
\begin{align*}
K_0^{\gr}(L_\K(E^1_8)) &=  \mathbb{Z} ^3, \\
K_0^{\gr}(L_\K(E^1_8))^+ &=\left\{ {\bf u} \in \mathbb{Z}^3 \, | \, {\bf u} \cdot {\bf z} > 0 \right\}  \cup \, \{0\}, \text{  where } \bf{z} \text{  is approximately } \begin{bmatrix} 0.554958 \\
0.692021 \\
1 \\
\end{bmatrix} ,\\
{}^1(a,b,c)& =(a+b+c, a+c, b+c).
\end{align*}
whereas, 
\begin{align*}
K_0^{\gr}(L_\K(E^1_{12})) &=  \mathbb{Z} ^3, \\
K_0^{\gr}(L_\K(E^1_{12}))^+ &=\left\{ {\bf u} \in \mathbb{Z}^3 \, | \, {\bf u} \cdot {\bf z'} > 0 \right\}  \cup \, \{0\}, \text{  where } \bf{z'} \text{  is approximately } 
\begin{bmatrix} 
1.24698\\
 0.554958\\
 1 \\
\end{bmatrix} ,\\
{}^1(a,b,c)& =(a+b+c, a, a+c).
\end{align*}

The invertible matrix  $T = \begin{bmatrix}
    0 & 0 & 1\\
    -1 & 0 & 1 \\
    2 & 1 & -1 \\
\end{bmatrix}\in \GL_3(\mathbb Z)$ induces an order-preserving isomorphism 
\begin{equation}\label{hardbuteasy}
T \colon  K_0^{\gr}(L_\K(E^1_8)) \longrightarrow K_0^{\gr}(L_\K(E^1_{12})),
\end{equation}
  i.e., $A_{E^1_{12}}T=TA_{E^1_8}$ which is also pointed, meaning,  $\begin{bmatrix} 1 & 1 & 1 \end{bmatrix}T = \begin{bmatrix} 1 & 1 & 1 \end{bmatrix}$. Note that $T$  takes 
  Perron-Frobenius eigenvalue $\bf{z}$ of $E^1_8$ to a positive scalar multiple of the eigenvector $\bf{z}^\prime$ of $E^1_{12}$, thus preserving the positive cones. 

Therefore the matrices $A_{E^1_{8}}$ and $A_{E^1_{12}}$ are shift equivalent. We show that these matrices are indeed strong shift equivalent.  Starting from $E^1_8$, we partition the arrows coming out of $v_3$ into two sets $\{e,f \}$ and $\{g \}$ and perform the out-split. We obtain the graph $A$ below which coincides with the in-split of $E^{12}_1$ by partition the arrows coming into $v_1$ into two sets $\{f \}$ and $\{e, g \}$.
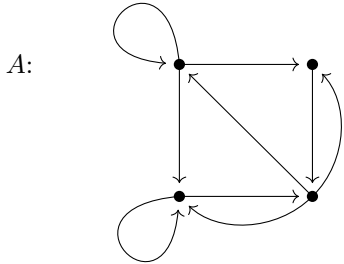
\begin{figure}[H]
\centering
\begin{tikzpicture}[scale=3.5]
\fill (0,0)  circle[radius=.6pt];
\draw (-.60, 0) node{$A$:};

\fill (.5,0)  circle[radius=.6pt];
\fill (0,-.5)  circle[radius=.6pt];
\fill (.5,-.5)  circle[radius=.6pt];
\draw[->, shorten >=5pt] (0,0) to (.5,0);
\draw[->, shorten >=5pt] (0,0) to (0,-.5);
\draw[->, shorten >=5pt] (0,-.5) to (.5,-.5);
\draw[->, shorten >=5pt] (.5,0) to (.5,-.5);
\draw[->, shorten >=5pt] (.5,-.5) to (0,0);
\draw[->, shorten >=5pt] (.5,-.5) to[in=-45, out=-135] (0,-.5);
\draw[->, shorten >=5pt] (.5,-.5) to[in=-45, out=45] (.5,0);
\draw[->, shorten >=5pt] (0,0) to[in=175,out=95, loop] (0,0);
\draw[->, shorten >=5pt] (0,-.5) to[in=265,out=185, loop] (0,-.5);
\end{tikzpicture}
\end{figure}

In fact we can say more: 

- $E^1_{8}$ is elementary shift equivalent to $E^1_{12}$ as $E^1_{8} = R_iS_i$ and $E^1_{12}=S_iR_i$, for $i=1, 2$, where,
$$R_1= \begin{bmatrix}
    0 & 0 & 1 \\
    1 & 0 & 0 \\
    1 & 1 & 0 \\
    \end{bmatrix} \text{ and } S_1= \begin{bmatrix}
    1 & 0 & 1 \\
    0 & 1 & 0 \\
    1 & 1 & 0 \\
\end{bmatrix}$$
$$R_2= \begin{bmatrix}
    1 & 0 & 0 \\
    0 & 1 & 1 \\
    1 & 0 & 1 \\
    \end{bmatrix} \text{ and } S_2= \begin{bmatrix}
    1 & 1 & 0 \\
    1 & 0 & 0 \\
    0 & 0 & 0 \\
\end{bmatrix}.$$

The following path of moves from $E^1_8$ to $E^1_{12}$ was found by S\o ren Eilers  using a computer search. Since both out-split and generalised in-split preserve Leavitt path algebras up to graded isomorphism, it follows that $L_\K(E^1_8) \cong_{\gr} L_\K(E^1_{12})$.

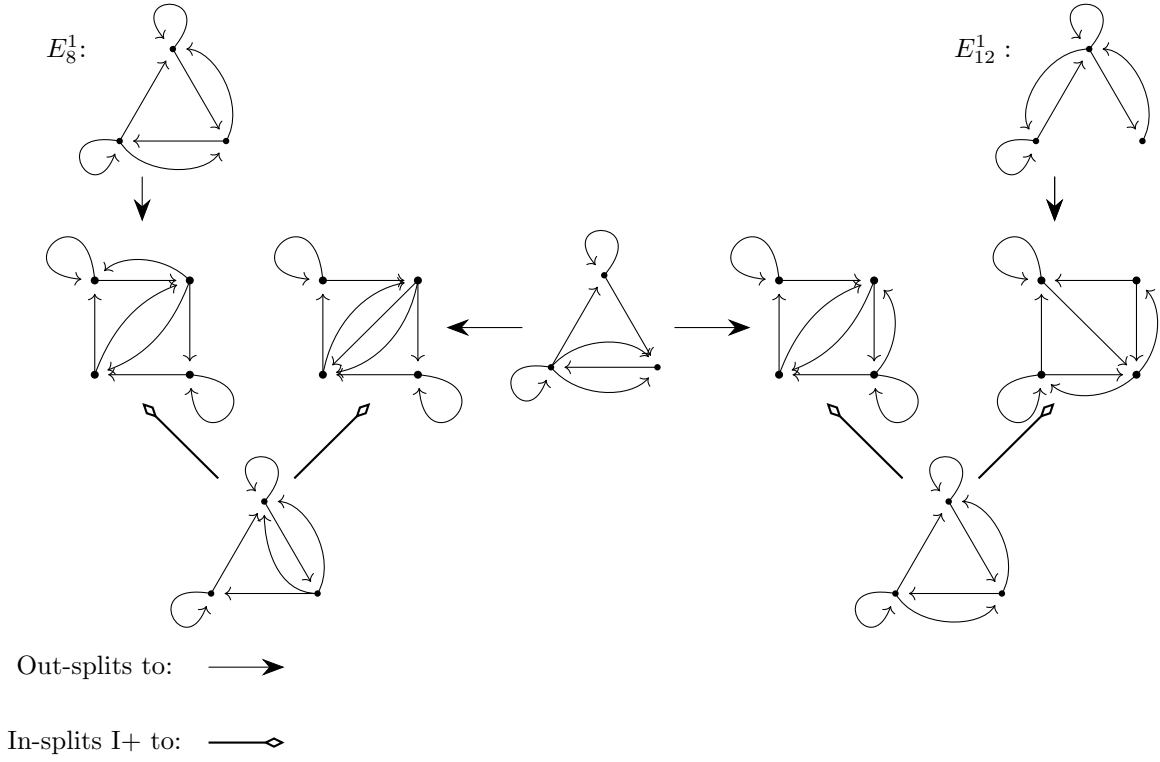
\begin{figure}[H]
\centering
    \begin{tikzpicture}
       \draw (0,0) node(E){\usebox{\E}};
       \draw (12,0) node(F){\usebox{\F}};
       \draw (0, -3) node(Cone){\usebox{\Cone}};
       \draw (3, -3) node(Ctwo){\usebox{\Ctwo}};
       \draw (5.6, -3) node(Cthree){\usebox{\Cthree}};
       \draw (9, -3) node(Cfour){\usebox{\Cfour}};
       \draw (12, -3) node(Cfive){\usebox{\Cfive}};
       \draw (1.25, -6) node(Done){\usebox{\Done}};       
       \draw (10.25, -6) node(Dtwo){\usebox{\Dtwo}};
       
\draw [-{Stealth[length=3mm]}] (Cone) -- (E);
\draw [-{Stealth[length=3mm]}] (5,-3)--(4,-3);
\draw [-{Stealth[length=3mm]}] (7,-3)--(8,-3);
\draw [-{Stealth[length=3mm]}] (Cfive) -- (F);
\draw [-{Stealth[length=3mm]}] (Cfive) -- (F);
\draw[-{Diamond[open]}, thick] (1,-5)--(0,-4);
\draw[-{Diamond[open]}, thick] (2,-5)--(3,-4);
\draw[-{Diamond[open]}, thick] (10,-5)--(9,-4);
\draw[-{Diamond[open]}, thick] (11,-5)--(12,-4);

\draw (0,-8) node(key){\usebox{\key}};
    \end{tikzpicture}
  
\end{figure}

We conclude this case by observing that $L_\K(E^1_{12}) \cong \M_2(L_\K(E^1_{12}))(0,1)$.  We know that $L_\K(E^1_8) \cong_{\gr} L_\K(A)$, as $A$ is the out-split of $E^1_8$.  On the other hand, by Example~\ref{gettingclose343}, $L_\K(A) \cong_{\gr} L_\K(F^4_1)$. One more maximum  out-split on $F^4_1$ gives us $F^3_1$ and thus  $L_\K(F^4_1) \cong_{\gr} L_\K(F^3_1)$. Since $L_\K(F^3_{1}) \cong \M_2(L_\K(E^1_{12}))(0,1)$, we obtain 
$$L_\K(E^1_{12})\cong_{\gr} L_\K(E^1_8) \cong \M_2(L_\K(E^1_{12}))(0,1).$$ 

\begin{figure}[H]
\begin{tikzpicture}[scale=3.5]
\fill (0,0)  circle[radius=.6pt];
\fill (-.35,-.61)  circle[radius=.6pt];
\fill (.35,-.61)  circle[radius=.6pt];
\fill (0,-.61)  circle[radius=.6pt];
\draw (-.70, 0) node{$F^4_{1}:$};
\draw[->, shorten >=5pt] (0,0) to (.35,-.61);
\draw[->, shorten >=5pt] (.35,-.61) to [in=0,out=60] (0,0);
\draw[->, shorten >=5pt] (-.35,-.61) to (0,0);
\draw[->, shorten >=5pt] (0,0) to [in=120,out=180
] (-.35,-.61) ;
\draw[->, shorten >=5pt] (0,0) to[in=130,out=50, loop] (0,0);
\draw[->, shorten >=5pt] (-.35,-.61) to[in=250,out=170, loop] (-.35,-.61);
\draw[->, shorten >=5pt] (0,-.61) to (-.35,-.61);
\draw[->, shorten >=5pt] (0,-.61) to (.35,-.61);
\draw[->, shorten >=5pt] (0,-.61) to (0,0);
\end{tikzpicture}
\hspace{10pt}
\begin{tikzpicture}[scale=3.5]
\fill (0,0)  circle[radius=.6pt];
\fill (-.35,-.61)  circle[radius=.6pt];
\fill (.35,-.61)  circle[radius=.6pt];
\fill (.10,-.61)  circle[radius=.6pt];
\fill (-.10,-.61)  circle[radius=.6pt];
\fill (0,-.34)  circle[radius=.6pt];
\draw (-.70, 0) node{$F^3_{1}:$};
\draw[->, shorten >=5pt] (0,0) to (.35,-.61);
\draw[->, shorten >=5pt] (.35,-.61) to [in=0,out=60] (0,0);
\draw[->, shorten >=5pt] (-.35,-.61) to (0,0);
\draw[->, shorten >=5pt] (0,0) to [in=120,out=180
] (-.35,-.61) ;
\draw[->, shorten >=5pt] (0,0) to[in=130,out=50, loop] (0,0);
\draw[->, shorten >=5pt] (-.35,-.61) to[in=250,out=170, loop] (-.35,-.61);
\draw[->, shorten >=5pt] (-.10,-.61) to (-.35,-.61);
\draw[->, shorten >=5pt] (.10,-.61) to (.35,-.61);
\draw[->, shorten >=5pt] (0,-.34) to (0,0);
\end{tikzpicture}
\end{figure}

{\bf 8.} Consider the graphs $E^1_9$ below with the adjacency matrix $A_{E^1_9} = \begin{bmatrix} 1 & 1 & 0 \\ 1 & 1 & 1 \\ 1 & 1 & 1 \\ \end{bmatrix}.$

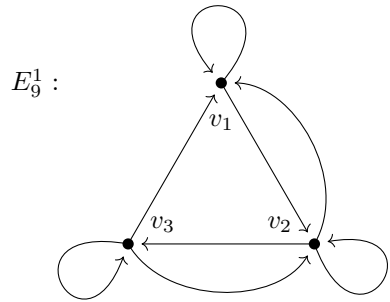
\begin{figure}[H]
\centering
\begin{tikzpicture}[scale=3.5]
\fill (0,0)  circle[radius=.6pt];
\fill (-.35,-.61)  circle[radius=.6pt];
\fill (.35,-.61)  circle[radius=.6pt];
\draw (0,-.15) node{$v_1$};
\draw (-.70, 0) node{$E^1_9:$};
\draw (-.35+.13,-.535) node{$v_3$};
\draw (.35-.13,-.535) node{$v_2$};
\draw[->, shorten >=5pt] (0,0) to (.35,-.61);
\draw[->, shorten >=5pt] (.35,-.61) to [in=0,out=60] (0,0);
\draw[->, shorten >=5pt] (.35,-.61) to (-.35,-.61);
\draw[->, shorten >=5pt] (-.35,-.61) to [in=-120,out=-60] (.35,-.61) ;
\draw[->, shorten >=5pt] (-.35,-.61) to (0,0);
\draw[->, shorten >=5pt] (0,0) to[in=130,out=50, loop] (0,0);
\draw[->, shorten >=5pt] (-.35,-.61) to[in=250,out=170, loop] (-.35,-.61);
\draw[->, shorten >=5pt] (.35,-.61) to[in=10,out=-70, loop] (-.35,-.61);
\end{tikzpicture}
\end{figure}

We have $\det(E^1_9)=0$. However the graph $E^1_9$  can be obtained by out-splitting the graph $F^1_3$ below at $w_2$, with the partition $\{ f\}$ and $\{e, g \}$.

\begin{figure}[H]
\centering
\begin{tikzpicture}[scale=3.5]
\fill (0,0)  circle[radius=.6pt];
\draw (0,-.2) node{$w_1$};
\draw (-0.60 ,0.2) node{$F^1_3:$};
\fill (.5,0)  circle[radius=.6pt];
\draw (.45,-.2) node{$w_2$};
\draw[->, shorten >=5pt] (0,0) to[in=135, out=-135, loop] (0,0);
\draw[->, shorten >=5pt] (.5,0) to[in=-90,out=0, loop] (.5,0);
\draw[->, shorten >=5pt] (.5,0) to[in=90,out=0, loop] (.5,0);
\draw[->, shorten >=5pt] (0,0) to[in=120,out=60] (.5,0);
\draw[->, shorten >=5pt] (.5,0) to[in=-60,out=-120] (0,0);
\draw (.25,-.2) node{$e$};
\draw (.8,.15) node{$f$};
\draw (.8,-.15) node{$g$};
\end{tikzpicture}
\end{figure}

Since $\det(F^1_3) =1$  and $F^1_3$ is primitive,  Lemma~\ref{DL} and Corollary~\ref{halfspace} give:
\begin{align*}
K_0^{\gr}(L_\K(F^1_3)) & =  \mathbb{Z} ^2,\\
K_0^{\gr}(L_\K(F^1_3))^+ & =\left\{ {\bf u} \in \mathbb{Z}^2 \, | \, {\bf u} \cdot {\bf z} > 0 \right\}  \cup \, \{0\}, \text{ where }\bf{z} \text{ is } \begin{bmatrix}
 \frac{1}{2} \left(\sqrt{5}-1\right) \\
1 \\
\end{bmatrix} ,\\
{}^1(a,b)& =(a+2b, a+b).
\end{align*}

As $E^1_9$ is an out-split of $F^1_3$, its graded $K$-theory coincides with that of $F^1_3$.

{\bf 9.} Consider the graph $E^1_{10}$ below with the adjacency matrix  $A_{E^1_{10}} = \begin{bmatrix} 0 & 1 & 1 \\ 1 & 0 & 0 \\ 1 & 0 & 0 \\ \end{bmatrix}$. 

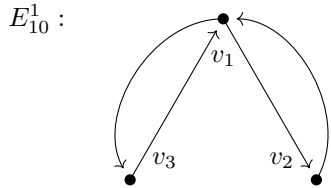
\begin{figure}[H]
\centering
\begin{tikzpicture}[scale=3.5]
\fill (0,0)  circle[radius=.6pt];
\fill (-.35,-.61)  circle[radius=.6pt];
\fill (.35,-.61)  circle[radius=.6pt];
\draw (0,-.15) node{$v_1$};
\draw (-.70, 0) node{$E^1_{10}:$};
\draw (-.35+.13,-.535) node{$v_3$};
\draw (.35-.13,-.535) node{$v_2$};
\draw[->, shorten >=5pt] (0,0) to (.35,-.61);
\draw[->, shorten >=5pt] (.35,-.61) to [in=0,out=60] (0,0);
\draw[->, shorten >=5pt] (-.35,-.61) to (0,0);
\draw[->, shorten >=5pt] (0,0) to [in=120,out=180
] (-.35,-.61) ;
\end{tikzpicture}
\end{figure}

We have $\det(E^1_{10})=0$. However, the graph $E^1_{10}$  can be obtained by the (maximal) out-splitting of the graph $F^1_4$ below. 
\begin{figure}[H]
\centering
\begin{tikzpicture}[scale=3.5]
\draw (-0.50 ,0) node{$F^1_4:$};

\fill (0,0)  circle[radius=.6pt];
\draw (0,-.2) node{$w_1$};
\fill (.5,0)  circle[radius=.6pt];
\draw (.5,-.2) node{$w_2$};
\draw[->, shorten >=5pt] (0,0) to[in=120,out=65] (.5,0);
\draw[->] (0,.05) to[in=90,out=90] (.5,.05);
\draw (0,.05) to (0,0);
\draw[->, shorten >=5pt] (.5,0) to[in=-60,out=-120] (0,0);
\end{tikzpicture}
\end{figure}

As $E^1_{10}$ is an out-split of $F^1_4$, its graded $K$-theory coincides with that of $F^1_4$. 

We have $A_{F^1_4}= \begin{bmatrix} 0 & 1 \\ 1/2 & 0 \\    
    \end{bmatrix}$. Since $A:=A_{F^1_4}^2=\begin{bmatrix} 2 & 0 \\ 0 & 2 \\    
    \end{bmatrix}$
    the dimension group can be calculated as the colimit of the following inductive system (see~(\ref{thuluncht}))
\begin{equation*}
 \xymatrix{
\mathbb{Z}^2 \ar[r]^A \ar[dr]^{\id}& \mathbb{Z}^2 \ar[r]^A \ar[d]^{A^{-1}}&\mathbb{Z}^2 \ar[r]^A \ar[dl]^{A^{-1}}& \cdots\\
&    \mathbb{Z} [ \frac{1}{2}]  \oplus  \mathbb{Z} [ \frac{1}{2}] & 
     }
 \end{equation*}

Thus we obtain 
\begin{align*}
K_0^{\gr}(L(F^1_4)) & =  \left( \mathbb{Z} \left[ \frac{1}{2} \right] \right)^2, \\
K_0^{\gr}(L(F^1_4))^+ & =  \left( \mathbb{N} \left[ \frac{1}{2} \right] \right)^2,\\
{}^1(a,b) & =(b, 2a). 
\end{align*}

{\bf 10.} Consider the graphs $E^1_{11}$ below with the adjacency matrix $E^1_{11}= \begin{bmatrix} 0 & 1 & 1 \\ 1 & 0 & 0 \\ 1 & 0 & 1 \\ \end{bmatrix}.$

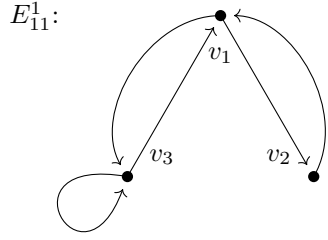
\begin{figure}[H]
\centering
\begin{tikzpicture}[scale=3.5]
\fill (0,0)  circle[radius=.6pt];
\fill (-.35,-.61)  circle[radius=.6pt];
\fill (.35,-.61)  circle[radius=.6pt];
\draw (0,-.15) node{$v_1$};
\draw (-.70, 0) node{$E^1_{11}$:};
\draw (-.35+.13,-.535) node{$v_3$};
\draw (.35-.13,-.535) node{$v_2$};
\draw[->, shorten >=5pt] (0,0) to (.35,-.61);
\draw[->, shorten >=5pt] (.35,-.61) to [in=0,out=60] (0,0);
\draw[->, shorten >=5pt] (-.35,-.61) to (0,0);
\draw[->, shorten >=5pt] (0,0) to [in=120,out=180
] (-.35,-.61) ;
\draw[->, shorten >=5pt] (-.35,-.61) to[in=250,out=170, loop] (-.35,-.61);
\end{tikzpicture}
\end{figure}

Since $\det(E^1_{11})=-1$ and the graph is  primitive,  Lemma~\ref{DL} and Corollary~\ref{halfspace} give:
\begin{align*}
K_0^{\gr}(L_\K(E^1_{11})) & =   \mathbb{Z}^3, \\
K_0^{\gr}(L_\K(E^1_11))^+ & =\left\{ {\bf u} \in \mathbb{Z}^3 \, | \, {\bf u} \cdot {\bf z} > 0 \right\}  \cup \, \{0\}, \text{ where } \bf{z} \text{  is approximately } \left[
\begin{array}{c}
 0.801938\\
  0.445042\\
  1 \\
\end{array}
\right],\\
{}^1(a,b,c)& =(b+c,a, a+c).
\end{align*}

{\bf 11.} Consider the graphs $E^1_{18}$ below with the adjacency matrix $E^1_{18}= \begin{bmatrix} 1 & 1 & 0 \\ 0 & 1 & 1 \\ 1 & 1 & 1 \\ \end{bmatrix}.$

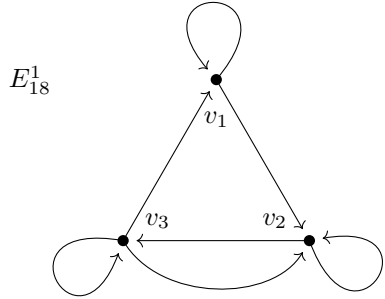
\begin{figure}[H]
\centering
\begin{tikzpicture}[scale=3.5]
\fill (0,0)  circle[radius=.6pt];
\fill (-.35,-.61)  circle[radius=.6pt];
\fill (.35,-.61)  circle[radius=.6pt];
\draw (0,-.15) node{$v_1$};
\draw (-.70, 0) node{$E^1_{18}$};
\draw (-.35+.13,-.535) node{$v_3$};
\draw (.35-.13,-.535) node{$v_2$};
\draw[->, shorten >=5pt] (0,0) to (.35,-.61);
\draw[->, shorten >=5pt] (.35,-.61) to (-.35,-.61);
\draw[->, shorten >=5pt] (-.35,-.61) to [in=-120,out=-60] (.35,-.61) ;
\draw[->, shorten >=5pt] (-.35,-.61) to (0,0);
\draw[->, shorten >=5pt] (0,0) to[in=130,out=50, loop] (0,0);
\draw[->, shorten >=5pt] (-.35,-.61) to[in=250,out=170, loop] (-.35,-.61);
\draw[->, shorten >=5pt] (.35,-.61) to[in=10,out=-70, loop] (-.35,-.61);
\end{tikzpicture}
\end{figure}

Since $\det(E^1_{18})=1$ and the graph is  primitive,  Lemma~\ref{DL} and Corollary~\ref{halfspace} give:
\begin{align*}
K_0^{\gr}(L_\K(E^1_{18})) & =   \mathbb{Z}^3, \\
K_0^{\gr}(L_\K(E^1_{18}))^+ & =\left\{ {\bf u} \in \mathbb{Z}^3 \, | \, {\bf u} \cdot {\bf z} > 0 \right\}  \cup \, \{0\}, \text{ where } \bf{z} \text{ is } \left[
\begin{array}{c}
 \frac{1}{3} \left(1-5 \sqrt[3]{\frac{2}{11+3
   \sqrt{69}}}+\sqrt[3]{\frac{1}{2} \left(11+3
   \sqrt{69}\right)}\right) \\ 
   \frac{1}{3}
   \left(-1+\sqrt[3]{\frac{1}{2} \left(25-3
   \sqrt{69}\right)}
   +\sqrt[3]{\frac{1}{2} \left(25+3
   \sqrt{69}\right)}\right) \\
   1 \\
\end{array}
\right],\\
{}^1(a,b,c) & =(a+c,a+b+c, b+c).
\end{align*}

\vspace{1truecm}

{\bf 12.} Consider the graphs $E^2_{1}$ below with the adjacency matrix $A_{E^2_1}= \begin{bmatrix} 0 & 1 & 0 \\ 0 & 1 & 1 \\ 1 & 1 & 1 \\ \end{bmatrix}$.

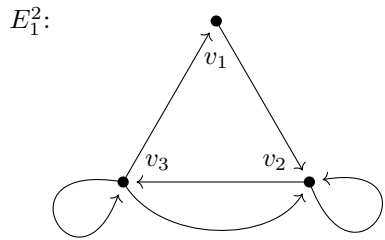
\begin{figure}[H]
\centering
\begin{tikzpicture}[scale=3.5]
\fill (0,0)  circle[radius=.6pt];
\fill (-.35,-.61)  circle[radius=.6pt];
\fill (.35,-.61)  circle[radius=.6pt];
\draw (0,-.15) node{$v_1$};
\draw (-.70, 0) node{$E^2_1$:};
\draw (-.35+.13,-.535) node{$v_3$};
\draw (.35-.13,-.535) node{$v_2$};
\draw[->, shorten >=5pt] (0,0) to (.35,-.61);
\draw[->, shorten >=5pt] (.35,-.61) to (-.35,-.61);
\draw[->, shorten >=5pt] (-.35,-.61) to [in=-120,out=-60] (.35,-.61) ;
\draw[->, shorten >=5pt] (-.35,-.61) to (0,0);
\draw[->, shorten >=5pt] (-.35,-.61) to[in=250,out=170, loop] (-.35,-.61);
\draw[->, shorten >=5pt] (.35,-.61) to[in=10,out=-70, loop] (-.35,-.61);
\end{tikzpicture}
\end{figure}

Since $\det(E^2_{1})=1$ and the graph is  primitive,  Lemma~\ref{DL} and Corollary~\ref{halfspace} give:
\begin{align*}
K_0^{\gr}(L_\K(E^2_1)) & =   \mathbb{Z}^3, \\
K_0^{\gr}(L_\K(E^2_1))^+ & =\left\{ {\bf u} \in \mathbb{Z}^3 \, | \, {\bf u} \cdot {\bf z} > 0 \right\}  \cup \, \{0\}, \text{ where } \bf{z}\text{ is } \left[
\begin{array}{c}
 \frac{1}{6} \left(-1-\frac{11}{\sqrt[3]{71+6
   \sqrt{177}}}+\sqrt[3]{71+6 \sqrt{177}}\right) \\ \frac{1}{6}
   \left(-1+\sqrt[3]{44-3 \sqrt{177}}+\sqrt[3]{44+3
   \sqrt{177}}\right) \\ 1 \\
\end{array}
\right],\\
{}^1(a,b,c) & =(c,a+b+c, b+c).
\end{align*}

{\bf 13.} Consider the graphs $E^2_{2}$ below with the adjacency matrix $A_{E^2_2}= \begin{bmatrix} 0 & 1 & 0 \\ 1 & 0 & 1 \\ 1 & 1 & 0 \\ \end{bmatrix}$.

\begin{figure}[H]
\centering
\begin{tikzpicture}[scale=3.5]
\fill (0,0)  circle[radius=.6pt];
\fill (-.35,-.61)  circle[radius=.6pt];
\fill (.35,-.61)  circle[radius=.6pt];
\draw (0,-.15) node{$v_1$};
\draw (-.70, 0) node{$E^2_2$};
\draw (-.35+.13,-.535) node{$v_3$};
\draw (.35-.13,-.535) node{$v_2$};
\draw[->, shorten >=5pt] (0,0) to (.35,-.61);
\draw[->, shorten >=5pt] (.35,-.61) to [in=0,out=60] (0,0);
\draw[->, shorten >=5pt] (.35,-.61) to (-.35,-.61);
\draw[->, shorten >=5pt] (-.35,-.61) to [in=-120,out=-60] (.35,-.61) ;
\draw[->, shorten >=5pt] (-.35,-.61) to (0,0);
\end{tikzpicture}
\end{figure}

Since $\det(E^2_{2})=1$ and the graph is  primitive,  Lemma~\ref{DL} and Corollary~\ref{halfspace} give:
\begin{align*}
K_0^{\gr}(L_\K(E^2_2))&=   \mathbb{Z}^3, \\
K_0^{\gr}(L_\K(E^2_2))^+&=\left\{ {\bf u} \in \mathbb{Z}^3 \, | \, {\bf u} \cdot {\bf z} > 0 \right\}  \cup \, \{0\}, \text{ where } \bf{z} \text{ is } \left[
\begin{array}{c}
 \frac{1}{2} \left(\sqrt{5}-1\right) \\ 1 \\ 1 \\
\end{array}
\right],\\
{}^1(a,b,c)&=(b+c,a+c, b).
\end{align*}

{\bf 14.} Consider the graphs $E^2_{3}, E^2_{4}, E^3_{4}$ below with the adjacency matrices 
$A_{E^2_3}= \begin{bmatrix} 0 & 1 & 0 \\ 1 & 0 & 1 \\ 1 & 1 & 1 \\ \end{bmatrix},   A_{E^2_4}= \begin{bmatrix} 1 & 1 & 0 \\ 1 & 0 & 1 \\ 1 & 1 & 0 \\ \end{bmatrix}$ and 
$A_{E^3_4} = \begin{bmatrix} 1 & 1 & 1 \\ 1 & 0 & 0 \\ 1 & 0 & 0 \\ \end{bmatrix}$, respectively.

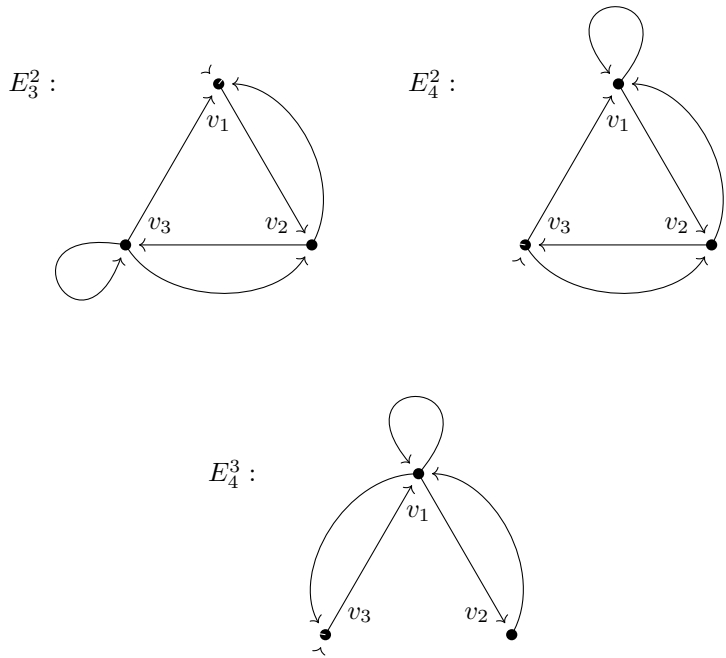
\begin{figure}[H]
\centering
\begin{tikzpicture}[scale=3.5]
\fill (0,0)  circle[radius=.6pt];
\fill (-.35,-.61)  circle[radius=.6pt];
\fill (.35,-.61)  circle[radius=.6pt];
\draw (0,-.15) node{$v_1$};
\draw (-.70, 0) node{$E^2_3:$};
\draw (-.35+.13,-.535) node{$v_3$};
\draw (.35-.13,-.535) node{$v_2$};
\draw[->, shorten >=5pt] (0,0) to (.35,-.61);
\draw[->, shorten >=5pt] (.35,-.61) to [in=0,out=60] (0,0);
\draw[->, shorten >=5pt] (.35,-.61) to (-.35,-.61);
\draw[->, shorten >=5pt] (-.35,-.61) to [in=-120,out=-60] (.35,-.61) ;
\draw[->, shorten >=5pt] (-.35,-.61) to (0,0);
\draw[draw=white, ->, shorten >=5pt] (0,0) to[in=130,out=50, loop] (0,0);
\draw[->, shorten >=5pt] (-.35,-.61) to[in=250,out=170, loop] (-.35,-.61);
\end{tikzpicture}
\hspace{10pt}
\begin{tikzpicture}[scale=3.5]
\fill (0,0)  circle[radius=.6pt];
\fill (-.35,-.61)  circle[radius=.6pt];
\fill (.35,-.61)  circle[radius=.6pt];
\draw (0,-.15) node{$v_1$};
\draw (-.70, 0) node{$E^2_4:$};
\draw (-.35+.13,-.535) node{$v_3$};
\draw (.35-.13,-.535) node{$v_2$};
\draw[->, shorten >=5pt] (0,0) to (.35,-.61);
\draw[->, shorten >=5pt] (.35,-.61) to [in=0,out=60] (0,0);
\draw[->, shorten >=5pt] (.35,-.61) to (-.35,-.61);
\draw[->, shorten >=5pt] (-.35,-.61) to [in=-120,out=-60] (.35,-.61) ;
\draw[->, shorten >=5pt] (-.35,-.61) to (0,0);
\draw[->, shorten >=5pt] (0,0) to[in=130,out=50, loop] (0,0);
\draw[draw=white, ->, shorten >=5pt] (-.35,-.61) to[in=250,out=170, loop] (-.35,-.61);
\end{tikzpicture}
\\
\begin{tikzpicture}[scale=3.5]
\fill (0,0)  circle[radius=.6pt];
\fill (-.35,-.61)  circle[radius=.6pt];
\fill (.35,-.61)  circle[radius=.6pt];
\draw (0,-.15) node{$v_1$};
\draw (-.70, 0) node{$E^3_4:$};
\draw (-.35+.13,-.535) node{$v_3$};
\draw (.35-.13,-.535) node{$v_2$};
\draw[->, shorten >=5pt] (0,0) to (.35,-.61);
\draw[->, shorten >=5pt] (.35,-.61) to [in=0,out=60] (0,0);
\draw[->, shorten >=5pt] (-.35,-.61) to (0,0);
\draw[->, shorten >=5pt] (0,0) to [in=120,out=180
] (-.35,-.61) ;
\draw[->, shorten >=5pt] (0,0) to[in=130,out=50, loop] (0,0);
\draw[draw=white, ->, shorten >=5pt] (-.35,-.61) to[in=250,out=170, loop] (-.35,-.61);
\end{tikzpicture}
\end{figure}

These graphs are the result of an out-split or in-split of the graphs $F^2_1$ and $F^2_2$ below:  

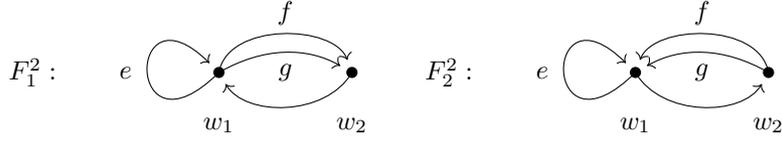
\begin{figure}[H]
\centering
\begin{tikzpicture}[scale=3.5]
\draw (-.7,0) node{$F^2_1:$};
\fill (0,0)  circle[radius=.6pt];
\draw (0,-.2) node{$w_1$};
\fill (.5,0)  circle[radius=.6pt];
\draw (.5,-.2) node{$w_2$};
\draw[->, shorten >=5pt] (0,0) to[in=135, out=-135, loop] (0,0);
\draw (-.35,0) node{$e$};
\draw (.25,.225) node{$f$};
\draw (.25,0) node{$g$};
\draw[->, shorten >=5pt] (0,0) to[in=110,out=80] (.5,0);
\draw[->, shorten >=5pt] (0,0) to[in=150,out=30](.5,0);
\draw[->, shorten >=5pt] (.5,0) to[in=-60,out=-120] (0,0);
\end{tikzpicture}
\,\,\,\,\,
\begin{tikzpicture}[scale=3.5]
\draw (-.7,0) node{$F^2_2:$};
\fill (0,0)  circle[radius=.6pt];
\draw (0,-.2) node{$w_1$};
\fill (.5,0)  circle[radius=.6pt];
\draw (.5,-.2) node{$w_2$};
\draw[->, shorten >=5pt] (0,0) to[in=135, out=-135, loop] (0,0);
\draw[->, shorten >=5pt] (.5,0) to[out=100,in=70] (0,0);
\draw[->, shorten >=5pt] (.5,0) to[out=150,in=30](0,0);
\draw[->, shorten >=5pt] (0,0) to[in=-120,out=-60] (.5,0);
\draw (-.35,0) node{$e$};
\draw (.25,.225) node{$f$};
\draw (.25,0) node{$g$};
\end{tikzpicture}
\end{figure}

\phantom{a} 
\begin{itemize}
    \item[-]$E^2_3$  is obtained from $F^2_1$ by an out-split at $w_1$ with partitions $\{e,f \}$ and $\{ g \}$.
    \item[-]$E^2_4$ is obtained from $F^2_2$ by an in-split at $w_1$ with partitions $\{e,f \}$ and $\{ g \}$.
    \item[-]$E^3_{4}$ is obtained from $F^2_1$ by a maximal in-split at $w_2$ (or a maximal out-split on $F^2_2$ at $w_2$). 
\end{itemize}

First note that $F^2_1$ and $F^2_2$ are elementary shift equivalent and thus the graded $K$-theory of the Leavitt path algebras of their associated graphs coincide:

- $F^2_{1}$ is elementary shift equivalent to $F^2_{2}$ as $F^2_{1} = RS$ and $F^2_{2}=SR$, where, 
$$R= \begin{bmatrix}
    1 & 1  \\
    1 & 0  \\
   \end{bmatrix} \text{  and } 
   S= \begin{bmatrix}
    1 & 0  \\
    0 & 2  \\
 \end{bmatrix}.$$

We have $\det(F^2_{1})=-2$. In order to calculate $K^{\gr}_0(L(F^2_1))$, we will use the alternative description of dimension group from Corollary~\ref{alternativeformdel}. 

\begin{prop}
  We have   $\Delta_{A_{F^2_1}}= \mathbb{Z}\left[ \frac{1}{2} \right] \begin{bmatrix} 1 & 1  \end{bmatrix} \oplus 0 \times \mathbb{Z}$.
\end{prop}

\begin{proof}
First note that $A_{F^2_1}^{-1}=\frac{1}{2}\begin{bmatrix}
 0 & 2 \\
 1 & -1 \\
\end{bmatrix}.$
Consider 
$$v=\frac{a}{2^k}\begin{bmatrix} 1 & 1 \end{bmatrix} + \begin{bmatrix} 0 & b \end{bmatrix} \in \mathbb{Z}\left[ \frac{1}{2} \right] \begin{bmatrix} 1 & 1 \end{bmatrix} \oplus 0 \times \mathbb{Z},$$ where $a,b \in \mathbb{Z}$ and $k\in \mathbb N$.  Since  
$$\begin{bmatrix} 1 & 1  \end{bmatrix} \begin{bmatrix}
 0 & 2 \\
 1 & -1 \\
\end{bmatrix} = \begin{bmatrix} 1 & 1  \end{bmatrix},$$ we have  
\begin{align*} 
&\frac{a}{2^k} \begin{bmatrix} 1 & 1  \end{bmatrix} A_{F^2_1}^k  
= \frac{a}{2^k} \begin{bmatrix} 1 & 1  \end{bmatrix} \begin{bmatrix}
 0 & 2 \\
 1 & -1 \\
\end{bmatrix}^k A_{F^2_1}^k \\ 
= & a \begin{bmatrix} 1 & 1  \end{bmatrix} \left( \frac{1}{2} \begin{bmatrix}
 0 & 2 \\
 1 & -1 \\
\end{bmatrix} \right)^k A_{F^2_1}^k 
= a \begin{bmatrix} 1 & 1 \end{bmatrix} A_{F^2_1}^{-k} A_{F^2_1}^k \\ 
=& a \begin{bmatrix} 1 & 1 \end{bmatrix}. 
\end{align*}
Thus $\frac{a}{2^k}\begin{bmatrix} 1 & 1 \end{bmatrix}$ is an element of $\Delta_{A_{F^2_1}}$. The vector $ \begin{bmatrix} 0 & b \end{bmatrix}$ is also an element of $\Delta_{A_{F^2_1}}$ as $\mathbb{Z}^2 \subseteq \Delta_{A_{F^2_1}}$.
Hence $$\mathbb{Z} \left[ \frac{1}{2} \right]\begin{bmatrix} 1 & 1 \end{bmatrix} \oplus 0 \times \mathbb{Z} \subseteq \Delta_{A_{F^2_1}}.$$

For the reverse inclusion, let $v \in \Delta_{A_{F^2_1}}$. Then by Corollary~\ref{alternativeformdel}, $v \in \mathbb{Z}^3A_{F^2_1}^{-k}$, for some $k \geq 0$. We will prove by induction that $v \in \mathbb{Z} \begin{bmatrix} 1 & 1  \end{bmatrix} \oplus 0 \times \mathbb{Z}$.

When $k=0$, $v \in \mathbb{Z}^2 \subseteq \mathbb{Z} \begin{bmatrix} 1 & 1  \end{bmatrix} \oplus 0 \times \mathbb{Z} \subseteq \mathbb{Z} \left[ \frac{1}{2} \right]\begin{bmatrix} 1 & 1 \end{bmatrix} \oplus 0 \times \mathbb{Z}$.

When $k=1$,  we can write $vA_{F^2_1}=a\begin{bmatrix} 1 & 1  \end{bmatrix} + \begin{bmatrix} 0 & b \end{bmatrix}$, where $a,b \in \mathbb{Z}$. Thus 
\begin{align*}
v=&\frac{a}{2}\begin{bmatrix} 1 & 1 \end{bmatrix} + \frac{1}{2} \begin{bmatrix} b & -b \end{bmatrix}\\ =& \frac{a}{2}\begin{bmatrix} 1 & 1 \end{bmatrix} + \frac{1}{2} \begin{bmatrix} b & b \end{bmatrix} + \frac{1}{2}  \begin{bmatrix} 0 & -2b\end{bmatrix} \\
=& \frac{a+b}{2}\begin{bmatrix} 1 & 1  \end{bmatrix} + \frac{1}{2} \begin{bmatrix} 0 & -2b\end{bmatrix} \\ =& \frac{a+b}{2}\begin{bmatrix} 1 & 1  \end{bmatrix} + \begin{bmatrix} 0 & -b\end{bmatrix}.
\end{align*} 

Hence $$v \in \frac{1}{2}\mathbb{Z} \begin{bmatrix} 1 & 1  \end{bmatrix} \oplus 0 \times \mathbb{Z} \subseteq \mathbb{Z} \left[ \frac{1}{2} \right]\begin{bmatrix} 1 & 1 \end{bmatrix} \oplus 0 \times \mathbb{Z}.$$
Assume now that  $vA_{F^2_1}^k \in \mathbb{Z}^2$ implies that $v \in \frac{1}{2^k}\mathbb{Z} \begin{bmatrix} 1 & 1 \end{bmatrix} \oplus 0 \times \mathbb{Z}$.
If $vA_{F^2_1}^{k+1} \in \mathbb{Z}^3$, then $vA_{F^2_1}=\frac{a}{2^k}\begin{bmatrix} 1 & 1 \end{bmatrix} + \begin{bmatrix} 0 & b \end{bmatrix}$, where $a,b \in \mathbb{Z}$. Thus \begin{align*}
v=&\frac{a}{2^{k+1}}\begin{bmatrix} 1 & 1 \end{bmatrix} + \frac{1}{2} \begin{bmatrix} b & -b \end{bmatrix}\\ =& \frac{a}{2^{k+1}}\begin{bmatrix} 1 & 1 \end{bmatrix} + \frac{1}{2} \begin{bmatrix} b & b \end{bmatrix} + \frac{1}{2} \begin{bmatrix} 0 & -2b\end{bmatrix} \\
=& \frac{a+2^k b}{2^{k+1}}\begin{bmatrix} 1 & 1 \end{bmatrix} + \frac{1}{2} \begin{bmatrix} 0 & -2b\end{bmatrix} \\ =& \frac{a+2^k b}{2^{k+1}}\begin{bmatrix} 1 & 1 \end{bmatrix} + \begin{bmatrix} 0 & -b\end{bmatrix}.
\end{align*} Hence $v \in \frac{1}{2^{k+1}}\mathbb{Z} \begin{bmatrix} 1 & 1 \end{bmatrix} \oplus 0 \times \mathbb{Z} \subseteq \mathbb{Z} \left[ \frac{1}{2} \right]\begin{bmatrix} 1 & 1  \end{bmatrix} \oplus 0 \times \mathbb{Z}$. So 

$$\Delta_{A_{F^2_1}} \subseteq \mathbb{Z} \left[ \frac{1}{2} \right]\begin{bmatrix} 1 & 1 \end{bmatrix} \oplus 0 \times \mathbb{Z},$$ and thus the equality follows. 
\end{proof}

Since $F^2_1$ is primitive, we can use Corollary~\ref{halfspace} for calculating the positive cone.  Therefore we have:

\begin{align*}
K_0^{\gr}(L(F^2_1)) &=  \mathbb{Z} \left[ \frac{1}{2} \right]\begin{bmatrix} 1 & 1 \end{bmatrix} \oplus 0 \times \mathbb{Z}, \\
K_0^{\gr}(L(F^2_1))^+ & =\left\{ {\bf u} \in \mathbb{Z} \left[ \frac{1}{2} \right]\begin{bmatrix} 1 & 1 \end{bmatrix} \oplus 0 \times \mathbb{Z} \,  | \, {\bf u} \cdot {\bf z} > 0 \right\}  \cup \, \{0\}, \text{ where }\bf{z} \text{ is } \left[
\begin{array}{c}
  2 \\ 1 \\
\end{array}
\right],\\
{}^1(a,b)&=(a+b,2a).
\end{align*}

Thus the graded $K$-theory of Leavitt path algebras of $E^2_{3}, E^2_{4}$ and  $E^3_{4}$ coincides with those of $F^2_{1}$ (and $F^2_{2}$) and all the associated adjacency matrices are all strongly shift equivalent. 

Note that although $F^2_1$ and $F^2_2$ are shift equivalent and thus there is an order-preserving module isomorphism $K^{\gr}_0(L_\K(F^2_1)) \cong K^{\gr}_0(L_\K(F^2_2))$ but the 
Perron-Frobenius vectors of these matrices are different. Furthermore, there is no pointed isomorphism between the graded $K$-theory of these algebras implying $L_\K(F^2_1)$ is not graded isomorphic to $L_\K(F^2_2)$. Indeed, any attempt to find an invertible matrix $T\in \GL_2( \mathbb{Z} \left[ \frac{1}{2} \right])$ such that $T F^2_1= F^2_2 T, $ preserving the positive cones and be pointed, i.e., $ \begin{bmatrix} 1 & 1  \\ \end{bmatrix}T =\begin{bmatrix} 1 & 1 \\ \end{bmatrix} $ yields an unsolvable system.

This is one the smallest examples that there is a graded Morita equivalence $\mbox{Gr-} L_\K(F^2_1) \rightarrow \mbox{Gr-} L_\K(F^2_2)$ but the algebras are not graded isomorphic.

Indeed we have:

- $E^2_{3}$ is elementary shift equivalent to $E^3_{4}$ as $E^2_{3} = RS$ and $E^3_{4}=SR$, where, 
$$R=\left[
\begin{array}{ccc}
 0 & 0 & 1 \\
 1 & 0 & 0 \\
 1 & 1 & 0 \\
\end{array}
\right] \text{ and } S= \left[
\begin{array}{ccc}
 1 & 0 & 1 \\
 0 & 1 & 0 \\
 0 & 1 & 0 \\
\end{array}
\right]$$

- $E^2_{4}$ is elementary shift equivalent to $E^3_{4}$ as $E^2_{4} = RS$ and $E^3_{4}=SR$, where, 
$$R=\left[
\begin{array}{ccc}
 1 & 0 & 0 \\
 0 & 1 & 1 \\
 1 & 0 & 0 \\
\end{array}
\right] \text{ and }  S=\left[
\begin{array}{ccc}
 1 & 1 & 0 \\
 0 & 0 & 1 \\
 1 & 0 & 0 \\
\end{array}
\right]$$

- $E^2_3$ and  $E^2_4$ are not elementary shift equivalent. 

In fact there is no $R, S\in \M_3(\mathbb N)$ such that  $E^2_3=RS$ and $E^2_4=SR$. This can be checked by a simple computer program (see~\S\ref{mathematicacode}). Since the entires of  $E^2_3$ and  $E^2_4$ consist of only $0$ and $1$, then the entries of $R$ and $S$ have to be either $0$ or $1$ as well. Thus there are only $512$ of these matrices to check. There are no such $R$ and $S$ and thus $E^2_3$ and  $E^2_4$ are not elementary shift equivalent. However,   $E^2_3$ and  $E^2_4$ are strong shift equivalent with lag 2, via the elementary shift equivalence with  $E^3_4$.

We have shown that the graded $K$-theory of $E^2_3, E^2_4$ and $E^3_4$  are order-preserving module isomorphic and consequently they are strongly shift equivalent and the associated Leavitt path algebras are graded Morita equivalent.  Next we show that if the graded $K$-theories are further pointed isomorphic, then the associated Leavitt path algebras are graded isomorphic.

\medskip 

$\bullet$   \emph{$L_\K(E^2_3)$ is not graded isomorphic to  $L_\K(E^3_4)$. }

\medskip

We show that there is no pointed module isomorphism between the corresponding graded $K_0$-groups of these algebras, thus they can't be graded isomorphic to each other.

First note that $E^2_3$ is elementary shift equivalent to $F^2_1$, as $E^2_{3} = R_1S_1$ and $F^2_{1}=S_1R_1$, with 
$$R_1=\left[
\begin{array}{cc}
 0 & 1 \\
 1 & 0 \\
 1 & 1 \\
\end{array}
\right] \text{ and } S_1=
\left[
\begin{array}{ccc}
 1 & 0 & 1 \\
 0 & 1 & 0 \\
\end{array}
\right].
$$
 We also have that $E^3_4$ is elementary shift equivalent to $F^2_1$, as $E^3_4 = R_2S_2$ and $F^2_1=S_2R_2$, with 
  $$R_2=\left[
\begin{array}{cc}
 1 & 0 \\
 0 & 1 \\
 0 & 1 \\
\end{array}
\right] \text{  and } 
S_2=\left[
\begin{array}{ccc}
 1 & 1 & 1 \\
 1 & 0 & 0 \\
\end{array}
\right].$$

Following \S\ref{sec:symbolic-dynamics}, the matrices $R_1$ and $R_2$ induce order-preserving module isomorphisms $\phi_{R_1}$ and $\phi_{R_2}$ (diagram below) which send the order units 
 $\begin{bmatrix} 1 & 1 & 1 \end{bmatrix} R_1= \begin{bmatrix} 2 & 2  \end{bmatrix}$ and  $\begin{bmatrix} 1 & 1 & 1 \end{bmatrix} R_2= \begin{bmatrix} 1 & 2  \end{bmatrix}$:
\[
\begin{tikzcd}
K_0^{\gr}(L_\K(E^2_3)) \arrow[r, "\phi_{R_1}" ] \arrow[d, dashed,"\mu" ]
& K_0^{\gr}(L_\K(F^2_1)) \arrow[d, "T" ] \\
K_0^{\gr}(L_\K(E^3_4)) \arrow[r, "\phi_{R_2}"]
& K_0^{\gr}(L_\K(F^2_1))
\end{tikzcd}
\]

If there was a pointed isomorphism  $\mu \colon K_0^{\gr}(L_\K(E^2_3)) \rightarrow K_0^{\gr}(L_\K(E^3_4))$, there would need to be a $T\in \GL_2( \mathbb{Z} \left[ \frac{1}{2} \right])$ such that $ T F^2_1= F^2_1 T$ and $\begin{bmatrix} 2 & 2  \end{bmatrix} T = \begin{bmatrix} 1 & 2  \end{bmatrix}$. Thus for 
 $T=\begin{bmatrix}
a & b \\
c & d 
\end{bmatrix}$ 
 these conditions give: 
\begin{align*}
\begin{bmatrix}
1 & 2 \\
1 & 0 
\end{bmatrix}
\begin{bmatrix}
a & b \\
c & d 
\end{bmatrix} & =
\begin{bmatrix}
a & b \\
c & d 
\end{bmatrix}
\begin{bmatrix}
1 & 2 \\
1 & 0 
\end{bmatrix}\\ 
\begin{bmatrix}
2 & 2 
\end{bmatrix}
\begin{bmatrix}
a & b \\
c & d 
\end{bmatrix} & =
\begin{bmatrix}
1 & 2 
\end{bmatrix}
\end{align*}

This yields four equations, 
$b=2c$, $a=c+d$,  $2a+2c=1$ and $b+d=1$. However, this system has no solution, implying that $L_\K(E^2_3)$ can't be graded isomorphic to  $L_\K(E^3_4)$. 

\medskip

$\bullet$   \emph{$L_\K(E^2_4)$ is not graded isomorphic to  $L_\K(E^3_4)$. }

\medskip 

We show that there is no pointed module isomorphism between the corresponding graded $K_0$-groups of these algebras, thus they can't be graded isomorphic to each other.

First note that $E^2_4$ is elementary shift equivalent to $F^2_2$ as $E^2_4=R_1 S_1$ and $F^2_2=S_1 R_1$ with
\[R_1=\left[
\begin{array}{cc}
 1 & 0 \\
 0 & 1 \\
 1 & 0 \\
\end{array}
\right] \text{  and } S_1=\left[
\begin{array}{ccc}
 1 & 1 & 0 \\
 1 & 0 & 1 \\
\end{array}
\right].
\]
We also have that $E^3_4$ is elementary shift equivalent to $F^2_2$ as  $E^3_4=R_2 S_2$ and $F^2_2=S_2 R_2$ with 
 \[R_2=\left[
\begin{array}{cc}
 1 & 1 \\
 1 & 0 \\
 1 & 0 \\
\end{array}
\right] \text{ and } S_2=\left[
\begin{array}{ccc}
 1 & 0 & 0 \\
 0 & 1 & 1 \\
\end{array}
\right].
\]
If there was a pointed isomorphism of  $\mu \colon  K_0^{\gr}(L_\K(E^2_4)) \rightarrow K_0^{\gr}(L_\K(E^3_4))$, similar to the argument of the first case, there would need to be a map $T$ such that $ T F^2_2= F^2_2 T$ and $\begin{bmatrix} 2 & 1  \end{bmatrix} T = \begin{bmatrix} 3 & 1  \end{bmatrix}$. We get a system of equations that is not solvable, implying that $L_\K(E^2_4)$ can't be graded isomorphic to  $L_\K(E^3_4)$.

\medskip

$\bullet$   \emph{$L_\K(E^2_3)$ is graded isomorphic to  $L_\K(E^2_4)$. }

\medskip

First note that $E^2_3$ is elementary shift equivalent to $F^2_1$, as  $E^2_3=R_1 S_1$ and $F^2_1=S_1 R_1$ with 
\[R_1=\left[
\begin{array}{cc}
 0 & 1 \\
 1 & 0 \\
 1 & 1 \\
\end{array}
\right] \text{  and } S_1=\left[
\begin{array}{ccc}
 1 & 0 & 1 \\
 0 & 1 & 0 \\
\end{array}
\right].\]

We also have that $E^2_4$ is elementary shift equivalent to $F^2_2$, as  $E^2_4=R_2 S_2$ and $F^2_2=S_2 R_2$ with 
\[R_2=
\left[
\begin{array}{cc}
 1 & 0 \\
 0 & 1 \\
 1 & 0 \\
\end{array}
\right] \text { and } S_2=\left[
\begin{array}{ccc}
 1 & 1 & 0 \\
 1 & 0 & 1 \\
\end{array}
\right].
\]

Recall that $F^2_{1}$ is elementary shift equivalent to $F^2_{2}$ as $F^2_{1} = RS$ and $F^2_{2}=SR$, where, 
$$R= \begin{bmatrix}
    1 & 1  \\
    1 & 0  \\
   \end{bmatrix} \text{  and } 
   S= \begin{bmatrix}
    1 & 0  \\
    0 & 2  \\
 \end{bmatrix}.$$

Following \S\ref{sec:symbolic-dynamics}, the matrices $R_1$ and $R_2$ and $S$  induce order-preserving module isomorphisms $\phi_{R_1}$, $\phi_{R_2}$ and $\phi_S$ (diagram below) which send the order units 
 $\begin{bmatrix} 1 & 1 & 1 \end{bmatrix} R_1= \begin{bmatrix} 2 & 2  \end{bmatrix}$ and  $\begin{bmatrix} 1 & 1 & 1 \end{bmatrix} R_2 S= \begin{bmatrix} 2 & 2  \end{bmatrix}$ in $K_0^{\gr}(L_\K(F^2_1))$:

\[
\begin{tikzcd}
K_0^{\gr}(L_\K(E^2_4)) \arrow[r, "\phi_{R_2}" ] 
& K_0^{\gr}(L_\K(F^2_2)) \arrow[d, "\phi_S" ] \\
K_0^{\gr}(L_\K(E^2_3)) \arrow[r, "\phi_{R_1}"]
& K_0^{\gr}(L_\K(F^2_1))
\end{tikzcd}
\]

Now Proposition~\ref{lateaddition} implies that 
\[L_\K(E^2_3)\cong_{\gr} L_\K(E^2_4) \cong_{\gr} \M_2(L_\K(F^2_1))\cong_{\gr} L_\K(F^2_1),\]
where the last isomorphism follows from an argument similar to (\ref{data475}): We have 
$K_0^{\gr}(L_\K(F^2_1))^+ \cong T_{F^2_1}$ and under this isomorphism $[L_\K(F^2_1)]\in  K_0^{\gr}(L_\K(F^2_1))^+$ is represented by $w_1+w_2 \in T_{F^2_1}$. In the talented monoid $T_{F^2_1}$, we have
$$w_1+w_2=w_1(1)+w_2(1)+w_1(1)+w_2(1).$$ Passing this back to $K_0^{\gr}(L_\K(F^2_1))^+$ and taking $\End_{L_\K(F^2_1)}$ of both sides, we obtain
$$L_\K(F^2_1)  \cong_{\gr} \M_2\big(L_\K(F^2_1) \big)\big(1,1\big) \cong_{\gr} \M_2\big(L_\K(F^2_1) \big).$$

{\bf 15.} Consider the graphs $E^2_{5}$ and $E^3_{1}$ below with the adjacency matrices $A_{E^2_5}=\begin{bmatrix} 0 & 1 & 0 \\ 0 & 1 & 1 \\ 1 & 1 & 0 \\
 \end{bmatrix}$ and $E^3_1=\begin{bmatrix} 0 & 1 & 0 \\ 0 & 0 & 1 \\ 1 & 1 & 1 \\
 \end{bmatrix}$, respectively.

\begin{figure}[H]
\centering
\begin{tikzpicture}[scale=3.5]
\fill (0,0)  circle[radius=.6pt];
\fill (-.35,-.61)  circle[radius=.6pt];
\fill (.35,-.61)  circle[radius=.6pt];
\draw (0,-.15) node{$v_1$};
\draw (-.70, 0) node{$E^2_5:$};
\draw (-.35+.13,-.535) node{$v_3$};
\draw (.35-.13,-.535) node{$v_2$};
\draw[->, shorten >=5pt] (0,0) to (.35,-.61);
\draw[->, shorten >=5pt] (.35,-.61) to (-.35,-.61);
\draw[->, shorten >=5pt] (-.35,-.61) to [in=-120,out=-60] (.35,-.61) ;
\draw[->, shorten >=5pt] (-.35,-.61) to (0,0);
\draw[->, shorten >=5pt] (.35,-.61) to[in=10,out=-70, loop] (-.35,-.61);
\end{tikzpicture}
\hspace{10pt}
\begin{tikzpicture}[scale=3.5]
\fill (0,0)  circle[radius=.6pt];
\fill (-.35,-.61)  circle[radius=.6pt];
\fill (.35,-.61)  circle[radius=.6pt];
\draw (0,-.15) node{$v_1$};
\draw (-.70, 0) node{$E^3_1:$};
\draw (-.35+.13,-.535) node{$v_3$};
\draw (.35-.13,-.535) node{$v_2$};
\draw[->, shorten >=5pt] (0,0) to (.35,-.61);
\draw[->, shorten >=5pt] (.35,-.61) to (-.35,-.61);
\draw[->, shorten >=5pt] (-.35,-.61) to [in=-120,out=-60] (.35,-.61) ;
\draw[->, shorten >=5pt] (-.35,-.61) to (0,0);
\draw[->, shorten >=5pt] (-.35,-.61) to[in=250,out=170, loop] (-.35,-.61);
\end{tikzpicture}
\end{figure}
 
 Since $\det(E^2_5)=\det(E^3_1)=1$  and the graphs are  primitive,  Lemma~\ref{DL} and Corollary~\ref{halfspace} give:
 \begin{align*}
K_0^{\gr}(L_\K(E^2_5))&=   \mathbb{Z} ^3, \\
K_0^{\gr}(L_\K(E^2_5))^+&=\left\{ {\bf u} \in \mathbb{Z}^3 \, | \, {\bf u} \cdot {\bf z} > 0 \right\}  \cup \, \{0\}, \text{ where } \bf{z} \text{ is } \left[
\begin{array}{c}
 \frac{1}{3} \left(1-\frac{2\ 2^{2/3}}{\sqrt[3]{13+3
   \sqrt{33}}}+\frac{\sqrt[3]{13+3 \sqrt{33}}}{2^{2/3}}\right) \\
   \frac{1}{6} \sqrt[3]{54-6
   \sqrt{33}}+\frac{\sqrt[3]{9+\sqrt{33}}}{6^{2/3}} \\ 1 \\
\end{array}
\right],\\
{}^1(a,b,c)&=(c,a+b+c,b).
\end{align*}

whereas, 
\begin{align*}
K_0^{\gr}(L_\K(E^3_1))&=   \mathbb{Z} ^3, \\
K_0^{\gr}(L_\K(E^3_1))^+&=\left\{ {\bf u} \in \mathbb{Z}^3 \, | \, {\bf u} \cdot {\bf z} > 0 \right\}  \cup \, \{0\}, \text{ where } \bf{z} \text{ is } 
\begin{bmatrix}
 \frac{1}{3} \left(-1-\frac{4\ 2^{2/3}}{\sqrt[3]{13+3
   \sqrt{33}}}+\sqrt[3]{2 \left(13+3 \sqrt{33}\right)}\right) \\
   \frac{1}{3} \left(-1-\frac{2}{\sqrt[3]{17+3
   \sqrt{33}}}+\sqrt[3]{17+3 \sqrt{33}}\right) \\ 1 \\
\end{bmatrix},
\\
{}^1(a,b,c)&=(c,a+c,b+c).
\end{align*}

The invertible matrix  $T = 
\begin{bmatrix}
    1 & 0 & 0 \\
    0 & 1 & 0 \\
    0 & 1 & 1 \\
\end{bmatrix}
\in \GL_3(\mathbb Z)$ induces an order-preserving isomorphism 
\begin{equation*}
T \colon  K_0^{\gr}(L_\K(E^2_5)) \longrightarrow K_0^{\gr}(L_\K(E^3_{1})),
\end{equation*}
  i.e., $A_{E^2_{5}}T=TA_{E^3_1}$. Note that $T$  takes Perron-Frobenius eigenvalue $\bf{z}$ of $E^2_5$ to a positive scalar multiple of the eigenvector $\bf{z}^\prime$ of $E^3_{1}$, thus preserving the positive cones. 

Therefore the matrices $A_{E^2_{5}}$ and $A_{E^3_{1}}$ are shift equivalent. We show that these matrices are indeed strong shift equivalent.  Starting from $E^2_5$, the maximal out-split at $v_3$ produces the graph below which coincides with the maximal in-split of $E^{3}_1$ at $v_2$.

\begin{figure}[H]
\centering
\begin{tikzpicture}[scale=3.5]
\fill (0,0)  circle[radius=.6pt];
\fill (.5,0)  circle[radius=.6pt];
\fill (0,-.5)  circle[radius=.6pt];
\fill (.5,-.5)  circle[radius=.6pt];
\draw[->, shorten >=5pt] (.5,0) to (0,0);
\draw[->, shorten >=5pt] (0,0) to [out=-45,in=45](0,-.5);

\draw[->, shorten >=5pt] (.5,-.5) to (.5,0);
\draw[->, shorten >=5pt] (0,0) to (.5,-.5);

\draw[->, shorten >=5pt] (0,0) to[in=135,out=45, loop] (0,0);
\draw[->, shorten >=5pt] (0,-.5) to[in=-120,out=135,] (0,0);
\end{tikzpicture}
\end{figure}

Indeed: 

 - $E^2_{5}$ is elementary shift equivalent to  $E^3_{1}$ as  $E^2_{5} = RS$ and $E^3_{1}=SR$ with
 \[R= \begin{bmatrix}
    0 & 1 & 0 \\
    0 & 0 & 1 \\
    1 & 1 & 0 \\
    \end{bmatrix} \text{ and } S= \begin{bmatrix}
    1 & 0 & 0 \\
    0 & 1 & 0 \\
    0 & 1 & 1 \\
\end{bmatrix}.\]

\medskip

$\bullet$  \emph{$L_\K(E^2_5)$ is not graded isomorphic to  $L_\K(E^3_1)$.}

\medskip

There is no pointed module isomorphism between the graded $K$-theory of these algebras:  Any attempt to find an invertible matrix $T\in \GL_3(\mathbb Z)$ such that $T E^2_{5}=E^3_{1} T,$ with the property that $ \begin{bmatrix} 1 & 1 & 1 \\ \end{bmatrix}T =\begin{bmatrix} 1 & 1 & 1 \\ \end{bmatrix} $ yields a system whose solution is not realisable over $\mathbb{Z}$.

{\bf 16.} Consider the graphs $E^2_{6}$ and $E^3_2$ below with the adjacency matrics $A_{E^2_6}= \begin{bmatrix} 1 & 1 & 0 \\ 1 & 1 & 1 \\ 1 & 1 & 0 \\ \end{bmatrix}$ and 
$A_{E^3_2}= \begin{bmatrix} 0 & 1 & 0 \\ 1 & 1 & 1 \\ 1 & 1 & 1 \\ \end{bmatrix}$, respectively.

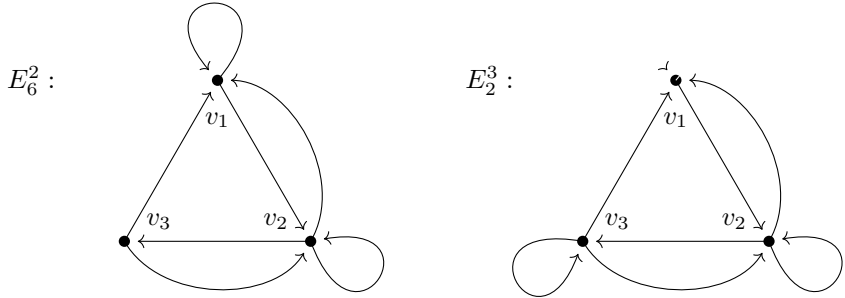
\begin{figure}[H]
\centering
\begin{tikzpicture}[scale=3.5]
\fill (0,0)  circle[radius=.6pt];
\fill (-.35,-.61)  circle[radius=.6pt];
\fill (.35,-.61)  circle[radius=.6pt];
\draw (0,-.15) node{$v_1$};
\draw (-.70, 0) node{$E^2_6:$};
\draw (-.35+.13,-.535) node{$v_3$};
\draw (.35-.13,-.535) node{$v_2$};
\draw[->, shorten >=5pt] (0,0) to (.35,-.61);
\draw[->, shorten >=5pt] (.35,-.61) to [in=0,out=60] (0,0);
\draw[->, shorten >=5pt] (.35,-.61) to (-.35,-.61);
\draw[->, shorten >=5pt] (-.35,-.61) to [in=-120,out=-60] (.35,-.61) ;
\draw[->, shorten >=5pt] (-.35,-.61) to (0,0);
\draw[->, shorten >=5pt] (0,0) to[in=130,out=50, loop] (0,0);
\draw[->, shorten >=5pt] (.35,-.61) to[in=10,out=-70, loop] (-.35,-.61);

\end{tikzpicture}
\begin{tikzpicture}[scale=3.5]
\fill (0,0)  circle[radius=.6pt];
\fill (-.35,-.61)  circle[radius=.6pt];
\fill (.35,-.61)  circle[radius=.6pt];
\draw (0,-.15) node{$v_1$};
\draw (-.70, 0) node{$E^3_2:$};
\draw (-.35+.13,-.535) node{$v_3$};
\draw (.35-.13,-.535) node{$v_2$};
\draw[->, shorten >=5pt] (0,0) to (.35,-.61);
\draw[->, shorten >=5pt] (.35,-.61) to [in=0,out=60] (0,0);
\draw[->, shorten >=5pt] (.35,-.61) to (-.35,-.61);
\draw[->, shorten >=5pt] (-.35,-.61) to [in=-120,out=-60] (.35,-.61) ;
\draw[->, shorten >=5pt] (-.35,-.61) to (0,0);
\draw[draw=white, ->, shorten >=5pt] (0,0) to[in=130,out=50, loop] (0,0);
\draw[->, shorten >=5pt] (-.35,-.61) to[in=250,out=170, loop] (-.35,-.61);
\draw[->, shorten >=5pt] (.35,-.61) to[in=10,out=-70, loop] (-.35,-.61);
\end{tikzpicture}
\end{figure}

We have $\det(E^2_6)=\det(E^3_{2})=0$. However they are the result of the out-split and in-split of the graph below.

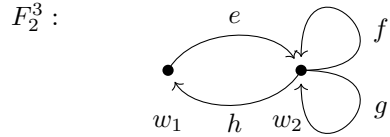
\begin{figure}[H]
\centering
\begin{tikzpicture}[scale=3.5]
\fill (0,0)  circle[radius=.6pt];
\draw (0,-.2) node{$w_1$};
\draw (-.5,.2) node{$F^3_2:$};
\fill (.5,0)  circle[radius=.6pt];
\draw (.45,-.2) node{$w_2$};
\draw[->, shorten >=5pt] (.5,0) to[in=-90,out=0, loop] (.5,0);
\draw[->, shorten >=5pt] (.5,0) to[in=90,out=0, loop] (.5,0);
\draw[->, shorten >=5pt] (0,0) to[in=120,out=60] (.5,0);
\draw[->, shorten >=5pt] (.5,0) to[in=-60,out=-120] (0,0);
\draw (.8,.15) node{$f$};
\draw (.8,-.15) node{$g$};
\draw (.25,.2) node{$e$};
\draw (.25,.-.2) node{$h$};

\end{tikzpicture}
\end{figure}

We have:

- $E^2_6$ is obtained from $F^3_2$ by an out-split at  $w_2$ with partition $\{h,f\}, \{ g\}$.

- $E^3_2$  is obtained from $F^3_2$ by an in-split at $w_2$ with partition $\{e,f \} , \{ g\}$.

Thus the graded $K$-theory of $L_\K(E^2_6)$ and $L_\K(E^3_2)$ are the same as that of $L_\K(F^3_2)$.  Since $\det(F^3_2) =1$  and the graph  is primitive,  Lemma~\ref{DL} and Corollary~\ref{halfspace} give:

\begin{align*}
K_0^{\gr}(L_\K(F^3_2))& =   \mathbb{Z} ^2, \\
K_0^{\gr}(L_\K(F^3_2))^+ & =\left\{ {\bf u} \in \mathbb{Z}^2 \, | \, {\bf u} \cdot {\bf z} > 0 \right\}  \cup \, \{0\}, \text{  where } \bf{z} \text{  is }\left[
\begin{array}{c}
 1+ \sqrt{2} \\
   1 \\
\end{array}
\right],\\
{}^1(a,b) & =(b,a+2b).
\end{align*}

Indeed: 

 - $E^2_{6}$ is elementary shift equivalent to  $E^3_{2}$ as  $E^2_{6} = RS$ and $E^3_{2}=SR$ with
 \[R= \begin{bmatrix}
    0 & 0 & 1 \\
    1 & 1 & 0 \\
    0 & 1 & 0 \\
    \end{bmatrix} \text{ and } S= \begin{bmatrix}
    0 & 0 & 1 \\
    1 & 1 & 0 \\
    1 & 1 & 0 \\
\end{bmatrix}.\]

\medskip

$\bullet$  \emph{$L_\K(E^2_6)$ is not graded isomorphic to  $L_\K(E^3_2)$.}

\medskip

First note that $E^2_6$ is elementary shift equivalent to $F^3_2$, as  $E^2_6=R_1 S_1$ and $F^3_2=S_1 R_1$ with 
\[R_1=\left[
\begin{array}{cc}
 1 & 0 \\
 1 & 1 \\
 1 & 0 \\
\end{array}
\right] \text{  and } S_1=\left[
\begin{array}{ccc}
 1 & 1 & 0 \\
 0 & 0 & 1 \\
\end{array}
\right].\]

We also have that $E^3_2$ is elementary shift equivalent to $F^3_2$, as  $E^3_2=R_2 S_2$ and $F^3_2=S_2 R_2$ with 
\[R_2=
\left[
\begin{array}{cc}
 0 & 1 \\
 1 & 0 \\
 1 & 0 \\
\end{array}
\right] \text { and } S_2=\left[
\begin{array}{ccc}
 1 & 1 & 1 \\
 0 & 1 & 0 \\
\end{array}
\right].
\]

Following \S\ref{sec:symbolic-dynamics}, the matrices $R_1$ and $R_2$  induce order-preserving module isomorphisms $\phi_{R_1}$, $\phi_{R_2}$  (diagram below) which send the order units 
 $\begin{bmatrix} 1 & 1 & 1 \end{bmatrix} R_1= \begin{bmatrix} 3 & 1  \end{bmatrix}$ and  $\begin{bmatrix} 1 & 1 & 1 \end{bmatrix} R_2 S= \begin{bmatrix} 2 & 1  \end{bmatrix}$ in $K_0^{\gr}(L_\K(F^3_2))$:

\[
\begin{tikzcd}
K_0^{\gr}(L_\K(E^2_6)) \arrow[r, "\phi_{R_2}" ] 
& K_0^{\gr}(L_\K(F^3_2)) \arrow[d, "\phi_T" ] \\
K_0^{\gr}(L_\K(E^3_2)) \arrow[r, "\phi_{R_1}"]
& K_0^{\gr}(L_\K(F^3_2))
\end{tikzcd}
\]

 Any attempt to find an invertible matrix $T\in \GL_2(\mathbb Z)$ such that $T F^3_{2}=F^3_{2} T,$ with the property that $ \begin{bmatrix} 3 & 1  \\ \end{bmatrix}T =\begin{bmatrix} 2 & 1  \\ \end{bmatrix} $ yields a system whose solution is not realisable over $\mathbb{Z}$. This implies that there is no pointed order-preserving isomorphism between $L_\K(E^2_6))$ and $L_\K(E^3_2))$.

{\bf 17.} Consider the graphs $E^3_{3}$ below with the adjacency matrix  $A_{E^3_3}= \begin{bmatrix}
    1 & 1 & 1 \\ 1 & 1 & 1 \\ 1 & 1 & 1 \\ \end{bmatrix}$.

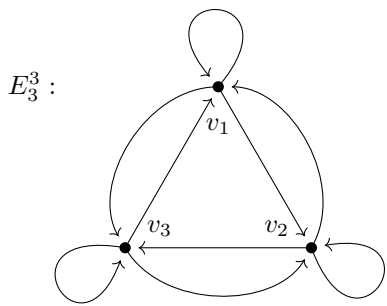
\begin{figure}[H]
\centering
\begin{tikzpicture}[scale=3.5]
\fill (0,0)  circle[radius=.6pt];
\fill (-.35,-.61)  circle[radius=.6pt];
\fill (.35,-.61)  circle[radius=.6pt];
\draw (0,-.15) node{$v_1$};
\draw (-.70, 0) node{$E^3_3:$};
\draw (-.35+.13,-.535) node{$v_3$};
\draw (.35-.13,-.535) node{$v_2$};
\draw[->, shorten >=5pt] (0,0) to (.35,-.61);
\draw[->, shorten >=5pt] (.35,-.61) to [in=0,out=60] (0,0);
\draw[->, shorten >=5pt] (.35,-.61) to (-.35,-.61);
\draw[->, shorten >=5pt] (-.35,-.61) to [in=-120,out=-60] (.35,-.61) ;
\draw[->, shorten >=5pt] (-.35,-.61) to (0,0);
\draw[->, shorten >=5pt] (0,0) to [in=120,out=180
] (-.35,-.61) ;
\draw[->, shorten >=5pt] (0,0) to[in=130,out=50, loop] (0,0);
\draw[->, shorten >=5pt] (-.35,-.61) to[in=250,out=170, loop] (-.35,-.61);
\draw[->, shorten >=5pt] (.35,-.61) to[in=10,out=-70, loop] (-.35,-.61);
\end{tikzpicture}
\end{figure}

We have $\det(E^3_{3})=0$. However the graph $E^3_3$  can be obtained by maximal out-splitting of the graph $R^1_3$ below.

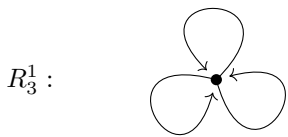
\begin{figure}[H]
\centering
\begin{tikzpicture}[scale=3.5]
\fill (0,0)  circle[radius=.6pt];
\draw (-.70, 0) node{$R^1_3:$};

\draw[->, shorten >=5pt] (0,0) to[in=105+30, out=15+30, loop] (0,0);
\draw[->, shorten >=5pt] (0,0) to[in=225+30, out=135+30, loop] (0,0);
\draw[->, shorten >=5pt] (0,0) to[in=-15+30, out=255+30, loop] (0,0);
\end{tikzpicture}
\end{figure}

Since the adjacency matrix $A_{R^1_3}=[3]$, then (\ref{linkkgroup}) (or Corollary~\ref{alternativeformdel}) immediately gives: 
\begin{align*}
K_0^{\gr}(L_\K(R^1_3)) & =   \mathbb{Z}\left[ \frac{1}{3} \right] , \\
K_0^{\gr}(L_\K(R^1_3))^+&= \mathbb{N}\left[ \frac{1}{3} \right]
,\\
{}^1a&=3a.
\end{align*}

As $E^3_3$ is an out-split of $R^1_3$, the graded $K$-theory of its Leavitt path algebra coincides with that of $R^1_3$.

{\bf 18.}
Consider the graphs $E^4_{1}$ below with the adjacency matrix $A_{E^4_1}= \begin{bmatrix} 0 & 1 & 0 \\ 1 & 1 & 1 \\ 1 & 1 & 0 \\ \end{bmatrix}$.

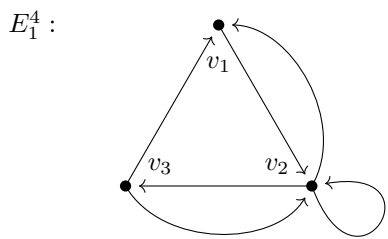
\begin{figure}[H]
\centering
\begin{tikzpicture}[scale=3.5]
\fill (0,0)  circle[radius=.6pt];
\fill (-.35,-.61)  circle[radius=.6pt];
\fill (.35,-.61)  circle[radius=.6pt];
\draw (0,-.15) node{$v_1$};
\draw (-.70, 0) node{$E^4_1:$};
\draw (-.35+.13,-.535) node{$v_3$};
\draw (.35-.13,-.535) node{$v_2$};
\draw[->, shorten >=5pt] (0,0) to (.35,-.61);
\draw[->, shorten >=5pt] (.35,-.61) to [in=0,out=60] (0,0);
\draw[->, shorten >=5pt] (.35,-.61) to (-.35,-.61);
\draw[->, shorten >=5pt] (-.35,-.61) to [in=-120,out=-60] (.35,-.61) ;
\draw[->, shorten >=5pt] (-.35,-.61) to (0,0);
\draw[->, shorten >=5pt] (.35,-.61) to[in=10,out=-70, loop] (-.35,-.61);
\end{tikzpicture}
\end{figure}

Since $\det(E^4_{1})=1$ and the graph is  primitive,  Lemma~\ref{DL} and Corollary~\ref{halfspace} give:
\begin{align*}
K_0^{\gr}(L_\K(E^4_1)) & =   \mathbb{Z} ^3, \\
K_0^{\gr}(L_\K(E^4_1))^+ & =\left\{ {\bf u} \in \mathbb{Z}^3 \, | \, {\bf u} \cdot {\bf z} > 0 \right\}  \cup \, \{0\}, \text{  where } \bf{z} \text{ is } \left[
\begin{array}{c}
 \sqrt[3]{\frac{1}{3^2}}\, \sqrt[3]{\frac{1}{2}
   \left(9+\sqrt{93}\right)} -\sqrt[3]{\frac{2}{3
   \left(9+\sqrt{93}\right)}} \\ 
   \frac{1}{3}
   \left(1+\sqrt[3]{
   \frac{1}{2} \left(29-3
   \sqrt{93}\right)
   }+\sqrt[3]{\frac{1}{2} \left(29+3
   \sqrt{93}\right)}\right) \\ 1 \\
\end{array}
\right]
\\
{}^1(a,b,c) & =(b+c,a+b+c, b).
\end{align*}

{\bf 19.} Consider the graphs $E^4_{2}$ below with the adjacency matrix $A_{E^4_2}= \begin{bmatrix} 1 & 1 & 1 \\ 1 & 1 & 1 \\ 1 & 1 & 0 \\ \end{bmatrix}$.

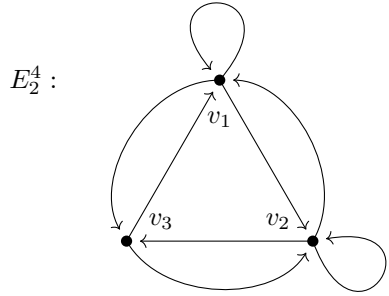
\begin{figure}[H]
\centering
\begin{tikzpicture}[scale=3.5]
\fill (0,0)  circle[radius=.6pt];
\fill (-.35,-.61)  circle[radius=.6pt];
\fill (.35,-.61)  circle[radius=.6pt];
\draw (0,-.15) node{$v_1$};
\draw (-.70, 0) node{$E^4_2:$};
\draw (-.35+.13,-.535) node{$v_3$};
\draw (.35-.13,-.535) node{$v_2$};
\draw[->, shorten >=5pt] (0,0) to (.35,-.61);
\draw[->, shorten >=5pt] (.35,-.61) to [in=0,out=60] (0,0);
\draw[->, shorten >=5pt] (.35,-.61) to (-.35,-.61);
\draw[->, shorten >=5pt] (-.35,-.61) to [in=-120,out=-60] (.35,-.61) ;
\draw[->, shorten >=5pt] (-.35,-.61) to (0,0);
\draw[->, shorten >=5pt] (0,0) to [in=120,out=180
] (-.35,-.61) ;
\draw[->, shorten >=5pt] (0,0) to[in=130,out=50, loop] (0,0);
\draw[->, shorten >=5pt] (.35,-.61) to[in=10,out=-70, loop] (-.35,-.61);
\end{tikzpicture}
\end{figure}

We have $\det(E^4_{2})=0$. However the graph $E^4_2$  can be obtained by the out-splitting of the graph $F^3_1$  below at $w_2$ with partitions $\{ e,g\}$ and $\{ f,h \}$.

\begin{figure}[H]
\centering
\begin{tikzpicture}[scale=3.5]
\fill (0,0)  circle[radius=.6pt];
\draw (0,-.2) node{$w_1$};
\draw (-.2,.2) node{$F^3_1$};
\fill (.5,0)  circle[radius=.6pt];
\draw (.45,-.2) node{$w_2$};
\draw[->, shorten >=5pt] (.5,0) to[in=-90,out=0, loop] (.5,0);
\draw[->, shorten >=5pt] (.5,0) to[in=90,out=0, loop] (.5,0);
\draw[->, shorten >=5pt] (0,0) to[in=120,out=60] (.5,0);
\draw[->, shorten >=5pt] (.5,0) to[in=-60,out=-120] (0,0);
\draw[->, shorten >=5pt] (.5,0) to (0,0);
\draw (.25,-.2) node{$e$};
\draw (.8,.15) node{$f$};
\draw (.8,-.15) node{$g$};
\draw (.25,.07) node{$h$};
\end{tikzpicture}
\end{figure}

As $E^4_2$ is an out-split of $F^3_1$, its graded $K$-theory coincides with that of $F^3_1$.  We have $\det(F^3_{1})=-2$. In order to calculate $K_0^{\gr}(L_\K(F^3_1))$, we will use the alternative description of dimension group from Corollary~\ref{alternativeformdel}.

\begin{prop}
  We have   $\Delta_{A_{F^3_1}}= \left (\mathbb{Z}\left[ \frac{1}{2} \right] \right)^2$.
\end{prop}

\begin{proof} First note that $A_{F^3_1}^{-1}=\frac{1}{2}\begin{bmatrix}
 -2 & 1 \\
 2 & 0 \\
\end{bmatrix}.$

Since $\det(A_{A_{F^3_1}})=-2$ and that $\Delta_{A_{F^3_1}} \subseteq \left (\mathbb{Z}\left[ \frac{1}{2} \right] \right)^2$, we will show that one obtains the spanning set over $\mathbb{Z}$ of $\bigcup_{k=0}^\infty \{ \begin{bmatrix} \frac{1}{2^k} &  0\end{bmatrix},\begin{bmatrix} 0& \frac{1}{2^k}  \end{bmatrix} \}$ within $\Delta_{A_{F^3_1}}$ . We will prove this by induction.\\
When $k=0$, we have that $\{ \begin{bmatrix} 1 &  0\end{bmatrix},\begin{bmatrix} 0 & 1  \end{bmatrix} \} \subseteq \mathbb{Z}^2 \subseteq \Delta_{A_{F^3_1}}$.\\
When $k=1$, we have 
\begin{align*}
    &\begin{bmatrix} 1 & 1\end{bmatrix}A^{-1}_{F^3_1}= \begin{bmatrix} 0 & \frac{1}{2} \end{bmatrix}, \text{ and } \\
    &\begin{bmatrix} 1 & 1\end{bmatrix}A^{-2}_{F^3_1}= \begin{bmatrix} 0 & \frac{1}{2} \end{bmatrix} A^{-1}_{F^3_1} = \begin{bmatrix} \frac{1}{2} & 0 \end{bmatrix}. \\  
\end{align*}
Assume now that $\left( \frac{1}{2^k}\mathbb{Z} \right)^2 \subseteq \Delta_{A_{F^3_1}}$. This means that for all $\begin{bmatrix} a & b \end{bmatrix} \in \left( \frac{1}{2^k}\mathbb{Z} \right)^2$, there is a $\begin{bmatrix} c & d \end{bmatrix} \in \mathbb{Z}^2$ and an $l \geq 0$ such that $\begin{bmatrix} c & d \end{bmatrix}A^{-l}_{F^3_1}= \begin{bmatrix} a & b \end{bmatrix}$. \\
When $a=b=\frac{1}{2^k}$, we have 
\begin{align*}
     &\begin{bmatrix} c & d \end{bmatrix}A^{-l-1}_{F^3_1}= \begin{bmatrix} \frac{1}{2^k} & \frac{1}{2^k} \end{bmatrix}A^{-1}_{F^3_1} =  \begin{bmatrix} 0 & \frac{1}{2^{k+1}} \end{bmatrix}, \text{ and } \\
    &\begin{bmatrix} c & d\end{bmatrix}A^{l-2}_{F^3_1}= \begin{bmatrix} \frac{1}{2^k} & \frac{1}{2^k} \end{bmatrix}A^{-2}_{F^3_1} = \begin{bmatrix} 0 & \frac{1}{2^{k+1}} \end{bmatrix} A^{-1}_{F^3_1} = \begin{bmatrix} \frac{1}{2^{k+1}} & 0 \end{bmatrix}. \\     
\end{align*}
So  $\bigcup_{k=0}^{\infty} \left (\frac{1}{2^k}\mathbb{Z} \right)^2= \left (\mathbb{Z}\left[ \frac{1}{2} \right] \right)^2 \subseteq \Delta_{A_{F^3_1}}$ and thus the equality follows. 
\end{proof}

Since $F^3_1$ is primitive, we can use Corollary~\ref{halfspace} for calculating the positive cone. Therefore we have: 
\begin{align*}
K_0^{\gr}(L_\K(F^3_1)) & =  \left( \mathbb{Z} \left[ \frac{1}{2} \right] \right) ^2, \\
K_0^{\gr}(L_\K(F^3_1))^+ & =\left\{ {\bf u} \in \left( \mathbb{Z} \left[ \frac{1}{2} \right] \right) ^2 \, | \, {\bf u} \cdot {\bf z} > 0 \right\}  \cup \, \{0\}, \text{  where }\bf{z} \text{ is } \left[
\begin{array}{c}
 \frac{1}{2} \left(\sqrt{3}-1\right) \\ 1 \\
\end{array}
\right]
\\
{}^1(a,b)&=(2b,a+2b).
\end{align*}

{\bf 20.} Consider the graphs $E^5_{1}$ below with the adjacency matrix $A_{E^5_1}= \begin{bmatrix} 1 & 1 & 1 \\ 1 & 0 & 1 \\ 1 & 1 & 0 \\ \end{bmatrix}$.

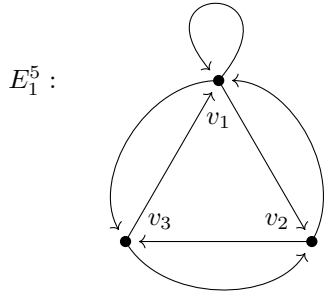
\begin{figure}[H]
\centering
\begin{tikzpicture}[scale=3.5]
\fill (0,0)  circle[radius=.6pt];
\fill (-.35,-.61)  circle[radius=.6pt];
\fill (.35,-.61)  circle[radius=.6pt];
\draw (0,-.15) node{$v_1$};
\draw (-.70, 0) node{$E^5_1:$};
\draw (-.35+.13,-.535) node{$v_3$};
\draw (.35-.13,-.535) node{$v_2$};
\draw[->, shorten >=5pt] (0,0) to (.35,-.61);
\draw[->, shorten >=5pt] (.35,-.61) to [in=0,out=60] (0,0);
\draw[->, shorten >=5pt] (.35,-.61) to (-.35,-.61);
\draw[->, shorten >=5pt] (-.35,-.61) to [in=-120,out=-60] (.35,-.61) ;
\draw[->, shorten >=5pt] (-.35,-.61) to (0,0);
\draw[->, shorten >=5pt] (0,0) to [in=120,out=180
] (-.35,-.61) ;
\draw[->, shorten >=5pt] (0,0) to[in=130,out=50, loop] (0,0);
\end{tikzpicture}
\end{figure}

Since $\det(E^5_{1})=1$ and the graph is  primitive,  Lemma~\ref{DL} and Corollary~\ref{halfspace} give:

\begin{align*}
K_0^{\gr}(L_\K(E^5_1)) & =   \mathbb{Z} ^3, \\
K_0^{\gr}(L_\K(E^5_1))^+ & =\left\{ {\bf u} \in \mathbb{Z}^3 \, | \, {\bf u} \cdot {\bf z} > 0 \right\}  \cup \, \{0\}, \text{ where } \bf{z}\text{ is }\left[
\begin{array}{c}
 \sqrt{2} \\ 1 \\ 1 \\
\end{array}
\right]\\
{}^1(a,b,c) & =(a+b+c,a+c, a+b).
\end{align*}

{\bf 21.} Consider the graphs $E^6_{1}$ below with the adjacency matrix $A_{E^6_1}= \begin{bmatrix} 0 & 1 & 1 \\ 1 & 0 & 1 \\ 1 & 1 & 0 \\ \end{bmatrix}.$

\begin{figure}[H]
\centering
\begin{tikzpicture}[scale=3.5]
\fill (0,0)  circle[radius=.6pt];
\fill (-.35,-.61)  circle[radius=.6pt];
\fill (.35,-.61)  circle[radius=.6pt];
\draw (0,-.15) node{$v_1$};
\draw (-.70, 0) node{$E^6_1$};
\draw (-.35+.13,-.535) node{$v_3$};
\draw (.35-.13,-.535) node{$v_2$};
\draw[->, shorten >=5pt] (0,0) to (.35,-.61);
\draw[->, shorten >=5pt] (.35,-.61) to [in=0,out=60] (0,0);
\draw[->, shorten >=5pt] (.35,-.61) to (-.35,-.61);
\draw[->, shorten >=5pt] (-.35,-.61) to [in=-120,out=-60] (.35,-.61) ;
\draw[->, shorten >=5pt] (-.35,-.61) to (0,0);
\draw[->, shorten >=5pt] (0,0) to [in=120,out=180
] (-.35,-.61) ;
\end{tikzpicture}
\end{figure}

We have that $\det(E^6_{1})=2$. In order to calculate $K_0^{\gr}(L_\K(E^6_1))$, we will use the alternative description of
dimension group from Corollary~\ref{alternativeformdel}. The following proposition follows the same path as in Proposition~\ref{profidsh6}.

\begin{prop}\label{profidsh7}    We have $\Delta_{A_{E^6_1}}= \mathbb{Z}\left[ \frac{1}{2} \right] \begin{bmatrix} 1 & 1 & 1 \end{bmatrix} \oplus 0 \times \mathbb{Z}^2$.
\end{prop}

\begin{proof}
First note that $A_{E^6_1}^{-1}=\frac{1}{2}\begin{bmatrix}
 -1 & 1 & 1 \\
 1 & -1 & 1 \\
 1 & 1 & -1 
\end{bmatrix}.$ 
Consider $$v=\frac{a}{2^k}\begin{bmatrix} 1 & 1 & 1 \end{bmatrix} + \begin{bmatrix} 0 & b & c \end{bmatrix} \in \mathbb{Z}\left[ \frac{1}{2} \right] \begin{bmatrix} 1 & 1 & 1 \end{bmatrix} \oplus 0 \times \mathbb{Z}^2,$$ where $a,b,c \in \mathbb{Z}$ and $k\in \mathbb N$. Since  
$$\begin{bmatrix} 1 & 1 & 1 \end{bmatrix} \begin{bmatrix}
 -1 & 1 & 1 \\
 1 & -1 & 1 \\
 1 & 1 & -1 
\end{bmatrix} = \begin{bmatrix} 1 & 1 & 1 \end{bmatrix},$$ we have
\begin{align*} 
&\frac{a}{2^k} \begin{bmatrix} 1 & 1 & 1 \end{bmatrix} A_{E^6_1}^k  
= \frac{a}{2^k} \begin{bmatrix} 1 & 1 & 1 \end{bmatrix} \begin{bmatrix}
 -1 & 1 & 1 \\
 1 & -1 & 1 \\
 1 & 1 & -1 
\end{bmatrix}^k A_{E^6_1}^k \\ 
= & a \begin{bmatrix} 1 & 1 & 1 \end{bmatrix} \left( \frac{1}{2} \begin{bmatrix}
 -1 & 1 & 1 \\
 1 & -1 & 1 \\
 1 & 1 & -1 
\end{bmatrix} \right)^k A_{E^6_1}^k 
= a \begin{bmatrix} 1 & 1 & 1 \end{bmatrix} A_{E^6_1}^{-k} A_{E^6_1}^k \\ 
=& a \begin{bmatrix} 1 & 1 & 1 \end{bmatrix}. 
\end{align*}
Thus $\frac{a}{2^k}\begin{bmatrix} 1 & 1 & 1 \end{bmatrix}$ is an element of $\Delta_{A_{E^6_1}}$. The vector $ \begin{bmatrix} 0 & b & c \end{bmatrix}$ is also an element of $\Delta_{A_{E^6_1}}$ as $\mathbb{Z}^3 \subseteq \Delta_{A_{E^6_1}}$.\\
Hence $$ \mathbb{Z} \left[ \frac{1}{2} \right]\begin{bmatrix} 1 & 1 & 1 \end{bmatrix} \oplus 0 \times \mathbb{Z}^2 \subseteq \Delta_{A_{E^6_1}}.$$

For the reverse inclusion, let $v \in \Delta_{A_{E^6_1}}$. Then by Corollary~\ref{alternativeformdel},  $v \in \mathbb{Z}^3A_{E^6_1}^{-k}$, for some $k \geq 0$. We will prove by induction that $v \in \mathbb{Z} \begin{bmatrix} 1 & 1 & 1 \end{bmatrix} \oplus 0 \times \mathbb{Z}^2$. 

When $k=0$, $v \in \mathbb{Z}^3 \subseteq \mathbb{Z} \begin{bmatrix} 1 & 1 & 1 \end{bmatrix} \oplus 0 \times \mathbb{Z}^2 \subseteq \mathbb{Z} \left[ \frac{1}{2} \right]\begin{bmatrix} 1 & 1 & 1 \end{bmatrix} \oplus 0 \times \mathbb{Z}^2$. 

When $k=1$,  we can write $vA_{E^6_1}=a\begin{bmatrix} 1 & 1 & 1 \end{bmatrix} + \begin{bmatrix} 0 & b & c\end{bmatrix}$, where $a,b,c \in \mathbb{Z}$. Thus 
\begin{align*}
v=&\frac{a}{2}\begin{bmatrix} 1 & 1 & 1 \end{bmatrix} + \frac{1}{2} \begin{bmatrix} b+c & -b+c & b-c\end{bmatrix}\\ =& \frac{a}{2}\begin{bmatrix} 1 & 1 & 1 \end{bmatrix} + \frac{1}{2} \begin{bmatrix} b+c & b+c & b+c \end{bmatrix} + \frac{1}{2}  \begin{bmatrix} 0 & -2b & -2c\end{bmatrix} \\
=& \frac{a+(b+c)}{2}\begin{bmatrix} 1 & 1 & 1 \end{bmatrix} + \frac{1}{2} \begin{bmatrix} 0 & -2b & -2c\end{bmatrix} \\ =& \frac{a+(b+c)}{2}\begin{bmatrix} 1 & 1 & 1 \end{bmatrix} + \begin{bmatrix} 0 & -b & -c\end{bmatrix}.
\end{align*} 

Hence,  $$v \in \frac{1}{2}\mathbb{Z} \begin{bmatrix} 1 & 1 & 1 \end{bmatrix} \oplus 0 \times \mathbb{Z}^2 \subseteq \mathbb{Z} \left[ \frac{1}{2} \right]\begin{bmatrix} 1 & 1 & 1 \end{bmatrix} \oplus 0 \times \mathbb{Z}^2.$$
  
Assume now that  $vA_{E^6_1}^k \in \mathbb{Z}^3$ implies that $v \in \frac{1}{2^k}\mathbb{Z} \begin{bmatrix} 1 & 1 & 1 \end{bmatrix} \oplus 0 \times \mathbb{Z}^2$.
If $vA_{E^6_1}^{k+1} \in \mathbb{Z}^3$, then $vA_{E^6_1}=\frac{a}{2^k}\begin{bmatrix} 1 & 1 & 1 \end{bmatrix} + \begin{bmatrix} 0 & b & c \end{bmatrix} $ where $a,b,c \in \mathbb{Z}$. Thus \begin{align*}
v=&\frac{a}{2^{k+1}}\begin{bmatrix} 1 & 1 & 1 \end{bmatrix} + \frac{1}{2} \begin{bmatrix} b+c & -b+c & b-c\end{bmatrix}\\ =& \frac{a}{2^{k+1}}\begin{bmatrix} 1 & 1 & 1 \end{bmatrix} + \frac{1}{2} \begin{bmatrix} b+c & b+c & b+c \end{bmatrix} + \frac{1}{2} \begin{bmatrix} 0 & -2b & -2c\end{bmatrix} \\
=& \frac{a+2^k(b+c)}{2^{k+1}}\begin{bmatrix} 1 & 1 & 1 \end{bmatrix} + \frac{1}{2} \begin{bmatrix} 0 & -2b & -2c\end{bmatrix} \\ =& \frac{a+2^k(b+c)}{2^{k+1}}\begin{bmatrix} 1 & 1 & 1 \end{bmatrix} + \begin{bmatrix} 0 & -b & -c\end{bmatrix}.
\end{align*} Hence $v \in \frac{1}{2^{k+1}}\mathbb{Z} \begin{bmatrix} 1 & 1 & 1 \end{bmatrix} \oplus 0 \times \mathbb{Z}^2 \subseteq \mathbb{Z} \left[ \frac{1}{2} \right]\begin{bmatrix} 1 & 1 & 1 \end{bmatrix} \oplus 0 \times \mathbb{Z}^2$. So 
$$\Delta_{A_{E^6_1}} \subseteq \mathbb{Z} \left[ \frac{1}{2} \right]\begin{bmatrix} 1 & 1 & 1 \end{bmatrix} \oplus 0 \times \mathbb{Z}^2,$$ and thus the equality follows. 
\end{proof}

Since $E^6_1$ is primitive, we can use Corollary~\ref{halfspace} for calculating the positive cone.  Therefore we have:

\begin{align*}
K_0^{\gr}(L_\K(E^6_1)) & = \mathbb{Z}\left[ \frac{1}{2} \right] \begin{bmatrix} 1 & 1 & 1 \end{bmatrix} \oplus 0 \times \mathbb{Z}^2, \\ 
 K_0^{\gr}(L_\K(E^6_1))^+& =\left\{ {\bf u} \in \mathbb{Z}\left[ \frac{1}{2} \right] \begin{bmatrix} 1 & 1 & 1 \end{bmatrix} \oplus 0 \times \mathbb{Z}^2 \, | \, {\bf u} \cdot {\bf z} > 0 \right\} \cup \, \{0\}, \text{ where }\bf{z} \text{ is } \left[
\begin{array}{c}
1\\
1\\
1 \\
\end{array}
\right],\\
{}^1(a,b,c) & =(a,b,c) A_{E_6^1}= (b+c, a+c, a+b).
\end{align*}

{\bf 22.} Consider the graphs $E^7_{1}$ below with the adjacency matrix $A_{E^7_1}= \begin{bmatrix} 0 & 1 & 1 \\ 1 & 1 & 0 \\ 1 & 0 & 1 \\ \end{bmatrix}$.

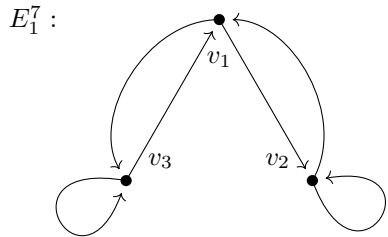
\begin{figure}[H]
\centering
\begin{tikzpicture}[scale=3.5]
\fill (0,0)  circle[radius=.6pt];
\fill (-.35,-.61)  circle[radius=.6pt];
\fill (.35,-.61)  circle[radius=.6pt];
\draw (0,-.15) node{$v_1$};
\draw (-.70, 0) node{$E^7_1:$};
\draw (-.35+.13,-.535) node{$v_3$};
\draw (.35-.13,-.535) node{$v_2$};
\draw[->, shorten >=5pt] (0,0) to (.35,-.61);
\draw[->, shorten >=5pt] (.35,-.61) to [in=0,out=60] (0,0);
\draw[->, shorten >=5pt] (-.35,-.61) to (0,0);
\draw[->, shorten >=5pt] (0,0) to [in=120,out=180
] (-.35,-.61) ;
\draw[->, shorten >=5pt] (-.35,-.61) to[in=250,out=170, loop] (-.35,-.61);
\draw[->, shorten >=5pt] (.35,-.61) to[in=10,out=-70, loop] (-.35,-.61);
\end{tikzpicture}
\end{figure}

We have that $\det(E^7_{1})=2$. In order to calculate $K_0^{\gr}(L_\K(E^7_1))$, we will use the alternative description of
dimension group from Corollary~\ref{alternativeformdel}. Since $A_{E^7_1}^{-1}=\frac{1}{2}\begin{bmatrix}
 -1 & 1 & 1 \\
 1 & 1 & -1 \\
 1 & -1 & 1 
\end{bmatrix}$ and 
$$\begin{bmatrix} 1 & 1 & 1 \end{bmatrix} \begin{bmatrix}
 -1 & 1 & 1 \\
 1 & 1 & -1 \\
 1 & -1 & 1 
\end{bmatrix} = \begin{bmatrix} 1 & 1 & 1 \end{bmatrix},$$
a similar computations as in Propositions~\ref{profidsh6} and \ref{profidsh7} can be performed. We thus obtain 

\begin{align*}
K_0^{\gr}(L_\K(E^7_1)) & = \mathbb{Z}\left[ \frac{1}{2} \right] \begin{bmatrix} 1 & 1 & 1 \end{bmatrix} \oplus 0 \times \mathbb{Z}^2, \\ 
 K_0^{\gr}(L_\K(E^7_1))^+& =\left\{ {\bf u} \in \mathbb{Z}\left[ \frac{1}{2} \right] \begin{bmatrix} 1 & 1 & 1 \end{bmatrix} \oplus 0 \times \mathbb{Z}^2 \, | \, {\bf u} \cdot {\bf z} > 0 \right\} \cup \, \{0\}, \text{ where }\bf{z} \text{ is } \left[
\begin{array}{c}
1\\
1\\
1 \\
\end{array}
\right],\\
{}^1(a,b,c) & =(a,b,c) A_{E_7^1}= (b+c, a+b, a+c).
\end{align*}

\begin{rmk}
The items (3), (21) and (22) show that the graded Grothendieck groups of Leavitt path algebras of $E^1_3$, $E^6_1$ and $E^7_1$, their positive cones, as well as Perron-Frobenius Eigenvalue and vectors are the same. However these matrices are not shift equivalent. In fact, the action of $\mathbb{Z}[x,x^{-1}]$ on $K_0^{\gr}(L_\K(E^1_3))$ $K_0^{\gr}(L_\K(E^6_1))$ and $K_0^{\gr}(L_\K(E^7_1))$ are different. Thus one cannot use (\ref{lol153332}) to conclude their $K$-theories are also the same.  As the Table~\ref{tablegray}  shows the $K$-theory (i.e., the Frank-Bowen groups) of these three graphs are different. 
\end{rmk}

{\bf 23.} Consider the graphs $E^7_{2}$ below with the adjacency matrix  $E^7_2= \begin{bmatrix} 1 & 1 & 1 \\ 1 & 1 & 0 \\ 1 & 0 & 1 \\ \end{bmatrix}$.

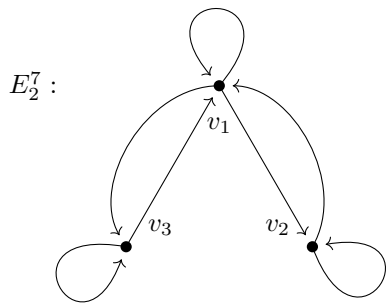
\begin{figure}[H]
\centering
\begin{tikzpicture}[scale=3.5]
\fill (0,0)  circle[radius=.6pt];
\fill (-.35,-.61)  circle[radius=.6pt];
\fill (.35,-.61)  circle[radius=.6pt];
\draw (0,-.15) node{$v_1$};
\draw (-.70, 0) node{$E^7_2:$};
\draw (-.35+.13,-.535) node{$v_3$};
\draw (.35-.13,-.535) node{$v_2$};
\draw[->, shorten >=5pt] (0,0) to (.35,-.61);
\draw[->, shorten >=5pt] (.35,-.61) to [in=0,out=60] (0,0);
\draw[->, shorten >=5pt] (-.35,-.61) to (0,0);
\draw[->, shorten >=5pt] (0,0) to [in=120,out=180
] (-.35,-.61) ;
\draw[->, shorten >=5pt] (0,0) to[in=130,out=50, loop] (0,0);
\draw[->, shorten >=5pt] (-.35,-.61) to[in=250,out=170, loop] (-.35,-.61);
\draw[->, shorten >=5pt] (.35,-.61) to[in=10,out=-70, loop] (-.35,-.61);
\end{tikzpicture}
\end{figure}

Since $\det(E^7_{2})=-1$ and the graph is  primitive,  Lemma~\ref{DL} and Corollary~\ref{halfspace} give:

\begin{align*}
K_0^{\gr}(L_\K(E^7_2)) &=   \mathbb{Z} ^3, \\
K_0^{\gr}(L_\K(E^7_2))^+ &=\left\{ {\bf u} \in \mathbb{Z}^3 \, | \, {\bf u} \cdot {\bf z} > 0 \right\}  \cup \, \{0\}, \text{ where } \bf{z}  \text{ is } \left[
\begin{array}{c}
 \sqrt{2} \\ 1 \\ 1 \\
\end{array}
\right]
\\
{}^1(a,b,c)&=(a+b+c,a+b, a+c).
\end{align*}

We are in a position to prove Theorem~\ref{alergy1}

\begin{proof}[Proof of Theorem~\ref{alergy1}]
(1) $\rightarrow$ (2) Suppose the adjacency matrices $A_E$ and $A_F$ are shift equivalent. Then the Krieger's triple for $E$ and $F$ coincide, which implies that  $K^{\gr}_0(L_\K(E)) \cong K^{\gr}_0(L_\K(F))$ as order-preserving $\mathbb Z[x,x^{-1}]$-modules (see~\S\ref{sec:symbolic-dynamics}). From (\ref{lol153332}) it follows that $K_0(L_\K(E)) \cong K_0(L_\K(F))$ and that they have the same Perron-Frobenius eigenvalue by Proposition~\ref{marcusbks}. Thus the graphs $E$ and $F$ will be in the same group in the Table~\ref{tablegray}. The analysis of the cases (1)-(25) above shows that the graphs in each group are indeed strongly shift equivalent. 

(2) $\rightarrow$ (3) By Williams' work~\cite{williams,Lind-Marcus2021}, two matrices are strong shift equivalent if and only if they can be converted to each other with a series of in-split and out-split moves. Each of these moves gives a Leavitt path algebra which is graded equivalent to the original one~\cite{AP, HazaratDyn2013}. 

(3) $\rightarrow$ (4) This implication will be done at the end of the proof.

(4) $\rightarrow$ (5)
Suppose that the graph $C^*$-algebras $C^*(E)$ and $C^*(F)$ are equivariantly Morita equivalent. Then $K_0^{\mathbb T}(C^*(E)) \cong K_0^{\mathbb T}(C^*(F))$ as order-preserving $\mathbb Z[x,x^{-1}]$-modules (see \cite[p.~297]{comb} and \cite[Prop.~2.9.1]{chrisp}). For a finite graph $E$, 
there are canonical order isomorphisms of $\mathbb Z[x,x^{-1}]$-modules
\begin{equation}\label{noalergy}
K^{\gr}_0(L(E)) \cong K_0(L(E \times \mathbb Z)) \cong  K_0(C^*(E \times \mathbb Z))\cong K_0^{\mathbb T}(C^*(E)),
\end{equation}
where $E\times \mathbb Z$ is the covering graph of $E$ (see for example~\cite[p. 275]{mathann}, \cite[Proof of Theorem A]{eilers2}). Thus (5) follows. 

(5) $\rightarrow$ (6) Since the module isomorphism between the graded $K$-groups is order-preserving, the corresponding positive cones are $\mathbb Z$-isomorphic. But the positive cones are $\mathbb Z$-isomorphic with the talented monoids (see~\S\ref{sec:symbolic-dynamics}). Thus (6) follows. 

(6) $\rightarrow$ (1) Since the talented monoid $T_E$ of a graph $E$ is isomorphic to $\mathcal V^{\gr}(L_\K(E))$,  the group completion of $T_E$ retrieves $K^{\gr}_0(L_\K(E))$. 
Thus if $T_E \cong T_F$ as $\mathbb Z$-monoid then $K^{\gr}_0(L_\K(E))\cong K^{\gr}_0(L_\K(F))$ as order-preserving $\mathbb Z[x,x^{-1}]$-modules. Therefore the Krieger's triple of $E$ and $F$ coincide, which in return implies that the adjacency matrices $A_E$ and $A_F$ are shift equivalent.

(3) $\rightarrow$ (4) Suppose $L_\K(E)$ and $L_\K(F)$ are graded Morita equivalent. Then their graded $K$-groups are module isomorphic, i.e. (5). Since we already proved (5) $\rightarrow$ (6) $\rightarrow$ (1)  $\rightarrow$ (2), then the graphs are strongly shift equivalent. From \cite[Theorem 3.2]{pask}, the in-splitting and out-splittings we used to convert from $E$ to $F$ give a Morita equivalence between $C^*(E)$ and $C^*(F)$. Furthermore, \cite[Theorem 3.4]{ER} guarantees that this equivalence is equivariant. 
\end{proof}

 \begin{proof}[Proof of Theorem~\ref{alergy2}]
 
 (1) $\rightarrow$ (2) This is immediate. 
 
 (2) $\rightarrow$ (1) This follows from analysis of the cases (1)-(24) above. 
 \end{proof}

We remark that the analytic version of  Theorem~\ref{alergy2} will follow from the work of Bratteli and Kishimoto~\cite[Crollary~4.1]{bratteli} which uses the methodology from classification of $C^*$-algebras and the Rokhlin property.

\subsection{Codes}\label{mathematicacode}

The following code, written in Mathematica,  will check whether two matrices $A$ and $B$ with entries $0$ and $1$ and of arbitrary sizes, are elementary shift equivalent. 

\medskip

\begin{verbatim}
Eshift[A_, B_] := Module[{r = Length[A], c = Length[B], xSet = {}},
  d = r * c;
  (* all r*c matrices with entries 0 and 1 *)
  t = Tuples[Range[0, 1], {r, c}];
  s = Tuples[Range[0, 1], {c, r}];
  (* Range[0,k] gives matrices with entries between 0 and k *)
  Do[
   If[
    t[[i]] . s[[j]] == A && s[[j]] . t[[i]] == B,
    AppendTo[xSet, {t[[i]], s[[j]]}]
    ],
   {i, 1, 2^d}, {j, 1, 2^d}
   ];
  Map[MatrixForm, xSet, {2}] 
  ]

\end{verbatim}
  
  With the above program we can check that, for example,  $E^2_{3}$ is elementary shift equivalent to $E^3_{4}$:
  
  \begin{verbatim}
Eshift[{{0, 1, 0}, {1, 0, 1}, {1, 1, 1}}, {{1, 1, 1}, {1, 0, 0}, {1, 0, 0}}]
   
\end{verbatim}

$$\left \{ \left \{ \left[
\begin{array}{ccc}
 0 & 0 & 1 \\
 1 & 0 & 0 \\
 1 & 1 & 0 \\
\end{array}
\right], \left[
\begin{array}{ccc}
 1 & 0 & 1 \\
 0 & 1 & 0 \\
 0 & 1 & 0 \\
\end{array}
\right]\right \},
 \left \{ \left[
\begin{array}{ccc}
 0 & 0 & 1 \\
 1 & 0 & 0 \\
 1 & 1 & 0 \\
\end{array}
\right], \left[
\begin{array}{ccc}
 1 & 0 & 1 \\
 0 & 1 & 0 \\
 0 & 1 & 0 \\
\end{array}
\right]\right \}
\right \}$$

\medskip 

Thus $E^2_{3} = RS$ and $E^3_{4}=SR$, where $R$ and $S$ could be any of the pairs above.  On the other hand, $E^2_{3}$ and $E^2_{4}$ are not elementary shift equivalent:

\begin{verbatim}
Eshift[{{0, 1, 0}, {1, 0, 1}, {1, 1, 1}}, {{1, 1, 0}, {1, 0, 1}, {1,  1, 0}}]
{} 
   \end{verbatim}

\end{document}